\documentclass[11pt,leqno]{amsart}
\usepackage{latexsym,amssymb,amsmath,multicol, rotating,lscape}

\usepackage{mathrsfs}

\usepackage{letltxmacro}
\usepackage{longtable}
\usepackage{tabularx}
\usepackage{ltablex}
\usepackage{multirow}
\usepackage{array}
\font\sm=msbm7 scaled \magstep 2
\font\germ=eufm10 scaled 1200

\newtheorem{thm}{Theorem}[section]
\newtheorem{lem}[thm]{Lemma}
\newtheorem{prop}[thm]{Proposition}
\newtheorem{cor}[thm]{Corollary}
\newcommand{\thmref}[1]{Theorem~\ref{#1}}
\newcommand{\lemref}[1]{Lemma~\ref{#1}}
\newcommand{\rmkref}[1]{Remark~\ref{#1}}
\newcommand{\propref}[1]{Proposition~\ref{#1}}
\newcommand{\corref}[1]{Corollary~\ref{#1}}
\newcommand{\secref}[1]{Section~\ref{#1}}
\newcommand{\C}{{\mathbb C}}
\theoremstyle{remark}
\newtheorem{rmk}{Remark}[section]

\DeclareMathOperator{\lcm}{lcm}

\begin{document}
	
	\title[figurate numbers and forms of mixed type]
	{Figurate numbers, forms of mixed type and their representation numbers}
	
	\author{B. Ramakrishnan and Lalit Vaishya}
	\address{(B. Ramakrishnan) Indian Statistical Institute, North-East Centre, Punioni, Solmara, Tezpur 784 501
		Assam, India}
	\address{(Lalit Vaishya) The Institute of Mathematical Sciences, HBNI, CIT Campus, Taramani, Chennai 600 113, India}
	\email[B. Ramakrishnan]{ramki@isine.ac.in, b.ramki61@gmail.com}
	\email[Lalit Vaishya]{lalitvaishya@gmail.com, lalitv@imsc.res.in}
	
	
	\date{\today}
	
	\maketitle
	
	\begin{abstract} 
		In this article, we consider the problem of determining  formulas for the number of representations of a natural number $n$  by a sum of figurate numbers with certain positive integer coefficients. To achieve this, we prove that the associated generating function gives rise to a modular form of integral weight under certain condition on the coefficients when even number of higher figurate numbers are considered. In particular, we obtain modular property of the generating function corresponding to a sum of even number of triangular numbers with coefficients. We also obtain modularity property of the generating function of mixed forms involving figurate numbers (including the squares and triangular numbers) with coefficients and forms of the type $m^2+mn+n^2$ with coefficients. In particular, we show the modularity of the generating function of odd number of squares and odd number of triangular numbers (with coefficients). As a consequence, explicit formulas for the number of representations of these mixed forms are obtained using a basis of the corresponding space of modular forms of integral weight. We also obtain several applications concerning the triangular numbers with coefficients similar to the ones obtained in \cite{ono}. In \cite{xia}, Xia-Ma-Tian considered some special cases of mixed forms and obtained the number of representations of these 21 mixed forms using the $(p,k)$ parametrisation method. We also derive these 21 formulas using our method and further obtain as a consequence, the $(p,k)$ parametrisation of the Eisenstein series $E_4(\tau)$ and its duplications. It is to be noted that the $(p,k)$ parametrisation of $E_4$ and its duplications were derived by a different method in \cite{{aw},{aaw}}.
		We illustrate our method with several examples. 
	\end{abstract}

	\section{Introduction and Statement of results}
	Finding formulas for the number of representations of a positive integer by certain class of quadratic forms is one of the classical problems in number theory, especially obtaining formulas for the number of representations of a natural number $n$ by a sum of $k$ integer squares, denoted by $r_k(n)$. In this connection, triangular numbers played an important role in finding formulas for $r_k(n)$, $8\vert k$ (see \cite{kohnen-imamoglu}). For a fixed $a \ge 1$, we define the $n^{\rm th}$-figurate number  by  $f_a(n):= \frac{a n^{2} + (a-2)n}{2},$ following Ono-Robins-Wahl \cite{ono}. Note that  the function $f_a(n)$ denotes the $n^{\rm th}$-triangular number (resp.  square number and pentagonal number)  when $a=1$ (resp.  $2$ and $3$). In the literature,  the $n^{\rm th}$ triangular number is denoted by $T_{n}$, and  for a fixed $a \ge 3$, the integer $f_{a}(n)$ is known as  the $n^{\rm th}$ higher figurate number. These functions are associated with counting the number of vertices of  some geometric objects.  We denote the generating functions for these figurate numbers as $\Psi(\tau)$ (for triangular numbers), $\theta(\tau)$ (for square numbers) and $\Phi_{a}(\tau)$ (for $a\ge 3$), and they are given respectively by 
	\begin{equation}\label{fig-gen-fun}
		\begin{split}
			\Psi(\tau) := \sum_{n= 0}^{\infty} q^{\frac{n(n+1)}{2}}, \qquad \theta(\tau)= \sum_{n= 0}^{\infty} q^{n^{2}}, \qquad {\rm and } \qquad {\rm for } ~a \ge 3, \quad \Phi_{a}(\tau) = \sum_{n \in {\mathbb Z}} q^{f_a(n)}.
		\end{split}
	\end{equation}
	where $q = e^{2 i \pi \tau}, \ \ \tau \in {\mathbb H}.$  In 1995,  Ono-Robins-Wahl \cite{ono} considered the problem of determining formulas for the number of representations of a natural number $n$ by $k$ triangular numbers, denoted as $\delta_k(n)$. When $2\vert k$, they used the theory of modular forms of integral weight to get formulas for $\delta_k(n)$. In their work, they have given several applications of the function $\delta_k(n)$.  One of them is the formula for $\delta_{24}(n)$,  involving the Ramanujan tau function $\tau(n)$.  For general $k$, they related $\delta_k(n)$ with the number of representations of $n$ as a sum of $k$ odd square integers. Further, they also obtained applications to  the number of Leech lattice points inside  $k$-dimensional sphere.  In the case of  higher figurate numbers, they proved that the generating function for the higher figurate numbers is given in terms of generalised eta quotients using the Jacobi triple product identity.  In the past few years, there are many works in the literature, which consider the problem of finding formulas for the number of  representations of a natural number by quadratic forms with certain coefficients. We give here a few references to this effect \cite{{ramanujan}, {aalw}, {aalw1}, {w}, {apw}, {xia}, {rss-ijnt}}.

	In a similar vein, in this article, we consider  linear combinations with positive coefficients of these figurate numbers  with $a\neq 2$. More precisely,  we consider the following linear combinations associated to higher figurate numbers with coefficients $c_{i}$ given by ($a \ge 3$);
	\begin{equation}\label{figurate}
		F_{a, {\mathcal C}}({\bf x}) =   \sum_{i=1}^k c_i f_{a}(x_i), \qquad {\rm where} \qquad f_a(n) = \frac{a n^{2} + (a-2)n}{2}, % ~ {\bf x} := (x_{1},x_{2}, \ldots , x_{k})\in {\mathbb Z}^{k}, ~ k ~{\rm is ~even }, 
	\end{equation}
	and for  triangular numbers given by;
	\begin{equation}\label{triangular}
		T_\mathcal{C}({\bf x}) =   \sum_{i=1}^k c_i \frac {x_{1}(x_{i}+1)}{2}, ~\qquad  {\rm where}  ~\qquad  {\bf x} := (x_{1},x_{2}, \ldots , x_{k})\in {\mathbb Z}^{k},  
	\end{equation}
	and ${\mathcal C} =  (c_1,\cdots, c_k),$ $c_i\le c_{i+1}$.  
	We first prove the modularity property of the generating functions corresponding to the  higher figurate numbers ($a \ge 3$), and the  triangular numbers ($a=1$) under certain conditions. Moreover, as in the work of Ono-Robins-Wahl, we also derive some applications related to the triangular numbers with coefficients.

	Next, we consider, following the work of Xia-Ma-Tian \cite{xia}, a linear combination of  three types of mixed forms and find their corresponding representation formulas. More specifically, in \cite{xia}, they considered the following type of forms in $2m$ variables, which is given as follows.  
	\begin{equation}\label{mixed}
		{\mathcal M}({\bf x}, {\bf y}, {\bf z}) :=  \sum_{i=1}^u a_i (x_{2i-1}^2 + x_{2i-1}x_{2i}+ x_{2i}^2) + \sum_{i=1}^v b_i y_i^2  +
		\sum_{i=1}^k c_i T_{z_i},
	\end{equation}
	where ${\bf x} := (x_{1},x_{2}, \ldots , x_{2u})\in {\mathbb Z}^{2u}$,~ ${\bf y} := (y_{1}, y_{2}, \ldots , y_{v}) \in {\mathbb Z}^v$, ~
	${\bf z} := (z_{1}, z_{2}, \ldots , z_{k})\in {\mathbb N}_0^k$, ~\linebreak ${\mathbb N}_0 = {\mathbb N} \cup \{0\}$. Further,  $2m = 2u + v+ k$; $u,v,k$ are non-negative integers  such that $v+ k$ is even and the coefficients $a_i\in \{1,2,4,8\}$, $b_i\in \{1,2,3,6\}$ and $c_i\in \{1,2,3,4,6\}$. The above is a mixed form consisting of quadratic forms of type $m^2+mn+n^2$,  squares and triangular numbers with coefficients. Though, they defined a more general type of mixed forms ${\mathcal M}({\bf x}, {\bf y}, {\bf z})$, with coefficients $a_i$, $b_i$ and $c_i$ as above, in their work, they actually dealt with 21 mixed forms (and obtained the corresponding representation formulas) with $m=4$ ($ 8$ variable forms).  
	
	In the second part of our work, we consider forms of type \eqref{mixed} (all possible mixed forms) and show that the corresponding generating function is a modular form of weight $m/2$ with precise level and character (depending on the coefficients $a_{i}, b_i$ and $c_{i}$). As examples, we use the theory of modular forms to get the required formulas corresponding to the mixed forms in $4, 6, 8$ and $10$ variables. We obtain these formulas explicitly for some special cases by using suitable basis for the space of modular forms. Moreover, we also obtain the representation formulas for the  21 mixed forms considered  in the work of  Xia-Ma-Tian \cite{xia}, by explicitly constructing a basis for the space of modular forms of weight $4$ on  $\Gamma_{0}(12)$. 
	
	\smallskip
	In \cite{xia}, they  obtained the representation formulas for the 21 mixed forms using the $(p,k)$- parametrisations of the normalised Eisenstein series $E_4(\tau)$ of weight $4$, the Dedekind eta function $\eta(\tau)$ (and their duplications) and by using certain theta function identities.
	In the third part of this article,  we observe that our formulas for the 21 mixed forms (which are obtained by using the theory of modular forms) together with the $(p,k)$-parametrisations of the three generating functions of the  forms appearing in 
	${\mathcal M}({\bf x}, {\bf y}, {\bf z})$ give rise to the $(p,k)$-parametrisations of $E_4(\tau)$ (and its duplications). A detailed presentation of the above is presented  in \S 3. It is to be noted that the
	$(p,k)$ parametrisations of $E_4(d\tau)$, $d\vert 12$, were obtained by a different method in \cite{{aw}, {aaw}}.   We also remark that in \cite{lam}, a few cases of triangular forms with coefficients and mixed forms consisting of squares and triangular numbers with some coefficients have been studied and formulas were obtained using different methods. 
	
	\bigskip
	\noindent 
	Before we state our main results, we fix some notations. \\
	Let $\eta(\tau) := q^{\frac{1}{24}} \displaystyle{\prod_{n=i}^{\infty}} (1-q^{n})$ denote the Dedekind eta-function.  Following the notion of \cite{ono}, let $\eta_{\delta,g}(\tau)$ denote the  generalised  Dedekind eta-function given by 
	\begin{equation}\label{gen-Dedekind-eta}
		\displaystyle \eta_{\delta,g}(\tau) = q^{\frac{P_2(g/\delta)\delta}{2}} \prod_{n \ge 0 \atop {n \equiv g \!\! \pmod {\delta}}}(1-q^{n})\prod_{n \ge 0 \atop {n \equiv -g \!\! \pmod {\delta}}}(1-q^{n}), 
	\end{equation}
	where $P_2(x)$ is the generalised Bernoulli polynomial given by; $P_2(x) = \{x\}^2-\{x\}+\frac{1}{6}$ and $\{x\}$ denotes the fractional part of $x$. 
	
	\begin{rmk}
		For $\delta = 1$ and $g=0,$ we have
		\begin{equation}\label{eta-relation}
			\eta_{1,0}(\tau) = \eta^{2}(z).
		\end{equation}
		A simple calculation shows that for any $ c \ge 1$, we have 
		\begin{equation}\label{duplication-rel}
			\eta_{\delta,g}(c \tau) = \eta_{c\delta,cg}( \tau).
		\end{equation} 
	\end{rmk}
	
	Let $a \ge 3$ be a fixed integer.
	Let $\delta_{k}(n)$ and  $\mathcal{N}_{a}(k;n)$ (for a fixed $a \ge 3$) denote the number of representations of $n$ as a sum of $k$-triangular numbers and as a sum of $k$-higher figurate numbers (for a fixed $a \ge 3$), respectively. 
	Let ${\mathcal C}$ denote the $k$-tuple $(c_1,\cdots, c_k),$ $c_i\le c_{i+1}$. We denote by $\delta_k({\mathcal C};n)$ and  $\mathcal{N}_{a}^{\mathcal C}(k;n)$(for a fixed $a \ge 3$), the number of representations of $n$ as a sum of $k$-triangular numbers and as a sum of $k$- higher figurate numbers (for a fixed $a \ge 3$)  with coefficients $c_i,$ respectively, i.e.,
	\begin{equation*}
		\begin{split}
			\delta_k({\mathcal C};n) &= \#\left\{(x_1,\cdots, x_k) \in {\mathbb Z}^{k} | \ n = \sum_{i=1}^k c_{i}\frac{x_{i}(x_{i}+1)}{2} \right\} \\ 
			{\rm and ~ for ~ a~ fixed ~ integer~ a \ge 3, } \quad \mathcal{N}_{a}^{\mathcal C}(k;n) &= \#\left\{(x_1,\cdots, x_k) \in {\mathbb Z}^{k} | \ n = \sum_{i=1}^k c_{i} f_{a}(x_i) \right\}. \qquad  \qquad  \qquad  \qquad  \qquad 
		\end{split}
	\end{equation*}
	In particular, if all $c_{i}=1$ then $\delta_k({\mathcal C};n) = \delta_{k}(n) $ and $\mathcal{N}_{a}^{\mathcal C}(k;n) = \mathcal{N}_{a}(k;n)$. We denote the generating function for $\delta_k({\mathcal C};n)$ and  $\mathcal{N}_{a}^{\mathcal C}(k;n)$  by $\Psi_{\mathcal C}(\tau)$ and $\Phi_a^{\mathcal C}(\tau)$, respectively. These are given by:
	\begin{equation}\label{gen-coeff}
		\begin{split}
			\Psi_{\mathcal C}(\tau) &= \prod_{i=1}^k \Psi(c_i \tau) = \sum_{n=0}^\infty \delta_k({\mathcal C};n) q^n \\
			{\rm and} \quad \Phi_a^{\mathcal C}(\tau)  &= \prod_{i=1}^k \Phi_{a}(c_{i} \tau)= \sum_{n=0}^\infty \mathcal{N}_{a}^{\mathcal C}(k;n) q^n.
			\\
		\end{split}
	\end{equation}
	It is known  from \cite[Proposition 1, Theorem 10]{ono}  that the generating  functions for the figurate numbers defined in \eqref{fig-gen-fun} are explicitly obtained in terms of eta quotients and generalised eta quotients, i.e., 
	\begin{equation}\label{Gen-function}
		\begin{split}
			\Psi(\tau) = q^{-1/8} \frac{\eta^{2}(2\tau)}{\eta(\tau)}, \quad \theta(\tau)=  \frac{\eta^5(2 \tau)}{\eta^2( \tau)\eta^2(4 \tau)}, ~ {\rm and ~ for ~ a \ge 3,} ~ \Phi_{a}(\tau) = q^{-\frac{(a-2)^2}{8a}} \eta(a\tau) \frac{\eta_{a,1}(2\tau)}{\eta_{a,1}(\tau)},
		\end{split}
	\end{equation}
	where $\eta(\tau)$ is the Dedekind eta function and $\eta_{\delta,g}(\tau)$ is the generalised Dedekind eta function defined by  \eqref{gen-Dedekind-eta}.  So, the generating functions $\Psi_{\mathcal C}(\tau)$ and $\Phi_a^{\mathcal C}(\tau)$ are given in terms of Dedekind eta quotient and generalised Dedekind eta quotient, respectively, i.e.,
	\begin{equation}\label{triangular-gen-fun}
		\begin{split} 
			\Psi_{\mathcal C}(\tau)  & = \prod_{i=1}^k \Psi(c_i \tau) ~
			= q^{-{\bf h}/8} \prod_{i=1}^k \frac{\eta^2(2c_i \tau)}{\eta(c_i \tau)} \\
		\end{split}
	\end{equation}
	and for $a\ge 3$,
	\begin{equation}\label{figurate-gen-fun}
		\begin{split} 
			\Phi_a^{\mathcal C}(\tau) & = \prod_{i=1}^k \Phi_{a}(c_{i} \tau) = q^{-\frac{(a-2)^2}{8a}{\bf h}} \prod_{i=1}^k \eta(c_i a\tau) \prod_{i=1}^k \frac{\eta_{a,1}(2c_i\tau)}{\eta_{a,1}(c_i \tau)}
		\end{split}
	\end{equation}
	where ${\bf h}= c_1 + c_2 + \cdots + c_k$. 
	Let ${\bf c}:= c_1 c_2  \cdots c_k$ . Our first result states that for a fixed positive integer $a \not = 2$,~ $\Phi_a^{\mathcal C}(\tau)$ is a modular form under certain conditions. More precisely, we have the following theorem.
	
	\begin{thm}\label{modularity}
		Let $a \not=2 $ be a fixed positive integer and $k \in 2{\mathbb Z}^{+}$. 
		Let $\Phi_a^{\mathcal C}(\tau)$ be the generating function for $\mathcal{N}_{a}^{\mathcal C}(k;n)$ (when $a\ge 3$) and 
		$\Psi_{\mathcal C}(\tau)$ be the generating function for $\delta_k({\mathcal C};n)$ (when $a=1$). Let $F(\tau)$ be one of these 
		generating functions (depending on $a\ge 3$ or $a=1$). 
		Then, $q^{\frac{(a-2)^2}{8a} {\bf h}} F(\tau)$ is a 
		%% generalised Dedekind eta quotient and it  is a 
		modular form of weight $\frac{k}{2}$ on the congruence subgroup $\Gamma_0(N)$ (given explicitly in terms of $a$ and $c_i$) if, and only if, 
		${\bf h} \equiv 0 \pmod{\frac{8a}{d}}$ where $d = \gcd((a-2)^2,4a)$, $N = 2a \lcm[c_1, \ldots,c_k]$ when $a$ is odd and $N = a \lcm[c_1, \ldots,c_k]$ when $a$ is even, and the character $\chi$  is given by $\chi = \left(\frac{(-1)^{k/2} 4{\bf c}}{\cdot}\right)$.
	\end{thm}
	
	\smallskip
	
	\begin{rmk}\label{rmk:1}
		When $a=1$, the above theorem gives modular property of the generating function $\Psi_{\mathcal C}(\tau)$. More precisely, it states that  $q^{{\bf h}/8} \Psi_{\mathcal C}(\tau)$ is a modular form of weight $k/2$ on $\Gamma_0(N)$ with character 
		$\chi_{1} = \left(\frac{(-1)^{k/2} 4{\bf c}}{\cdot}\right)$ if, and only if,  ${\bf h} \equiv 0 \pmod{8}$, where $N = 2\cdot \lcm[c_1, c_2, \dots, c_k]$ (with ${\bf h}$ and ${\bf c}$ as defined above). 
	\end{rmk}
	
	\smallskip
	
	The mixed form given by \eqref{mixed} has three components which are linear combinations of quadratic forms of type  $m^{2} + mn + n^{2}$, squares and triangular numbers. For the second part of the paper, we take two sub classes of the above mixed forms which are combinations of triangular numbers with squares or forms of the type $m^{2} + mn + n^{2}$. We give a notation for these combinations in the following. 
	\begin{equation}\label{st-lt}
		\begin{split}
			{\mathcal M }_{l, t}({\bf x}, {\bf z}) & = \sum_{i=1}^u a_i (x_{2i-1}^2 + x_{2i-1}x_{2i}+ x_{2i}^2)  + \sum_{i=1}^k c_i T_{z_i}, 
			\quad k {\rm ~is~even~},\\
			{\mathcal M }_{s, t}({\bf y}, {\bf z}) & =  \sum_{i=1}^v b_i y_i^2  + \sum_{i=1}^k c_i T_{z_i}, \quad v+k {\rm ~is~even~}, 
		\end{split}
	\end{equation}
	with ${\bf x}, {\bf y}, {\bf z}, a_i, b_i, c_i$ as in \eqref{mixed}. The generating function for the mixed form ${\mathcal M}_{l,t}({\bf x}, {\bf z})$ is given by 
	$$
	\psi_{l,t}(\tau) = \prod_{i=1}^u {\mathcal F}(a_i \tau) \prod_{j=1}^{k} \Psi(c_{j}\tau),
	$$
	where  ${\mathcal F}(\tau) =\displaystyle{ \sum_{m,n\in {\mathbb Z}}} q^{m^2+mn+n^2}$. As $k$ is even, by 
	%\thmref{modularity} (with $a=1$), 
	\rmkref{rmk:1}, we know that $q^{{\bf h}/8} \prod_{j=1}^{k} \Psi(c_{j}\tau) = q^{{\bf h}/8} \Psi_{\mathcal C}(\tau)$ is a modular form when ${\bf h} = \sum_ic_i$ is a multiple of $8$. Further, ${\mathcal F}(\tau)$ is a modular form in $M_1(\Gamma_0(3), \big(\frac{\cdot}{3}\big))$, which follows from \cite[Theorem 4]{schoen}. Combining these two modular properties, we have the following theorem, which gives the modular property of $\psi_{l,t}(\tau)$.
	
	\begin{thm}\label{lt}
		For an even integer $k\ge 2$, let $\psi_{l,t}(\tau)$ be the generating function for the mixed form ${\mathcal M }_{l, t}({\bf x}, {\bf z})$ 
		in $2u+k$ variables defined in \eqref{st-lt}.  Then $q^{{\bf h}/8} \psi_{l,t}(\tau)$ is a 
		modular form in $M_{u+k/2}(\Gamma_0(L), \chi')$, where $L = \lcm[3~ \lcm
		[a_1, \ldots, a_u], 2~\lcm[c_1, \ldots, c_k]]$ and 
		$\chi' = \chi$, if $u$ is even and $\chi' = \big(\frac{\cdot}{3}\big) \chi$, if $u$ is odd, provided ${\bf h}$ is a multiple of $8$.
		Here, $\chi$ is the character as defined in \thmref{modularity}.
	\end{thm}
	
	Next, we consider the mixed form ${\mathcal M}_{s,t}({\bf y}, {\bf z})$ in $v+k$ variables with $2\vert (v+k)$, and denote its 
	generating function by $\psi_{s,t}(\tau)$. Then we have 
	\begin{equation}\label{gen;st}
		\psi_{s,t}(\tau) =  \prod_{i=1}^{v} \theta(b_{i}\tau) \prod_{j=1}^{k} \Psi(c_{j}\tau),
	\end{equation}
	where $\theta(\tau)$ is the classical theta function given by $\theta(\tau) = \sum_{n\in {\mathbb Z}} q^{n^2}$. In the next result, we shall obtain the modularity of the function $\psi_{s,t}(\tau)$. To derive this, we use the fact that both the generating functions that appear on the right hand-side of \eqref{gen;st} are expressed as eta-quotients. 
	
	\begin{thm}\label{odd-triangular}
		Let $v$ and $k$ be positive integers such that $v+k$ is even and ${\bf h}$ be the sum of the coefficients $c_i$. 
		Set $M = \lcm[4~\lcm[b_1, \ldots, b_v], 2~\lcm[c_1, \ldots, c_k]]$. 
		Then the function $q^{{\bf h}/8} \psi_{s,t}(\tau)$ is a modular form in $M_{(v+k)/2}(\Gamma_0(M), \chi'')$, 
		if, and only if, ${\bf h }\equiv 0 \pmod {8}$,
		where  
		$$
		\chi'' = \begin{cases}
			\left(\frac{(-1)^{(v+k)/2} 4 \prod_{i=1}^{v} b_{i}\prod_{j=1}^k c_{j}} {\cdot}\right),& {\rm ~if~} v {\rm ~and~} k {\rm ~are~even}, \\   
			\left(\frac{(-1)^{(v+k)/2} 8 \prod_{i=1}^{v} b_{i}\prod_{j=1}^k c_{j}} {\cdot}\right),& {\rm ~if~} v {\rm ~and~} k {\rm ~are~odd}. \\   
		\end{cases}
		$$  
	\end{thm}
	
	\smallskip
	
	\begin{rmk}
		It is to be noted that the generating functions of  odd number of squares and  odd number of triangular numbers will have odd weight. 
		Therefore, the product of these generating functions will be modular of integral weight (explicit description of weight, level, character are given by the above theorem).  Using this fact, we obtain new relation between the representations of 
		an integer as a sum of an odd number of triangular numbers and an odd number of squares.  Details are presented in  \corref{cor:1.5}, \rmkref{rmk:2}.
	\end{rmk}
	
	\smallskip
	
	Note that the generating function $\Phi(\tau)$ of the mixed form ${\mathcal M}({\bf x}, {\bf y}, {\bf z})$ 
	in $2u + v+k$ variables (with $2\vert (v+k)$) is expressed as 
	\begin{equation}\label{mixed-eg}
		\Phi(\tau) = \prod_{i=1}^u {\mathcal F}(a_i \tau) \psi_{s,t}(\tau) = \prod_{i=1}^u {\mathcal F}(a_i \tau) 
		\prod_{i=1}^{v} \theta(b_{i}\tau) \prod_{j=1}^{k} \Psi(c_{j}\tau).
	\end{equation}
	Therefore, the modular property of ${\mathcal F}$ together with \thmref{odd-triangular} give the following result.

	\begin{cor}\label{mixed-m} 
		Let $u,v,k$ be positive integers with $v+k$ even. Then the generating function $\Phi(\tau)$ associated to the mixed form 
		${\mathcal M}({\bf x}, {\bf y}, {\bf z})$ in $2u+v+k$ variables is a modular form upto a 
		rational power of $q$. More precisely, when ${\bf h}$ is divisible by $8$, the function $q^{{\bf h}/8} \Phi(\tau)$ is a modular form in $M_{u+\frac{v+k}{2}}(\Gamma_0(N_1), \omega)$, where $N_1 = \lcm[3\lcm[a_1, \ldots, a_u], M]$ and $\omega =  \chi''$, if $u$ is even and $\omega = \big(\frac{\cdot}{3}\big) \chi''$, 
		if $u$ is odd. Here $M$ and $\chi''$ are as in \thmref{odd-triangular}.
	\end{cor}
	
	\smallskip
	
	Let $r_{v}({\mathcal B};n)$ denote the number of representations  of  a positive integer  $n$ by the sum of $v$ squares with coefficients $b_i$, i.e., $\sum_{i=1}^v b_i y_i^2$, where ${\mathcal B}$ denotes the  $v$-tuple $(b_1, \ldots, b_v)$. 
	When all $b_i=1$, then $r_v({\mathcal B};n)$ is nothing but $r_v(n)$, the number of representations of 
	$n$ as a sum of $v$ squares. As mentioned before, $\delta_{k}({\mathcal C};n)$ denotes the number of representations of  a positive integer $n$ by 
	the linear combination of $k$ triangular numbers $\sum_{i=1}^k c_i T_{z_i}$ and when all $c_i$'s are equal to $1$, then $\delta_k({\mathcal C};n) = \delta_k(n)$.
	To simplify the notation, we will be writing the $v$-tuples and $k$-tuples  in the explicit examples as follows: $(b_1^{\alpha_1}, \ldots, b_ v^{\alpha_v})$, $(c_1^{e_1}, \ldots, c_k^{e_k}), \alpha_i, e_i \ge 0$, $\alpha_1+\cdots + \alpha_v =v$ and $e_1 + \cdots + e_k =k$. If some of the exponents are zero, then the corresponding coefficient is taken to be zero. For a natural number $n$,  $\delta_{k}( c_1^{e_1}, \ldots, c_k^{e_k}; n)$ (resp. ${\mathcal N}_{s,t}(b_1^{\alpha_1}, \ldots, b_ v^{\alpha_v}; c_1^{e_1}, \ldots, c_k^{e_k}; n)$)
	denotes the number of representations of $n$ by the the linear combination of $k$ triangular numbers $\sum_{i=1}^k c_i T_{z_i}$ (resp. the mixed form ${\mathcal M}_{s,t}({\bf y}, {\bf z})$).  
	When all the $b_i$'s and $c_i$'s are equal to $1$,  it is denoted by $\delta_{k}(n)$ and  ${\mathcal N}_{s,t}(n)$ respectively.  
	
	\smallskip
	\noindent
	The following corollary (to \thmref{odd-triangular}) gives relations between the above representation formulas.  
	\begin{cor}\label{cor:1.5}
		Let $v, k\ge 1$ be integers such that $v+k$ is even.  Assume that ${\bf h} = e_1c_1+ e_2c_2+\cdots + e_kc_k \equiv 0\pmod{8}$ 
		such that $p = {\bf h}/8 \ge 0$ is an integer. 
		Then,  we have 
		\begin{equation}\label{cor:1.6}
			\begin{split}
				\delta_{k}({\mathcal C} ;n) = {\mathcal N}_{s,t}(b_1^{\alpha_1}, \ldots, b_ v^{\alpha_v}; c_1^{e_1}, \ldots, c_k^{e_k}; n-p) 
				- r_{v}({\mathcal B};n) -\sum_{m=1}^{n-1} r_{v}({\mathcal B};m)\delta_{k}({\mathcal C};n-m),
			\end{split}
		\end{equation}
		where $\alpha_i, e_i$ are non-negative integers. 
		In particular,  if all $b_i$'s and $c_i$'s are equal to 1, then under the condition that $8\vert k$, we get 
		\begin{equation}
			\begin{split}
				\delta_{k}(n) = {\mathcal N }_{s,t}\left(n-\frac{k}{8}\right) - r_{v}(n) -\sum_{m=1}^{n-1} r_{v}(m)\delta_{k}(n-m).
			\end{split}
		\end{equation}
		Note that in the last identity, $p={\bf h}/8 = k/8$.
	\end{cor}
	
	\begin{rmk}\label{rmk:2}
		%In the simplest case when all the coefficients $b_i$ and $c_i$ are equal to 1, 
		The above corollary gives a recursive formula for $\delta_k(n)$ in terms  of ${\mathcal N}_{s,t}(n)$ and $r_{v}(n)$, when $8 \vert k$. When $v$ and $k$ are odd positive integers, equation \eqref{cor:1.6} gives a new relation between $\delta_k({\mathcal C};n)$ and 
		$r_v({\mathcal B};n)$, in particular, between $\delta_k(n)$ and $r_v(n)$.  
		%%It gives new information when both $v$ and $k$ are odd positive integers
	\end{rmk}
	
	\smallskip
	
	When $v=k$ is a positive integer, we get a relation between the number of representations of a natural number $n$ by the linear combination (with coefficients $c_i$) of $k$ triangular numbers and the number of representations of $n$ as a linear combination (with coefficients 
	$c_i$) of $k$ odd integer squares. By taking $v=k$, we let
	\begin{equation}\label{qkn}
		q_{k}({\mathcal C};n) := \# \{(y_{1}, y_{2}, \ldots, y_{k}) \in {\mathbb Z}^k \vert   n = \sum_{i=1}^k c_i y_i^2, y_{i} \ge 0, ~~~ 2\not\vert y_{i}\}.  
	\end{equation}
	Note that $q_k({\mathcal C};n) = q_k(n)$, whenever all the $c_i$'s are equal to 1. 
	Thus, we have the following result for $\delta_{k}({\mathcal C};n)$, which is an analogue of the corresponding result obtained  in \cite{ono} with all $c_i$'s being equal to 1. 
	
	\begin{prop}\label{odd-square}
		For a positive integer $n$, we have 
		\begin{equation}\label{11}
			\delta_{k}({\mathcal C};n)  = q_{k}({\mathcal C}; 8n+{\bf h}),
		\end{equation}
		where ${\bf h} = c_1 + \cdots + c_k$. In other words, the number of representations of $n$ by $T_{\mathcal C}({\bf z})$ is the same as the number of representations 
		of $8n + {\bf h}$ as a sum of $k$ odd integer squares with coefficients $c_i$, $1\le i\le k$. 
		In particular, when all $c_{i} =1$, we have $ {\bf h} = k$. So, the above gives
		\begin{equation}\label{delta-odd square}
			\delta_{k}(n)  = q_{k}(8n+k),
		\end{equation}
		which is the same as in \cite[Proposition 2]{ono}.
		%of \cite{ono}.
	\end{prop}
	\noindent
	%\textbf{Proof of \propref{odd-square}:}
	\begin{proof}
		Let $n\ge 1$ be an integer represented by $T_{\mathcal C}({\bf z})$. Then we have 
		\begin{equation*}
			%\begin{split}
			n  = \sum_{i=1}^{k} c_{i} \frac{z_{i}(z_{i}+1)}{2} ~ \iff ~ 8n = 4\sum_{i=1}^k c_i z_i(z_i+1).
		\end{equation*} 
		%and so $8n = 4\sum_{i=1}^k c_i z_i(z_i+1)$. 
		Since ${\bf h} = c_1+\ldots + c_k$, we get  
		\begin{eqnarray*}
			8n + {\bf h} & =  \sum_{i=1}^{k} c_{i} (4z_i^2 + 4z_i+1) & = \sum_{i=1}^k c_i (2z_i+1)^2.
		\end{eqnarray*}
		Since one can trace back these identities, it follows that $\delta_{k}({\mathcal C};n)  = q_{k}({\mathcal C}; 8n+{\bf h})$.
	\end{proof}
	
	\smallskip
	We  now give relations between $r_{k}(n)$, $\delta_{k}(n)$, $\delta_{2k}(n)$ by using the following identity (for a proof of this identity, we refer to \cite[pp. 40]{berndt}):  %It is well-known that
	\begin{equation}\label{psi-theta}
		\psi^2(\tau) = \theta(\tau) \psi(2\tau).
	\end{equation}
	Taking $k^{\rm th}$ power of the above identity and using \eqref{delta-odd square}, we get the following corollary.
	
	\begin{cor}\label{relations}
		For a natural number $n$, we have 
		\begin{equation}
			%\begin{split}
			q_{2k}(8n+2k) ~=~ \delta_{2k}(n) ~=~ 
			%\sum_{0\le m \le n} r_{k}(m)\delta_{k}(\frac{n-m}{2}) ~= 
			\sum_{a,b\in {\mathbb N}_{0}\atop{a+ 2b =n}} r_{k}(a)\delta_{k}(b) ~= ~\sum_{a,b\in {\mathbb N}_{0} \atop{a+ 2b =n}} r_{k}(a) q_{k}(8b+k). 
			%&= \sum_{0\le m \le n} r_{k}(n-2m)\delta_{k}(m), \\
			%2\delta_{k}(n) &= \sum_{0\le m \le n} r_{k}(m)\delta_{k}(\frac{n-m}{2}) - \sum_{1\le m \le n-1} \delta_{k}(m)\delta_{k}(n-m).
			%\end{split}
		\end{equation} 
		%In the above we used the convention that $\delta_k(x) =0$, is $x$ is not a non-negative integer.
	\end{cor}
	
	\begin{rmk}
		If we consider \corref{relations} for forms with coefficients, then the identity \eqref{psi-theta} for the case of forms with coefficients  ${\mathcal C} = 
		(c_1, \ldots, c_k)$ takes the following form: 
		\begin{equation}\label{rmk:general}
			\Psi_{\mathcal C}^2(\tau) = \prod_{i=1}^k\theta(c_i\tau) \Psi(2c_i\tau)=\prod_{i=1}^k\theta(c_i\tau) \Psi_{\mathcal C}(2\tau).
		\end{equation}
		The LHS can be viewed as the generating function of $2k$ triangular numbers with coefficients $(c_1,c_1, c_2,c_2\ldots, c_k, c_k) =: {\mathcal C}^2$ and the corresponding $n^{\rm th}$ Fourier coefficient is denoted as $\delta_{2k}({\mathcal C}^2;n)$. The number of representations of $n$ as a sum of $k$ squares with coefficients $(c_1,\ldots, c_k)$ is written in our notation as $r_k({\mathcal C};n)$. Therefore, comparing the $n^{\rm th}$ 
		Fourier coefficients of \eqref{rmk:general}, we have the following identity: 
		\begin{equation}\label{16}
			\delta_{2k}({\mathcal C}^2;n) = \sum_{a,b\in {\mathbb N}_{0}\atop{a+ 2b =n}} r_{k}({\mathcal C}; a)\delta_{k}({\mathcal C}; b).
		\end{equation}
		When we consider $2k$ triangular numbers with coefficients $(c_1,\ldots,c_k,c_1,\ldots,c_k)$, the corresponding result for \propref{odd-square} 
		is the following: 
		\begin{equation}\label{17}
			\delta_{2k}({\mathcal C}^2; n)  = q_{2k}({\mathcal C}^2; 8n+{2 \bf h}),
		\end{equation}
		where ${\bf h} = c_1+ \ldots + c_k$. Using \eqref{17} and \eqref{11} in \eqref{16}, the analogous result for \corref{relations} in the case of forms with coefficients is 
		given by; 
		\begin{equation}
			q_{2k}({\mathcal C}^2; 8n+{2 \bf h}) = \delta_{2k}({\mathcal C}^2; n) = \sum_{a,b\in {\mathbb N}_{0}\atop{a+ 2b =n}} r_{k}({\mathcal C}; a)\delta_{k}({\mathcal C}; b) = \sum_{a,b\in {\mathbb N}_{0}\atop{a+ 2b =n}} r_{k}({\mathcal C}; a) q_{k}({\mathcal C}; 8b + {\bf h}).
		\end{equation}
	\end{rmk}
	\noindent
	These formulas give interesting relations between the convolution sums involving the arithmetical functions $r_{k}({\mathcal C}; n)$, $\delta_{k}({\mathcal C}; n)$ and $q_{k}({\mathcal C}; n)$.
	% and  $\delta_{k}({\mathcal C}^{2}; n)$.
	
	\bigskip

	We illustrate \corref{relations} with some examples. When $k=4,6,8,12$, and $16$, we use the formula for $\delta_k(n)$ given by Ono-Robins-Wahl 
	\cite[pp. 77--81]{ono} to get the following relations among $\delta_k(n)$, $r_k(n)$. 
	%We use the formula for $\delta_{k}(n)$ for some values of $k$ in \cite[p.5,7,8]{ono} to get following identities.
	
	\begin{cor}\label{relations1}
		For a natural number $n$, we have
		\begin{equation*}
			\begin{split}
				\delta_{4}(n)  = & \frac{1}{4} \sum_{a,b \in {\mathbb N}_{0}\atop{a+2b = n}} r_2(a) r_2(8b+2),\\
				\delta_{6}(n)  = & \frac{1}{8} \sum_{a,b \in {\mathbb N}_{0}\atop{a+2b = n}} r_3(a) r_3(8b+3),\\
				\delta_{8}(n)  = & \sum_{a,b \in {\mathbb N}_{0}\atop{a+2b = n}} r_4(a) \sigma(2b+1),\\
				\delta_{12}(n) = & - \frac{1}{8} \sum_{a,b \in {\mathbb N}_{0}\atop{a+2b = n}} r_6(a) \sigma_{2; {\bf 1},\chi_{-4}}(4b+3),\\
				\delta_{16}(n) = & \sum_{a,b \in {\mathbb N}_{0}\atop{a+2b = n}} r_8(a) \sigma_{3}^{\#}(b+1),\\
			\end{split}
		\end{equation*}
		where $\sigma_{3}^{\#}(n) := \displaystyle{\sum_{d \vert n, ~{\frac{n}{d}-odd}}} d^{3}$ and 
		$\sigma_{k;\chi,\psi}(n)$ is the generalised divisor function given by
		\begin{equation}\label{Gen-DF}
			\begin{split}
				\sigma_{k;\chi,\psi}(n) :=  \displaystyle{\sum_{d\vert n} \psi(d) \cdot \chi(n/d) d^{k}},
			\end{split}
		\end{equation}
		with $\chi$, $\psi$ being primitive Dirichlet characters. In the above formula ${\bf 1}$ denotes the trivial character 
		${\bf 1}(n) =1$ for all $n\ge 1$  and $\chi_{-4}$ is the odd Dirichlet character modulo $4$. 
	\end{cor}
	
	\begin{proof}
		The following formulas are established in \cite[pp. 77--81]{ono}.
		\begin{equation*}
			\begin{split}
				\delta_{2}(n) ~&=~ \frac{1}{4} r_2(8n+2); \quad  \delta_{3}(n) = \frac{1}{8} r_3(8n+3); \quad \delta_{4}(n) ~=~ \sigma(2b+1),\\
				\delta_{6}(n) ~&=~ - \frac{1}{8} \sigma_{2;{\bf 1},\chi_{-4}}(4n+3);\quad  \delta_{8}(n) = \sigma_{3}^{\#}(n+1). 
			\end{split}
		\end{equation*}
		Using these formulas in \corref{relations}, we have the required formulas.
	\end{proof}
	
	\begin{rmk}
		When $k=4,6,8$, we have the following known formulas for $r_k(n)$, which are given below. 
		\begin{equation*}
			\begin{split}
				r_4(n) &= 8 \sigma(n)- 32 \sigma(n/4),\\
				r_6(n) &= -4 \sigma_{2;{\bf 1}, \chi_{-4}}(n) + 16 \sigma_{2; \chi_{-4},{\bf 1}}(n),\\
				r_8(n) &= 16 \sigma_{3}(n)- 32 \sigma_{3}(n/2) + 256 \sigma_{3}(n/4). 
			\end{split}
		\end{equation*}
		Among these, $r_2(n)$ and $r_4(n)$ are well-known. To get the above expression for $r_6(n)$, we write $\theta^6(\tau)$ in terms of the 
		two Eisenstein series $E_{3;{\bf 1}, \chi_{-4}}(\tau)$ and $E_{3;,\chi_{-4},{\bf 1}}(\tau)$ (See appendix for definition) and it is easy to get the above formula by comparing the 
		$n^{\rm th}$ Fourier coefficients. Now, using the above formulas in \corref{relations1}, we obtain the following formulas which involve  only the divisor functions.
		For a natural number $n$, we have
		\begin{equation*}
			\begin{split}
				\delta_{8}(n)  &=  8 \sum_{a,b \in {\mathbb N}_{0}\atop{a+2b = n}} \sigma(a) \sigma(2b+1)-32 \sum_{a,b \in {\mathbb N}_{0}\atop{a+2b = n}} \sigma(a/4) \sigma(2b+1) .\\
				\delta_{12}(n) &=   \frac{1}{2} \sum_{a,b \in {\mathbb N}_{0}\atop{a+2b = n}} \sigma_{2;{\bf 1},\chi_{-4}}(a) \sigma_{2;{\bf 1},\chi_{-4}}(4b+3) 
				- 2 \sum_{a,b \in {\mathbb N}_{0}\atop{a+2b = n}} \sigma_{2;\chi_{-4},{\bf 1}}(a) \sigma_{2; {\bf 1},\chi_{-4}}(4b+3).\\
				\delta_{16}(n) &=  16 \sum_{a,b \in {\mathbb N}_{0}\atop{a+2b = n}} \sigma_{3}(a) \sigma_{3}^{\#}(b+1) - 32 \sum_{a,b \in {\mathbb N}_{0}\atop{a+2b = n}} \sigma_{3}(a/2) \sigma_{3}^{\#}(b+1) + 256 \sum_{a,b \in {\mathbb N}_{0}\atop{a+2b = n}} \sigma_{3}(a/4) \sigma_{3}^{\#}(b+1).\\
			\end{split}
		\end{equation*}
	\end{rmk}
	
	\bigskip
	
	\begin{prop}\label{ellipsoid}
		Let $R$ be a positive real number. Then, the $k$-dimensional ellipsoid with axis lengths $R/\sqrt{c_{i}}$, $1\le i\le k$, 
		centred at $(1/2,1/2, \ldots, 1/2)$ contains $2^{k} \sum_{n = 1}^{[\frac{R^2}{2}-\frac{{\bf h}}{8}]} 
		\delta_{k}({\mathcal C};n)$ lattice points in ${\mathbb Z}^k$.
		%where $R$ is a positive real number. 
	\end{prop}
	
	\begin{proof}
		The proof is similar to the one given in \cite{ono}. 
		Let $r$ be a positive integer and assume that the $k$-dimensional ellipsoid with axis lengths $r/\sqrt{c_{i}}$, $1\le i\le k$ centred at $(1/2, \ldots, 1/2)$ 
		contains a lattice point $(z_{1}, z_{2}, \ldots, z_{k}) \in {\mathbb Z}^k$. Then, we have 
		\begin{equation*}
			%\begin{split}
			r^2 ~ =~ \sum_{i=1}^{k} c_{i} \left(z_{i}-\frac{1}{2}\right)^2 = \sum_{i=1}^{k} c_{i} [(z_{i}^2 - z_{i}) + 1/4] ~=  \sum_{i=1}^{k} c_{i} z_{i}(z_{i}-1) 
			+ \frac{1}{4} \sum_{i=1}^{k} c_{i} ~=~\sum_{i=1}^{k} c_{i} z_{i}(z_{i}-1) + \frac{{\bf h}}{4}. \\
		\end{equation*}
		Therefore, $\frac{r^2}{2} -  \frac{{\bf h}}{8} = \sum_{i=1}^{k} c_{i} \frac{z_{i}(z_{i}-1)}{2}$. 
		This shows that we have a representation of $\frac{r^2}{2}- \frac{{\bf h}}{8}$ as a sum of $k$ triangular numbers with coefficients $c_i$, $1\le i\le k$. Since $z_i$ and $(-z_i-1)$ represent the same triangular number, the total 
		number of lattice points inside this ellipsoid is equal to  $2^{k} \delta_{k}(\mathcal{C};[\frac{r^2}{2}-\frac{{\bf h}}{8}])$. 
		This completes the proof. 
	\end{proof}

	\smallskip
	The rest of the article is organised as follows.  In \S2, we recall some of the results on eta quotients and generalised eta quotients and  present proofs of results considered in \S 1.  In \S 3, we obtain formulas for the mixed forms considered in the work of Xia-Ma-Tian \cite{xia} using a basis of the respective vector space of modular forms.  We also derive the $(p,k)$-parametrisation of Eisenstein series $E_{4}$ and its duplications. In \S 4, we give general formulas using the modular properties obtained in \S 1 and also 
	give some explicit formulas for the forms associated to figurate numbers and mixed forms (to make it brief, we have provided only a few examples in each case).  In the appendix section, we provide explicit basis for the space of modular forms 
	which are used to obtain the sample formulas presented in \S 4.

	%\newpage
	\section{Proofs of theorems}
	
	%\subsection{Modularity of generating function associated to higher figurate form.} ~~
	We need the following result on the modularity of the generalised eta quotient due to S. Robins, which will be useful to prove the modularity result for the generating function for higher figurate numbers. So, first we state Robin's theorem. 
	\smallskip
	
	\noindent {\bf Theorem A} (S.Robins \cite{Robins}): 
	{\em For $N \in {\mathbb N}$, let $$r_{\delta, g} \in \begin{cases} \frac{1}{2}{\mathbb Z} & {\rm ~if~} g = 0, \frac{\delta}{2},\\
			{\mathbb Z}  & otherwise.\\
		\end{cases}
		$$
		Let $f(\tau) = \displaystyle{\prod_{\delta \vert N \atop{g \!\! \pmod {\delta}}}\eta_{\delta, g}^{r_{\delta, g}}(\tau)}$ be a generealised eta-quotient such that 
		$k = \frac{1}{2}\displaystyle{\sum_{\delta \vert N} r_{\delta, 0}}$ is a non-negative integer. If $f(\tau)$ satisfies the the following conditions 
		\begin{enumerate}
			\item[{(i)}]
			$\displaystyle{\sum_{\delta \vert N \atop{g \!\! \pmod {\delta}}} \delta \ P_2\left(\frac{g}{\delta}\right) r_{\delta,g} \equiv 0\pmod{2}}$ $\quad {\rm and} \quad $
			$\displaystyle{\sum_{\delta \vert N \atop{g \!\! \pmod {\delta}}} \frac{N}{\delta} P_2(0) r_{\delta,g} \equiv 0\pmod{2}}$.
			\item[{(ii)}]
			The order of vanishing at cusp $(s = \frac{\lambda}{\mu \epsilon}, ~ \gcd(\lambda,N) = 1 = \gcd(\mu,N),~ \epsilon \vert N )$ is equal to 
			$$
			\displaystyle{\sum_{\delta \vert N \atop{g \!\! \!\! \pmod {\delta}}} \frac{\gcd(\delta, \epsilon)^2}{\delta \epsilon} P_2 \left({\frac{\lambda g}{\gcd(\delta, \epsilon)}}\right) r_{\delta,g}},
			$$ 
			where $P_2(x):= \{x\}^2 -\{x\}+\frac{1}{6}$ is generalised Bernoulli polynomial and $\{x\}$ is the fractional part of $x$.
		\end{enumerate}
		If the order of vanishing at each cusp is non-negative, then $f(\tau)$ is a modular form in $M_k(\Gamma_1(N))$. If the order of vanishing is positive at each cusp, then $f(\tau) \in S_k(\Gamma_1(N))$.\\
	}
	
	\smallskip
	
	We also need the following lemma in our proof. 
	
	\smallskip
	
	\begin{lem}\label{positivity}
		Let $f(x) =\frac{1}{12} +\frac{1}{2}P_{2}(2x) - P_{2}(x).$ Then for $ 0 \le x \le 1$, we have $f(x) \ge 0.$ Moreover, it is a periodic function of period $1.$
	\end{lem}

	\begin{proof}
		Let $f(x) :=\frac{1}{12} +\frac{1}{2}P_{2}(2x) - P_{2}(x)$.  Since $ \{2x\} = 2x $ for   $ 0 \le x < \frac{1}{2}$, and $ \{2x\} = 2x-1 $  for  $ \frac{1}{2} < x \le 1$. So, we have
		$$
		f(x) = 
		\begin{cases}
			x^2 &{\rm ~if~} 0 \le x < \frac{1}{2},\\
			(x-1)^2 &{\rm ~if~} \frac{1}{2} < x \le 1, \\
			\frac{1}{4} & {\rm ~if~} x = \frac{1}{2}.
		\end{cases}
		$$
		This gives that $f(x) \ge 0$ for $0 \le x \le 1$, Moreover $f(x)$ is a periodic function of period 1. Thus, we have $f(x) \ge 0$ for all $x \ge 0$.\\
	\end{proof}

	\subsection{Proof of \thmref{modularity}} First we consider the case when $a\ge 3$. 
	We prove the modular property of the function
	$q^{\frac{(a-2)^2}{8a} {\bf h}}\Phi_a^{\mathcal C}(\tau)$ with ${\bf h} = \sum_{i=1}^k c_i$,  using the generalised eta quotient expression given in \eqref{figurate-gen-fun} and Theorem A. From \eqref{figurate-gen-fun}, we have 
	\begin{equation*}
		\begin{split}
			q^{\frac{(a-2)^2}{8a}\bf h}\Phi_a^{\mathcal C}(\tau) & = \prod_{i=1}^k \eta_{a c_i, 0}^{1/2}(\tau) \prod_{i=1}^k \frac{\eta_{2a c_i, 2c_i}(\tau)}{\eta_{a c_i, c_i}(\tau)}.
		\end{split}
	\end{equation*}
	Now, we check the first condition of modular property for the RHS of \eqref{figurate-gen-fun} which is generalised Dedekind eta quotient, i.e.,
	\begin{equation*}
		\begin{split}
			\displaystyle{\sum_{\delta \vert N \atop{g \!\! \pmod {\delta}}} \delta P_2(g/\delta) r_{\delta,g}} & =  \frac{1}{2} \sum_{i=1}^k a c_i P_2(0) + \sum_{i=1}^k 2a c_i P_2\left(\frac{2 c_i}{2ac_i}\right) -
			\sum_{i=1}^k a c_i P_2\left(\frac{c_i}{ac_i}\right) \\
			& \quad = \left(\frac{a}{12} + a P_2\left(\frac{1}{a}\right) \right) \sum_{i=1}^k  c_i   = \quad = \frac{(a-2)^2}{4a} \sum_{i=1}^k  c_i \equiv 0 \pmod{2} \\
			%& \quad = \frac{1}{12} \sum_{i=1}^k a c_i  -\sum_{i=1}^k a c_i P_2(\frac{1}{a}) \\
			%& \quad = \sum_{i=1}^k a c_i \frac{(a-2)^2}{4a^2}\equiv 0 \pmod{2}\\
		\end{split}
	\end{equation*}
	This gives that ${\bf h} \equiv 0 \pmod{\frac{8a}{d}}$, where ${\bf h} = \sum_{i=1}^k c_i$ and $d = \gcd((a-2)^2,4a)$ and second condition is trivially zero, i.e.,
	\begin{equation*}
		\begin{split}
			\displaystyle{\sum_{\delta \vert N \atop{g \!\! \pmod {\delta}}} \frac{N}{\delta} P_2(0) r_{\delta,g}} & = \sum_{i=1}^k \frac{N}{ac_i} \frac{1}{6}\frac{1}{2} + \sum_{i=1}^k  \frac{N}{2ac_i} \frac{1}{6} -  \sum_{i=1}^k \frac{N}{ac_i} \frac{1}{6} =0.
			\\
			%& \quad = 0 \\
		\end{split}
	\end{equation*}
	From Theorem A, the order of vanishing at the cusp $s$ $(s = \frac{\lambda}{\mu \epsilon}, \gcd(\lambda,N) = 1 = \gcd(\mu,N), \epsilon \vert N )$ of the  function $q^{\frac{(a-2)^2}{8a}{\bf h}}\Phi_a^{\mathcal C}(\tau)$ $\left(= \prod_{i=1}^k \eta_{a c_i, 0}^{1/2}(\tau) \prod_{i=1}^k \frac{\eta_{2a c_i, 2c_i}(\tau)}{\eta_{a c_i, c_i}(\tau)}\right)$ is equal to (say $\nu_{s}$)  
	\begin{equation*}
		\begin{split}
			\displaystyle{\sum_{i=1}^k \frac{\gcd(ac_i, \epsilon)^2}{a c_i \epsilon} P_2(0)\frac{1}{2}} + \displaystyle{\sum_{i=1}^k \frac{\gcd(2ac_i, \epsilon)^2}{2a c_i \epsilon} P_2\left({\frac{\lambda. 2c_i}{\gcd(2ac_i, \epsilon)}}\right).1} - \displaystyle{\sum_{i=1}^k \frac{\gcd(ac_i, \epsilon)^2}{a c_i \epsilon} P_2\left({\frac{\lambda c_i}{\gcd(ac_i, \epsilon)}}\right).1}
		\end{split}
	\end{equation*}
	
	\begin{equation*}
		\begin{split}
			=\displaystyle{\sum_{i=1}^k \left( \frac{\gcd(ac_i, \epsilon)^2}{a c_i \epsilon} \frac{1}{12} + \frac{\gcd(2ac_i, \epsilon)^2}{2a c_i \epsilon} P_2\left({\frac{\lambda. 2c_i}{\gcd(2ac_i, \epsilon)}}\right) -\frac{\gcd(ac_i, \epsilon)^2}{a c_i \epsilon} P_2\left({\frac{\lambda c_i}{\gcd(ac_i, \epsilon)}}\right)\right)}.
		\end{split}
	\end{equation*}
	We write the term inside the summation as $\nu_{s}^{i}$,
	i.e.,
	$$
	\nu_{s}^{i} = \left( \frac{\gcd(ac_i, \epsilon)^2}{a c_i \epsilon} \frac{1}{12} + \frac{\gcd(2ac_i, \epsilon)^2}{2a c_i \epsilon} P_2\left({\frac{\lambda. 2c_i}{\gcd(2ac_i, \epsilon)}}\right) -\frac{\gcd(ac_i, \epsilon)^2}{a c_i \epsilon} P_2\left({\frac{\lambda c_i}{\gcd(ac_i, \epsilon)}}\right)\right).
	$$ 
	It is enough to show that  $\nu_{s}^{i}$ is non-negative for each $i$.  We consider the following cases to show  $\nu_{s}^{i}$ is non-negative for each $i$.
	
	\noindent
	{\bf Case 1:} If $\gcd(2ac_i, \epsilon) = 2~~ \gcd(ac_i, \epsilon)$, then 
	\begin{equation*}
		\begin{split}
			\nu_{s}^{i} = \frac{\gcd(ac_i, \epsilon)^2}{a c_i \epsilon} \left(\frac{1}{12}+P_2\left({\frac{\lambda c_i}{\gcd(ac_i, \epsilon)}}\right)\right). \\
		\end{split}
	\end{equation*}
	\noindent 
	We know that for $ 0 \le x \leq 1$, we have $P_{2}(x) \ge \frac{-1}{12}$ and $P_{2}(x)$ is periodic function of period $1$. 
	this show that $\nu_{s}^{i} \ge 0$.
	
	\noindent
	{\bf Case 2:} If $\gcd(2ac_i, \epsilon) = \gcd(ac_i, \epsilon)$, then we have 
	\begin{equation*}
		\begin{split}
			\nu_{s}^{i} =  \frac{\gcd(ac_i, \epsilon)^2}{a c_i \epsilon} \left(\frac{1}{12}+ \frac{1}{2}P_2\left({\frac{2\lambda c_i}{\gcd(ac_i, \epsilon)}}\right)-P_2\left({\frac{\lambda c_i}{\gcd(ac_i, \epsilon)}}\right)\right). \\
		\end{split}
	\end{equation*}
	Note that the non-negativity of $\nu_{s}^{i}$ follows from \lemref{positivity}. This completes the proof in the case when $a\ge 3$.  
	
	Next we consider the case when $a=1$ and  the corresponding generating function is $\Psi_{\mathcal C}(\tau)$, which corresponds to triangular numbers. We use the expression for the generating function $\Psi_{\mathcal C}(\tau)$ in terms of eta-quotient given in  \eqref{Gen-function}  to prove the result with the help of modularity result for the eta quotients obtained by Dummit, Kisilevsky and McKay \cite{dummit} 
	(we also refer to \cite[Theorem 1.64]{ono-web}). From \eqref{triangular-gen-fun}, we have 
	\begin{equation}
		%\begin{split}
		\Psi_{\mathcal C}(\tau)  = \prod_{i=1}^k \Psi(c_i \tau) ~
		= q^{-{\bf h}/8} \prod_{i=1}^k \frac{\eta^2(2c_i \tau)}{\eta(c_i \tau)}.
		%\end{split}
	\end{equation}
	Therefore, $q^{{\bf h}/8} \Psi_{\mathcal C}(\tau)$ is the eta-quotient $\displaystyle{\prod_{i=1}^k 
		\frac{\eta^2(2c_i \tau)}{\eta(c_i \tau)}}$. Now we use \cite[Theorem 1.64]{ono-web} to get the required modular property. 
	The following properties are used to complete the proof of the theorem: (i) ${\bf h} = c_1+\cdots+c_k \equiv 0\pmod{8}$, which is by assumption;  
	\begin{equation*}
		{\rm (ii)} \hskip 3cm  \sum_{\delta\vert N} \frac{\gcd(d,\delta)^2 r_\delta}{d \delta}  = \sum_{i=1}^k \left(2 \frac{\gcd(d,2c_i)^2}{ 2dc_{i}} - \frac{\gcd(d,c_i)^2}{d c_i} \right)\ge 0, \hskip 3.5cm 
	\end{equation*}
	which follows from  the fact that $\gcd(d,2c_i) \ge \gcd(d, c_i)$ for every $d\ge 1$. 
	The weight and character of the modular form are determined easily. 
	This completes the proof of \thmref{modularity}.
	
	\smallskip
	
	\subsection{Proof of \thmref{odd-triangular}}
	Since $\theta(\tau) = \eta^5(2\tau)/(\eta^2(\tau)\eta^2(4\tau))$, using \eqref{gen;st} and  \eqref{triangular-gen-fun}, we have
	\begin{equation*}
		\begin{split}
			\psi_{s,t}(\tau) & =  \prod_{i=1}^{v} \theta(b_{i}\tau).\prod_{j=1}^{k} \Psi(c_{j}\tau)  =\prod_{i=1}^v \frac{\eta^5(2b_i \tau)}{\eta^2(b_i \tau)\eta^2(4b_i \tau)}  \times  q^{-{\bf h}/8} \prod_{j=1}^k \frac{\eta^2(2c_j \tau)}{\eta(c_j \tau)}  .
		\end{split}
	\end{equation*}
	Therefore, we have 
	\begin{equation} \label{sq-triang}
		q^{{\bf h}/8} \psi_{s,t}(\tau) = \displaystyle{\prod_{i=1}^v \frac{\eta^5(2b_i \tau)}{\eta^2(b_i \tau)\eta^2(4b_i \tau)} \prod_{i=1}^k \frac{\eta^2(2c_i \tau)}{\eta(c_i \tau)}}.
	\end{equation}
	The modularity of the above eta-quotient follows from \cite[Theorem 1.64]{ono-web} and the inequality 
	$$
	10 \gcd(d, 2b)^{2} - 8 \gcd(d, b)^{2} - 2 \gcd(d, 4b)^{2} \ge 0 {\rm~ for ~ each ~ positive~ integers~} d {\rm ~ and ~} b.
	$$
	%As in the proof of \thmref{modular}, we get the modular property of this eta-quotient. 
	This completes the proof.
	
	\smallskip

	\section{Mixed forms considered by Xia-Ma-Tian} 
	
	%\subsection{Formulas for the $21$ mixed forms and the $(p,k)$-parametrisation of the Eisenstein series $E_4(d\tau)$, $d\vert 12$.} \quad 
	
	In this section, we first derive formulas for the 21 mixed forms considered in the work of Xia-Ma-Tian \cite{xia} using our method. 
	In \corref{mixed-m} we observed that the generating function $q^{{\bf h}/8} \Phi(\tau)$ corresponding to the mixed form ${\mathcal M}({\bf x}, {\bf y}, {\bf z})$ is a modular form of weight $u+(v+k)/2$ on some subgroup $\Gamma_0(N)$ with character $\omega$, depending on the coefficients involved in the mixed form. 
	In the above, {\bf h} is the sum of the coefficients appearing in the triangular part of the mixed form.  For the $21$ cases, it turns out that all the generating functions lie in the space $M_4(\Gamma_0(12))$.  We use the following basis for the space $M_4(\Gamma_0(12))$ (as described in \cite{rss-ijnt}): 
	$$
	E_4(d\tau), d\vert 12; f_{4,6}(\tau), f_{4,6}(2\tau), f_{4,12}(\tau),
	$$
	where $E_4(\tau)$ is the normalised Eisenstein series of weight $4$ for $SL_2({\mathbb Z})$ and $f_{4,6}(\tau)$, $f_{4,12}(\tau)$ are 
	the newforms of weight $4$ on $\Gamma_0(6)$ and $\Gamma_0(12)$ respectively, which are given explicitly in terms of the eta-functions as 
	follows.  
	\begin{equation*}
		\begin{split}
			f_{4,6}(\tau) & = \eta^2(\tau) \eta^2(2\tau) \eta^2(3\tau) \eta^2(6\tau), \\
			f_{4,12}(\tau)& = \frac{\eta^2(2\tau) \eta^3(3\tau) \eta^3(4\tau)\eta^2(6\tau)}{\eta(\tau) \eta(12\tau)} - 
			\frac{\eta^3(\tau) \eta^2(2\tau) \eta^2(6\tau)\eta^3(12\tau)}{\eta(3\tau) \eta(4\tau)}.
		\end{split}
	\end{equation*}
	So, we express the generating functions $q^{{\bf h}/8} \Phi(\tau)$ as a linear combination of the above basis elements as 
	\begin{equation}\label{21mixed}
		q^{{\bf h}/8} \Phi(\tau) = \sum_{d\vert 12} \alpha_d E_4(d\tau) + c_1 f_{4,6}(\tau) + c_2 f_{4,6}(2\tau) + c_3 f_{4,12}(\tau).
	\end{equation}
	In the following table, we give the mixed forms and the corresponding coefficients $\alpha_d$, $d\vert 12$, $c_i$, $i=1,2,3$
	(we have used the expression given in \eqref{mixed-eg} for representing the generating functions).
	
	\smallskip

	\begin{center}
		
		{\bf Table 1. Coefficients table for the 21 mixed forms}
		
		\smallskip
		
		\renewcommand{\arraystretch}{1.2}
		%\tiny
		%\begin{sidewaystable}[htb]
		\begin{tabularx}{\textwidth}{|p{5.25cm}| p{0.8cm}|p{0.8cm}|p{0.8cm}|p{0.8cm}|p{0.8cm}|p{0.8cm}|p{0.8cm}|p{0.8cm}|p{0.8cm}|}
			
			%\caption{Coefficients for the 21 mixed forms  }\\
			\hline
			
			\multicolumn{1}{|c|}{\textbf{Forms}}  & \multicolumn{9}{|c|}{\textbf{Coefficients}}\\ \cline{2-10} 
			{} & {$\alpha_1$} & {$\alpha_2$} & {$\alpha_3$} & {$\alpha_4$} & {$\alpha_6$} & {$\alpha_{12}$} & {$c_1$} & {$c_2$} & {$c_3$}  
			\\ \hline 
			
			\endfirsthead
			\caption{Coefficients for the 21 mixed forms}\\
			\hline
			\multicolumn{1}{|c|}{\textbf{Form}}  & \multicolumn{9}{|c|}{\textbf{Coefficients}}\\ \cline{2-10} 
			{} & {$\alpha_1$} & {$\alpha_2$} & {$\alpha_3$} & {$\alpha_4$} & {$\alpha_6$} & {$\alpha_{12}$} & {$c_1$} & {$c_2$} & {$c_3$}  
			\\ \hline 
			
			\endhead
			\hline
			\multicolumn{10}{|r|}{{Continued on Next Page\ldots}}     \\ \hline
			
			\endfoot
			
			\endlastfoot
			\hline
			
			$ {\mathcal F}(2\tau)                           \theta^{3}(\tau) \theta^{3}(3\tau)                                              $ & $  \frac{1}{120} $  & $  0 $  & $  \frac{-3}{40} $  & $  \frac{-2}{15} $  & $  0 $  & $  \frac{6}{5} $  & $  0 $  & $  0 $  & $  4 $  \\ \hline 
			$ q{\mathcal F}(2\tau)                           \theta^{2}(\tau) \theta^{2}(3\tau) \Psi(2\tau) \Psi(6\tau)                     $ & $  \frac{1}{480} $  & $  \frac{-1}{480} $  & $  \frac{-3}{160} $  & $  0 $  & $  \frac{3}{160} $  & $  0 $  & $  0 $  & $  0 $  & $  \frac{1}{2} $  \\ \hline 
			$ q^{2}{\mathcal F}(2\tau)                           \theta(\tau) \theta(3\tau) \Psi^{2}(2\tau) \Psi^{2}(6\tau)                 $ & $  \frac{1}{1920} $  & $  \frac{-1}{1920} $  & $  \frac{-3}{640} $  & $  0 $  & $  \frac{3}{640} $  & $  0 $  & $  0 $  & $  0 $  & $  \frac{-1}{8} $  \\ \hline 
			$ q^{2}{\mathcal F}(4\tau)                           \theta^{2}(\tau)                   \Psi(\tau)  \Psi(3\tau) \Psi^{2}(6\tau) $ & $  \frac{1}{1920} $  & $  \frac{-1}{1920} $  & $  \frac{-3}{640} $  & $  0 $  & $  \frac{3}{640} $  & $  0 $  & $  \frac{-1}{2} $  & $  -1 $  & $  \frac{3}{8} $  \\ \hline 
			$ q^{2}{\mathcal F}(4\tau)                           \theta^{2}(3\tau)                  \Psi(\tau)  \Psi^{2}(2\tau) \Psi(3\tau) $ & $  \frac{1}{1920} $  & $  \frac{-1}{1920} $  & $  \frac{-3}{640} $  & $  0 $  & $  \frac{3}{640} $  & $  0 $  & $  \frac{1}{2} $  & $  1 $  & $  \frac{3}{8} $  \\ \hline 
			$ {\mathcal F}(\tau)  {\mathcal F}(2\tau)  \theta^{4}(3\tau)                                                                    $ & $  \frac{1}{300} $  & $  \frac{-1}{200} $  & $  \frac{13}{100} $  & $  \frac{2}{75} $  & $  \frac{-39}{200} $  & $  \frac{26}{25} $  & $  \frac{16}{5} $  & $  \frac{32}{5} $  & $  2 $  \\ \hline 
			$ {\mathcal F}(\tau)  {\mathcal F}(2\tau)  \theta^{2}(\tau) \theta^{2}(3\tau)                                                   $ & $  \frac{1}{60} $  & $  \frac{-1}{120} $  & $  \frac{-3}{20} $  & $  \frac{-2}{15} $  & $  \frac{3}{40} $  & $  \frac{6}{5} $  & $  0 $  & $  0 $  & $  6 $  \\ \hline 
			$ {\mathcal F}(\tau)  {\mathcal F}(2\tau)  \theta^{4}(\tau)                                                                     $ & $  \frac{13}{300} $  & $  \frac{-13}{200} $  & $  \frac{9}{100} $  & $  \frac{26}{75} $  & $  \frac{-27}{200} $  & $  \frac{18}{25} $  & $  \frac{48}{5} $  & $  \frac{96}{5} $  & $  -6 $  \\ \hline 
			$ q{\mathcal F}(\tau)  {\mathcal F}(2\tau)                                     \Psi^{2}(\tau) \Psi^{2}(3\tau)                   $ & $  \frac{1}{240} $  & $  \frac{-1}{240} $  & $  \frac{-3}{80} $  & $  0 $  & $  \frac{3}{80} $  & $  0 $  & $  0 $  & $  0 $  & $  0 $  \\ \hline 
			$ q^{2}{\mathcal F}(\tau)  {\mathcal F}(2\tau)                                     \Psi^{2}(2\tau) \Psi^{2}(6\tau)              $ & $  \frac{1}{640} $  & $  \frac{-19}{1920} $  & $  \frac{-9}{640} $  & $  \frac{1}{120} $  & $  \frac{57}{640} $  & $  \frac{-3}{40} $  & $  0 $  & $  0 $  & $  \frac{-3}{8} $  \\ \hline 
			$ {\mathcal F}(2\tau) {\mathcal F}(4\tau) \theta^{2}(\tau) \theta^{2}(3\tau)                                                    $ & $  \frac{1}{240} $  & $  \frac{1}{240} $  & $  \frac{-3}{80} $  & $  \frac{-2}{15} $  & $  \frac{-3}{80} $  & $  \frac{6}{5} $  & $  0 $  & $  0 $  & $  3 $  \\ \hline 
			$ q{\mathcal F}(2\tau) {\mathcal F}(4\tau)                                    \Psi^{2}(\tau)\Psi^{2}(3\tau)                     $ & $  \frac{1}{960} $  & $  \frac{-1}{960} $  & $  \frac{-3}{320} $  & $  0 $  & $  \frac{3}{320} $  & $  0 $  & $  0 $  & $  0 $  & $  \frac{3}{4} $  \\ \hline 
			$ {\mathcal F}^{3}(2\tau)                          \theta(\tau) \theta(3\tau)                                                   $ & $  \frac{1}{120} $  & $  0 $  & $  \frac{-3}{40} $  & $  \frac{-2}{15} $  & $  0 $  & $  \frac{6}{5} $  & $  0 $  & $  0 $  & $  0 $  \\ \hline 
			$ q{\mathcal F}^{3}(2\tau)                                                              \Psi(2\tau) \Psi(6\tau)                 $ & $ \frac{1}{240}$  & $ \frac{-3}{80}$  & $ \frac{-3}{80}$  & $ \frac{1}{30}$  & $ \frac{27}{80}$  & $\frac{-3}{10}$  & $ 0$  & $ 0$  & $ 0$  \\  \hline
			$ {\mathcal F}^{3}(4\tau)                        \theta(\tau) \theta(3\tau)                                                     $ & $  \frac{1}{1200} $  & $  \frac{-3}{400} $  & $  \frac{3}{400} $  & $  \frac{8}{75} $  & $  \frac{-27}{400} $  & $  \frac{24}{25} $  & $  \frac{9}{5} $  & $  \frac{18}{5} $  & $  0 $  \\ \hline 
			$ {\mathcal F}^{2}(\tau) {\mathcal F}(2\tau) \theta(\tau) \theta(3\tau)                                                         $ & $  \frac{1}{30} $  & $  \frac{-1}{40} $  & $  \frac{-3}{10} $  & $  \frac{-2}{15} $  & $  \frac{9}{40} $  & $  \frac{6}{5} $  & $  0 $  & $  0 $  & $  6 $  \\ \hline 
			$ q{\mathcal F}^{2}(\tau) {\mathcal F}(2\tau)                                       \Psi(2\tau) \Psi(6\tau)                     $ & $  \frac{1}{96} $  & $  \frac{-7}{160} $  & $  \frac{-3}{32} $  & $  \frac{1}{30} $  & $  \frac{63}{160} $  & $  \frac{-3}{10} $  & $  0 $  & $  0 $  & $  \frac{-3}{2} $  \\ \hline 
			$ {\mathcal F}(\tau) {\mathcal F}^{2}(2\tau)  \theta(\tau) \theta(3\tau)                                                        $ & $  \frac{1}{75} $  & $  \frac{-1}{50} $  & $  \frac{3}{25} $  & $  \frac{8}{75} $  & $  \frac{-9}{50} $  & $  \frac{24}{25} $  & $  \frac{24}{5} $  & $  \frac{48}{5} $  & $  0 $  \\ \hline 
			$ {\mathcal F}(2\tau){\mathcal F}^{2}(4\tau)  \theta(\tau) \theta(3\tau)                                                        $ & $  \frac{1}{480} $  & $  \frac{1}{160} $  & $  \frac{-3}{160} $  & $  \frac{-2}{15} $  & $  \frac{-9}{160} $  & $  \frac{6}{5} $  & $  0 $  & $  0 $  & $  \frac{3}{2} $  \\ \hline 
			$ {\mathcal F}(\tau) {\mathcal F}(2\tau) {\mathcal F}(4\tau) \theta(\tau) \theta(3\tau)                                         $ & $  \frac{1}{120} $  & $  0 $  & $  \frac{-3}{40} $  & $  \frac{-2}{15} $  & $  0 $  & $  \frac{6}{5} $  & $  0 $  & $  0 $  & $  6 $  \\ \hline 
			$ q{\mathcal F}(\tau) {\mathcal F}(2\tau) {\mathcal F}(4\tau) \Psi(2\tau) \Psi(6\tau)                                           $ & $  \frac{1}{960} $  & $  \frac{1}{64} $  & $  \frac{-3}{320} $  & $  \frac{-1}{60} $  & $  \frac{-9}{64} $  & $  \frac{3}{20} $  & $  0 $  & $  0 $  & $  \frac{3}{4} $  \\ \hline 
			
		\end{tabularx}
	\end{center}
	
	\smallskip
	
	\begin{rmk}
		By comparing the $n^{\rm th}$ Fourier coefficients of the expressions in \eqref{21mixed}, we obtain explicit formulas 
		obtained in Theorem 1.1 of \cite{xia}. The formulas obtained by our method are exactly the same as in \cite{xia}, except 
		for the expressions arising from the cusp form Fourier coefficients.  We describe the difference below. 
		
		The space of cusp forms $S_4(\Gamma_0(12))$ is three dimensional spanned by $f_{4,6}(\tau), f_{4,6}(2\tau)$ and $f_{4,12}(\tau)$. However, in \cite[Theorem 3.1]{xia}, they make use of only two cusp forms for getting the required expressions. We observe that the two cusp 
		forms appear in \cite[Theorem 3.1]{xia}, denoted by $G(\tau)$ and $H(\tau)$, can be given in terms of our basis elements as follows.
		\begin{equation}\label{gh}
			\begin{split}
				G(\tau)& =  -\frac{1}{6} f_{4,6}(\tau) -\frac{1}{3} f_{4,6}(2\tau) + \frac{1}{6} f_{4,12}(\tau), \\
				H(\tau)& =  \frac{1}{2} f_{4,6}(\tau) + f_{4,6}(2\tau) + \frac{1}{2} f_{4,12}(\tau). 
			\end{split}
		\end{equation}
		If we denote the $n^{\rm th}$ Fourier coefficients of $f_{4,N}(z)$ as $a_{4,N}(n)$, $N=6,12$, then we can get the $n^{\rm th}$ Fourier coefficients of $G$ and $H$ in terms of $a_{4,6}(n)$, $a_{4,6}(n/2)$ and $a_{4,12}(n)$. Thus, we observe that both the expressions (obtained by Xia-Ma-Tian and obtained by our method) are the same for all the 21 mixed forms. 
	\end{rmk}
	
	\smallskip
	
	\noindent {\bf The $(p,k)$ parametrisation of $E_4(z)$}:\\
	We now make the following observation about obtaining representations of $E_4(\tau)$ and its duplications in terms of $p,k$ using our identities \eqref{21mixed}. 
	To get the identities given by \eqref{21mixed} (with $G(\tau), H(\tau)$ in place of $f_{4,6}(\tau), f_{4,6}(2\tau), f_{4,12}(\tau)$), Xia-Ma-Tian used the method of $(p,k)$-parametrisation. 
	Let 
	\begin{equation*}
		\begin{split}
			p & =  p(\tau): =  \frac{\theta^2(\tau)-\theta^2(3\tau)}{2 \theta^2(3\tau)} \qquad {\rm and} \qquad  k  = k(\tau): =  \frac{\theta^3(3\tau)}{\theta(\tau)}. \\
		\end{split}
	\end{equation*}
	In \cite[Theorems 1,2,4]{aaw2}, $(p,k)$-parametrisations of ${\mathcal{•}al F}(d\tau)$, $d=1,2,4$ are given and in \cite[(2.3)]{aalw1} the parametrisation is 
	given for the theta series $\theta(\tau)$ and $\theta(3\tau)$. Further, in \cite{aaw1}, the representations of $q^{j/24}\prod_{n=1}^\infty(1-q^{nj})$, 
	$j=1,2,3,4,6,12$ in terms of $p$ and $k$ have been established. In \cite{aw}, Alaca and Williams derived the representations of $E_4(d\tau)$, $d=1,2,3,4,6,12$ in terms of $p$ and $k$. In \cite{xia}, all the above $(p,k)$-parameterizations were used to get the identity \eqref{21mixed} for all 
	the 21 mixed forms. Since we establish \eqref{21mixed} using the theory of modular forms, we use these identities along with the $(p,k)$-parameterizations obtained in   \cite[Theorems 1,2,4]{aaw2}, \cite[(2.3)]{aalw1} and \cite{aaw1} (for the generating functions ${\mathcal F}(\tau)$, 
	$\theta(\tau)$ and $\Psi(\tau)$) to derive the representations of $E_4(\tau)$ and its duplications in terms of $p$ and $k$. 
	Using the $(p,k)$-parametrisation for the left-hand side generating functions in \eqref{21mixed}, we get a system of equations involving 
	$E_4(d\tau)$, $d\vert 12$, $f_{4,6}(\ell\tau)$, $\ell =1,2$ and $f_{4,12}(\tau)$ and the functions $p$ and $k$, from which we derive the 
	required representations of $E_4(d\tau)$, $d\vert 12$, $f_{4,6}(\ell\tau)$, $\ell =1,2$ and $f_{4,12}(\tau)$ in terms of $p, k$, which we give below. (We have used Mathematica software for doing these computations.) 
	
	\begin{equation*}
		\begin{split}
			E_{4}(\tau) & =  (1 + 4 p + 64 p^{2} + 178 p^{3} + 235 p^{4} + 178 p^{5} + 64 p^{6} + 4 p^{7} + p^{8}       ) k^{4}, \\
			E_{4}(2\tau) & =  (1 + 4 p + 4 p^{2} + 28 p^{3} + 70 p^{4} + 28 p^{5} + 4 p^{6} + 4 p^{7} + p^{8}          ) k^{4}, \\
			E_{4}(3\tau) & =  (1 + 4 p + 4 p^{2} - 2 p^{3} + 10 p^{4} + 28 p^{5} + \frac{31}{4} p^{6} - \frac{29}{4} p^{7} + \frac{1}{16}p^{8}) k^{4}, \\
			E_{4}(4\tau) & =  (1 + 4 p + 4 p^{2} - 2 p^{3} - 5 p^{4} - 2 p^{5} + 4 p^{6} + 4 p^{7} + p^{8}             ) k^{4},  \\
			E_{4}(14\tau) & =  (1 + 4 p + 4 p^{2} - 2 p^{3} - 5 p^{4} - 2 p^{5} + \frac{1}{4} p^{6} + \frac{1}{4} p^{7} + \frac{1}{16}p^{8}         ) k^{4}, \\
			f_{4,6}(\tau) & = (-1 - 4 p -\frac{119}{32}  p^{2} + \frac{115}{32} p^{3} - \frac{913}{128} p^{4} - \frac{1695}{64} p^{5} - \frac{2049}{256} p^{6} + \frac{1801}{256} p^{7} -\frac{1}{16} p^{8}) k^{4},\\
			f_{4,6}(2\tau) & =(\frac{1}{2} + \frac{9}{4} p + \frac{175}{64}p^{2} -\frac{83}{64}  p^{3} +\frac{673}{256} p^{4} + \frac{1583}{128} p^{5} + \frac{2081}{512} p^{6} - \frac{1737}{512} p^{7} +\frac{1}{32} p^{8})k^{4},\\
			f_{4,12}(\tau) & =(\frac{1}{2} p + \frac{7}{4} p^{2} + \frac{7}{4} p^{3} - \frac{7}{4} p^{5} - \frac{7}{4} p^{6} -\frac{1}{2} p^{7})k^{4}. \\
		\end{split}
	\end{equation*}
	Now using \eqref{gh}, we obtain representations of $G(\tau)$ and $H(\tau)$ in terms of $p,k$, which is given below. \\
	
	\begin{equation*}
		\begin{split}
			G(\tau) & =   \frac{2 p^{3} + 5 p^{4} - 5 p^{6} - 2 p^{7}}{16} k^{4} ~=~  \frac{p^3(1-p)(1+ p)(1+2p)(2+p)}{16} k^{4},\\
			H(\tau) & = \frac{8 p + 28 p^{2} + 22 p^{3} - 15 p^{4} - 28 p^{5} - 13 p^{6} - 2 p^{7}}{16} k^{4}~=~ \frac{p(1-p)(1+p)(1+2p)(2+p)^3}{16} k^{4}. \\
			%& = \frac{p(1-p)(1+p)(1+2p)(2+p)^3}{16} k^{4}.
		\end{split}
	\end{equation*}
	
	\bigskip
	%\newpage

	\section{General formulas and examples}
	
	\subsection{General formulas}
	
	The main results in \S 1 give modular properties of the generating functions of the higher figurate numbers, triangular numbers with coefficients and the mixed forms.
	In particular, from \thmref{modularity}, we know that the generating function for the higher figurate number with coefficients (up to an integral power of $q$)  is a modular form in $M_{t}(\Gamma_{0}(N), \chi)$ for some integer weight $t$, level $N$ and nebentypus $\chi$ (which are explicitly given). Moreover, from Theorems \ref{modularity}, \ref{lt}, \ref{odd-triangular} and \corref{mixed-m} we know that the generating functions   
	for $T_{\mathcal C}({\bf z})$, ${\mathcal M}_{l,t}({\bf x}, {\bf z})$, ${\mathcal M}_{s,t}({\bf y}, {\bf z})$ and ${\mathcal M}({\bf x}, {\bf y}, 
	{\bf z})$ are all modular forms (upto a power of $q$) of integral weight for $\Gamma_0(N)$ with character $\chi$. The level $N$ 
	and character $\chi$ all depend on the coefficients chosen for the corresponding forms. Therefore, finding explicit basis for the space of modular forms of integral weight for some level and character is equivalent to getting explicit formulas for the corresponding representation 
	numbers. We will elaborate this a little bit. 
	Recall that the representation numbers corresponding to $T_{\mathcal C}({\bf z})$ and  ${\mathcal M}_{s,t}({\bf y}, {\bf z})$ are 
	denoted respectively by  $\delta_k(c_1^{e_1}, \ldots, c_k^{e_k};n)$ (= $\delta_k({\mathcal C};n)$) and ${\mathcal N}_{s,t}(b_1^{\alpha_1}, \ldots, b_ v^{\alpha_v}; c_1^{e_1}, \ldots, c_k^{e_k}; n)$. We now give notations to the remaining two forms. Let ${\mathcal N}_{l,t}(a_1^{\beta_1}, \ldots, a_ u^{\beta_u}; c_1^{e_1}, \ldots, c_k^{e_k}; n)$ denote the representation number for the form ${\mathcal M}_{l,t}({\bf x}, {\bf z})$ and ${\mathcal N}(a_1^{\beta_1}, \ldots, a_ u^{\beta_u}; b_1^{\alpha_1}, \ldots, b_ v^{\alpha_v};c_1^{e_1}, \ldots, c_k^{e_k}; n)$ denote the representation 
	number for the form ${\mathcal M}({\bf x}, {\bf y}, {\bf z})$ (with $v+k$ even). Now, let ${\mathcal G}(\tau)$ be one of the generating functions  $\Phi_a^{\mathcal C}(\tau)$ (for a fixed $a \ge 3$), ${\Psi}_{\mathcal C}(\tau)$, $\psi_{l,t}(\tau)$, $\psi_{s,t}(\tau)$ or ${\Phi}(\tau)$ (corresponding to the forms 
	$F_{a, {\mathcal C}}(\bf x)$, $T_{\mathcal C}({\bf z})$, ${\mathcal M}_{l,t}({\bf x}, {\bf z})$, ${\mathcal M}_{s,t}({\bf y}, {\bf z})$ or ${\mathcal M}({\bf x}, {\bf y}, {\bf z})$ respectively). Then by Theorems \ref{modularity}, \ref{lt}, \ref{odd-triangular} and \corref{mixed-m}, we see that $q^{\ell {\bf h}/8} 
	{\mathcal G}(\tau)$ (where $\ell = \frac{(a-2)^2}{a}, a \ge 1, a\not= 2$) belongs to the space $M_m(\Gamma_0(N),\tilde{\chi})$, where $2u+v+k =2m$ ($u, v\ge 0, k\ge 1$), 
	$$
	N = \begin{cases}
		2 a \cdot\lcm[c_1, c_2, \dots, c_k] & {\rm ~for~} \Phi_a^{\mathcal C}(\tau), {\rm if~} a\ge 3 {\rm~is~odd},\\
		a \cdot\lcm[c_1, c_2, \dots, c_k] & {\rm ~for~} \Phi_a^{\mathcal C}(\tau), {\rm if~} a {\rm ~is~ even},\\
		2\cdot\lcm[c_1, c_2, \dots, c_k] & {\rm ~for~} {\Psi}_{\mathcal C}(\tau),\\
		\lcm[3~\lcm[a_1, \ldots, a_u], 2~\lcm[c_1, \ldots, c_k]] & {\rm ~for~} \psi_{l,t}(\tau),\\
		\lcm[4~\lcm[b_1, \ldots, b_v], 2~\lcm[c_1, \ldots, c_k]] & {\rm ~for~} \psi_{s,t}(\tau),\\
		\lcm[3\lcm[a_1, \ldots, a_u], \lcm[4~\lcm[b_1, \ldots, b_v], 2~\lcm[c_1, \ldots, c_k]]]  & {\rm ~for~} \Phi(\tau)\\
	\end{cases}
	$$  
	and 
	$$
	\tilde{\chi} = \begin{cases}
		\chi&  {\rm ~for~}  \Phi_a^{\mathcal C}(\tau), ~a \ge 3\\
		\chi&  {\rm ~for~} {\Psi}_{\mathcal C}(\tau),\\
		\chi' &  {\rm ~for~} \psi_{l,t}(\tau),\\
		\chi'' & {\rm ~for~} \psi_{s,t}(\tau),\\
		\omega &{\rm ~for~} \Phi(\tau),\\
	\end{cases}
	$$
	where $\chi$, $\chi'$, $\chi''$ and $\omega$ are as in Theorems \ref{modularity}, \ref{lt}, \ref{odd-triangular} and \corref{mixed-m}, respectively and ${\bf h} = $ $(e_{1}c_{1} + \ldots +e_{k}c_{k}),$ and if some $e_{i}=0$ then the corresponding coefficient $c_{i}$ will not appear. If $\{f_1, \ldots, f_{\nu_{m,N,\tilde{\chi}}}\}$ is a basis for the space $M_m(\Gamma_0(N), \tilde{\chi})$, then we can express $q^{{\bf h}/8} {\mathcal G}(\tau)$ in terms of this basis as follows. 
	\begin{equation}\label{formula}
		q^{{\bf h}/8} {\mathcal G}(\tau) = \sum_{i=1}^{\nu_{m,N,\tilde{\chi}}} t_{i} f_i(\tau), ~t_i\in {\mathbb C}.
	\end{equation}
	Let $p = \ell {\bf h}/8$  (where $\ell = \frac{(a-2)^2}{a}, a \ge 1, a\not= 2$) and denote by $a_{f_i}(n)$, the $n^{\rm th}$ Fourier coefficient of the $i$-th basis element $f_i$. 
	Then comparing the $n^{\rm th}$ Fourier coefficients in both the sides of \eqref{formula}, we get 
	\begin{equation}\label{formula1}
		\sum_{i=1}^{\nu_{m,N,\tilde{\chi}}} t_{i} a_{f_i}(n) = \begin{cases}
			{\mathcal N}_{a}(c_1^{e_1}, \ldots, c_k^{e_k}; n-p) &  {\rm ~for~} \Phi_{a}(\tau) ,~ a \ge3, \\ 
			\delta_k(c_1^{e_1}, \ldots, c_k^{e_k};n-p) & {\rm ~for~} {\Psi}_{\mathcal C}(\tau),\\
			{\mathcal N}_{l,t}(a_1^{\beta_1}, \ldots, a_ u^{\beta_u}; c_1^{e_1}, \ldots, c_k^{e_k}; n-p) &  {\rm ~for~} \psi_{l,t}(\tau),\\
			{\mathcal N}_{s,t}(b_1^{\alpha_1}, \ldots, b_ v^{\alpha_v}; c_1^{e_1}, \ldots, c_k^{e_k}; n-p) &  {\rm ~for~} \psi_{s,t}(\tau),\\
			{\mathcal N}(a_1^{\beta_1}, \ldots, a_ u^{\beta_u}; b_1^{\alpha_1}, \ldots, b_ v^{\alpha_v};c_1^{e_1}, \ldots, c_k^{e_k}; n-p) &  {\rm ~for~} \Phi(\tau).\\
		\end{cases}
	\end{equation}
	
	\smallskip
	
	\subsection{Examples}
	
	In this section we give some sample formulas for all the cases considered in our work. In the case of higher figurate numbers, we have given examples corresponding to $a=3$. Bases of the spaces of modular forms for obtaining these formulas are given in the appendix. In the case of 4 and 6 variables, we have indicated the space of modular forms used to get these formulas. 
	
	\subsubsection{Sample formula : 4 variables}
	\begin{equation*}
		\begin{split}
			M_{2}(\Gamma_0(4), \chi_{0}): \hspace{2cm} &\\ 
			\delta_4(2^{4}              ;n-1)                   & =                  -  \sigma(n)  +   3  \sigma(n/2)  -2  \sigma(n/4),      \\ 
			M_{2}(\Gamma_0(6), \chi_{0}): \hspace{2cm} &\\
			\delta_4(1^{2} 3^{2}        ;n-1)                   & =                  - \sigma(n) + \sigma(n/2) + 3 \sigma(n/3) -3 \sigma(n/6),  \\ 
			M_{2}(\Gamma_0(8), \chi_{0}): \hspace{2cm} &\\ 
			{\mathcal N}_{s,t}(1^{2}       ; 4^{2}       ;n-1 ) & = -  \sigma(n)  -  \sigma(n/2)  +   10  \sigma(n/4)  -8  \sigma(n/8),    \\   
			{\mathcal N}_{s,t}(2^{2}       ; 4^{2}       ;n-1)  & = -  \sigma(n)  +   3  \sigma(n/2)  -2  \sigma(n/4),      \\   
			{\mathcal N}_{s,t}( 2^{1}      ;2^{2}4^{1}   ;n-1)  & = -  \sigma(n)  +   3  \sigma(n/2)  -2  \sigma(n/4),     \\ 
			M_{2}(\Gamma_0(8), \chi_{8}):  \hspace{2cm}&\\
			\delta_4( 1^{2} 2^{1} 4^{1} ;n-1)                   & =  \sigma_{1; \chi_{8},{\bf 1}}(n),  \\
			{\mathcal N}_{s,t}(1^{1} 2^{1} ; 4^{2}       ;n-1)  & =  \sigma_{1; \chi_{8},{\bf 1}}(n),  \\  
			{\mathcal N}_{s,t}(1^{1}       ; 2^{2}4^{1}  ;n-1)  & =  \sigma_{1; \chi_{8},{\bf 1}}(n),  \\  
			M_{2}(\Gamma_0(12), \chi_{0}):   \hspace{2cm}&\\
			{\mathcal N}_{l,t}(1^{1} ; 2^{1} 6^{1} ;n-1)        & = -  \sigma(n)    -   3  \sigma(n/2)     +   3  \sigma(n/3)    +   4  \sigma(n/4)     +   9  \sigma(n/6)    -   12  \sigma(n/12),    \\   
			{\mathcal N}_{l,t}(2^{1} ; 2^{1} 6^{1} ;n-1)        & = -  \sigma(n)    +   3  \sigma(n/2)     -   3  \sigma(n/3)    -   2  \sigma(n/4)    +   9  \sigma(n/6)    -   6  \sigma(n/12).    
		\end{split}
	\end{equation*}
	
	%\smallskip
	
	\subsubsection{Sample formula : 6 variables}
	\begin{equation*}
		\begin{split}
			M_{3}(\Gamma_0(4), \chi_{-4}):   \hspace{2cm}&\\
			\delta_6(1^4 2^2               ;n-1)  &   =  \sigma_{2;\chi_{-4},1}(n), \\
			{\mathcal N}_{s,t}(1^2         ;2^4                ;n-1)  & =   \sigma_{2;\chi_{-4},1}(n), \\
			M_{3}(\Gamma_0(6), \chi_{-3}):   \hspace{2cm}&\\
			\delta_6(1^{5} 3^{1}           ;n-1)  &  =  -\frac{1}{8}\sigma_{2;{\bf 1},\chi_{-3}}(n) + \frac{1}{8}\sigma_{2;{\bf 1},\chi_{-3}}(n/2) + \frac{9}{8}\sigma_{2;\chi_{-3},{\bf 1}}(n) + \frac{9}{8}\sigma_{2;\chi_{-3},{\bf 1}}(n/2),   \\
			\delta_6(1^{1} 3^{5}           ;n-1)  &  =  -\frac{1}{8}\sigma_{2;{\bf 1},\chi_{-3}}(n) + \frac{1}{8}\sigma_{2;{\bf 1},\chi_{-3}}(n/2) + \frac{1}{8}\sigma_{2;\chi_{-3},{\bf 1}}(n) + \frac{1}{8}\sigma_{2;\chi_{-3},{\bf 1}}(n/2),   \\
			{\mathcal N}_{l,t}(1^{1}      ;1^{2} 3^{2}      ;n-1)     & =  -\frac{1}{2}\sigma_{2;{\bf 1},\chi_{-3}}(n) + \frac{1}{2}\sigma_{2;{\bf 1},\chi_{-3}}(n/2) + \frac{3}{2}\sigma_{2;\chi_{-3},{\bf 1}}(n) + \frac{3}{2}\sigma_{2;\chi_{-3},{\bf 1}}(n/2),    \\
			{\mathcal N}_{l,t}(2^{1}      ;1^{2} 3^{2}      ;n-1)     & =  \frac{1}{4}\sigma_{2;{\bf 1},\chi_{-3}}(n)  - \frac{1}{4}\sigma_{2;{\bf 1},\chi_{-3}}(n/2) + \frac{3}{4}\sigma_{2;\chi_{-3},{\bf 1}}(n) + \frac{3}{4}\sigma_{2;\chi_{-3},{\bf 1}}(n/2),    
		\end{split}
	\end{equation*}
	
	\begin{equation*}
		\begin{split}
			&M_{3}(\Gamma_0(8), \chi_{-4}):   \hspace{14cm}\\
			&\hspace{2cm} \delta_6(2^4 4^2               ;n-2)    =   \sigma_{2;\chi_{-4},{\bf 1}}(n/2), \\
			&\hspace{2cm} \delta_6(4^6 ;n-3)    =  -\frac{1}{16}\sigma_{2;{\bf 1},\chi_{-4}}(n)  + \frac{1}{16}   \sigma_{2;{\bf 1},\chi_{-4}}(n/2)  + \frac{1}{16}    \sigma_{2;\chi_{-4},{\bf 1}}(n)  -\frac{1}{4}   \sigma_{2;\chi_{-4},{\bf 1}}(n/2),\\
			&\hspace{2cm} {\mathcal N}_{s,t}(1^{2} 2^{2} ;4^2                ;n-1)   =  \sigma_{2;\chi_{-4},{\bf 1}}(n),                                                                                                                  \\      
			&\hspace{2cm} {\mathcal N}_{s,t}(1^{2}       ;4^4                ;n-2)   =  -\frac{1}{4}\sigma_{2;{\bf 1},\chi_{-4}}(n)   + \frac{1}{4}  \sigma_{2;{\bf 1},\chi_{-4}}(n/2)    + \frac{1}{4}  \sigma_{2;\chi_{-4},{\bf 1}}(n), \\      
			&\hspace{2cm} {\mathcal N}_{s,t}(2^{2}       ;4^4                ;n-2)   =   \sigma_{2;\chi_{-4},{\bf 1}}(n/2),                                                \\                                                                     
			&\hspace{2cm} {\mathcal N}_{s,t}(2^{2}       ;2^4                ;n-1)   =   \sigma_{2;\chi_{-4},{\bf 1}}(n)  -4    \sigma_{2;\chi_{-4},{\bf 1}}(n/2),         \\                                                                     
			&\hspace{2cm} {\mathcal N}_{s,t}(1^{1} 2^{1} ;1^2 2^{1} 4^{1}    ;n-1)   =   \sigma_{2;\chi_{-4},{\bf 1}}(n),  \\  
			&\hspace{2cm} {\mathcal N}_{s,t}(1^{4}       ;4^2                ;n-1)   =  -\sigma_{2;{\bf 1},\chi_{-4}}(n)  +   \sigma_{2;{\bf 1},\chi_{-4}}(n/2)  + 2    \sigma_{2;\chi_{-4},{\bf 1}}(n),    \\
			&\hspace{2cm} {\mathcal N}_{s,t}(2^{4}       ;4^2                ;n-1)   =  \sigma_{2;{\bf 1},\chi_{-4}}(n)     \sigma_{2;{\bf 1},\chi_{-4}}(n/2)     + \sigma_{2;\chi_{-4},{\bf 1}}(n)   - 4  \sigma_{2;\chi_{-4},{\bf 1}}(n/2),          
		\end{split}
	\end{equation*}
	
	\begin{equation*}
		\begin{split}
			&M_{3}(\Gamma_0(8), \chi_{-8}):   \hspace{14cm}\\
			&\hspace{2cm} \delta_6(1^{2} 2^{1} 4^{3}     ;n-2)    =   \frac{1}{6} \sigma_{2;\chi_{-8},{\bf 1}}(n) - \frac{1}{6} \tau_{3,8,\chi_{-8}}(n),  \\                                     
			&\hspace{2cm} {\mathcal N}_{s,t}(1^{1} 2^{3} ;4^{2}              ;n-1)   =   \frac{2}{3} \sigma_{2;\chi_{-8},{\bf 1}}(n) + \frac{1}{3} \tau_{3,8,\chi_{-8}} (n),  \\
			&\hspace{2cm} {\mathcal N}_{s,t}(1^{3} 2^{1} ;4^{2}              ;n-1)   =   \frac{4}{3} \sigma_{2;\chi_{-8},{\bf 1}}(n) - \frac{1}{3} \tau_{3,8,\chi_{-8}} (n),  \\
			&\hspace{2cm} {\mathcal N}_{s,t}(1^{2}       ;1^{2} 2^{1} 4^{1}  ;n-1)   =   \frac{4}{3} \sigma_{2;\chi_{-8},{\bf 1}}(n) - \frac{1}{3} \tau_{3,8,\chi_{-8}} (n),  \\
			&\hspace{2cm} {\mathcal N}_{s,t}(2^{2}       ;1^{2} 2^{1} 4^{1}  ;n-1)  =   \frac{2}{3} \sigma_{2;\chi_{-8},{\bf 1}}(n) + \frac{1}{3} \tau_{3,8,\chi_{-8}} (n),  \\
			&\hspace{2cm} {\mathcal N}_{s,t}(1^{1} 2^{1} ;2^{4}              ;n-1)   =   \frac{2}{3} \sigma_{2;\chi_{-8},{\bf 1}}(n) + \frac{1}{3} \tau_{3,8,\chi_{-8}} (n),  \\
			&\hspace{2cm} {\mathcal N}_{s,t}(1^{1} 2^{1} ;4^{4}              ;n-2)   =   \frac{1}{6} \sigma_{2;\chi_{-8},{\bf 1}}(n) - \frac{1}{6} \tau_{3,8,\chi_{-8}}(n),  \\
			%M_{3}(8, \chi_{-8}): \qquad
			&\hspace{2cm} {\mathcal N}_{s,t}( 1^{3}; 2^{2} 4^{1}         ;n-1)       = \frac{4}{3} \sigma_{2;\chi_{-8},{\bf 1}}(n) - \frac{1}{3} \tau_{3,8,\chi_{-8}} (n), \\
			&\hspace{2cm} {\mathcal N}_{s,t}(1^{1} 2^{2}; 2^{2} 4^{1}    ;n-1)       = \frac{2}{3} \sigma_{2;\chi_{-8},{\bf 1}}(n) + \frac{1}{3} \tau_{3,8,\chi_{-8}} (n), 
		\end{split}
	\end{equation*}
	\begin{equation*}
		\begin{split}
			&M_{3}(\Gamma_0(12), \chi_{-3}):    \hspace{2cm}\\
			&{\mathcal N}_{l,s,t}(1^{1} ; 1^{1}3^{1}  ; 2^{1}6^{1} ;n-1)     = - \frac{1}{2}  \sigma_{2;{\bf 1 },\chi_{-3}}(n) +   \frac{1}{2}   \sigma_{2;{\bf 1 },\chi_{-3}}(n/2) +       \frac{3}{2}   \sigma_{2;\chi_{-3},{\bf 1 }}(n)  +   \frac{3}{2}   \sigma_{2;\chi_{-3},{\bf 1 }}(n/2),     \\ 
			& M_{3}(\Gamma_0(12), \chi_{-4}): \\
			& {\mathcal N}_{l,s,t}(1^{1} ; 1^{2}       ; 2^{1}6^{1} ;n-1)     =    \frac{16}{7}   \sigma_{2;\chi_{-4},{\bf 1 }}(n) +  \frac{72}{7}    \sigma_{2;\chi_{-4},{\bf 1 }}(n/3) - \frac{9}{7}   \tau_{3,12,\chi_{-4};1} (n) - \frac{30}{7}  \tau_{3,12,\chi_{-4};2} (n).  \\ 
		\end{split}
	\end{equation*}

%\smallskip
\subsubsection{Sample formulas for the forms  considered in the work of Xia-Ma-Tian \cite{xia}:} We use a basis of $M_{4}(\Gamma_{0}(12))$ to obtain these formulas. 
%By comparing the $n$-th Fourier coefficients of the expressions in \eqref{21mixed}, we obtain the following formulas: 
\begin{equation*}
\begin{split}
{\mathcal N}_{l,s}(2^{1}             ; 1^{3} 3^{3}                     ; n   ) &  = 2\sigma_{3}(n)  -  18\sigma_{3}(n/3) -  32\sigma_{3}(n/4)  +  288\sigma_{3}(n/12)   +  4a_{4,12}(n), \\
{\mathcal N}_{l,s,t}( 2^{1}           ; 1^{2} 3^{2} ; 2^{1} 6^{1}      ; n-1 ) &  = \frac{1}{2}\sigma_{3}(n) -    \frac{1}{2}\sigma_{3}(n/2) -    \frac{9}{2}\sigma_{3}(n/3)  +  \frac{9}{2}\sigma_{3}(n/6)    +  \frac{1}{2}a_{4,12}(n), \\
{\mathcal N}_{l,s,t}( 2^{1}           ; 1^{1} 3^{1} ; 2^{2} 6^{2}      ; n-2 ) &  = \frac{1}{8}\sigma_{3}(n) -    \frac{1}{8}\sigma_{3}(n/2) -    \frac{9}{8}\sigma_{3}(n/3)  +  \frac{9}{8}\sigma_{3}(n/6)    -    \frac{1}{8}a_{4,12}(n), \\
\end{split}
\end{equation*}
\begin{equation*}
\begin{split}
{\mathcal N}_{l,s,t}( 4^{1}           ; 1^{2}       ; 1^{1} 3^{1} 6^{2}; n-2 ) &  = \frac{1}{8}\sigma_{3}(n) -    \frac{1}{8}\sigma_{3}(n/2) -    \frac{9}{8}\sigma_{3}(n/3)  +  \frac{9}{8}\sigma_{3}(n/6)  -    \frac{1}{2}a_{4,6}(n)\\
& \quad  -  a_{4,6}(n/2) +  \frac{3}{8}a_{4,12}(n), \\
{\mathcal N}_{l,s,t}( 4^{1}           ; 3^{2}       ; 1^{1} 2^{2} 3^{1}; n-1 ) &  = \frac{1}{8}\sigma_{3}(n) -    \frac{1}{8}\sigma_{3}(n/2) -    \frac{9}{8}\sigma_{3}(n/3)  +  \frac{9}{8}\sigma_{3}(n/6)  +  \frac{1}{2}a_{4,6}(n)\\ 
& \quad +  a_{4,6}(n/2) +  \frac{3}{8}a_{4,12}(n), \\
{\mathcal N}_{l,s}( 1^{1} 2^{1}       ; 3^{4}                          ; n   ) &  = \frac{4}{5}\sigma_{3}(n) -    \frac{6}{5}\sigma_{3}(n/2) +  \frac{156}{5}\sigma_{3}(n/3) +  \frac{32}{5}\sigma_{3}(n/4) -   \frac{234}{5}\sigma_{3}(n/6) \\ 
& \quad +  \frac{1248}{5}\sigma_{3}(n/12)   + \frac{16}{5}a_{4,6}(n) +  \frac{32}{5}a_{4,6}(n/2) +  2a_{4,12}(n), \\
{\mathcal N}_{l,s}( 1^{1} 2^{1}       ; 1^{2} 3^{2}                    ; n   ) &  = 4\sigma_{3}(n) -  2\sigma_{3}(n/2) -  36\sigma_{3}(n/3) -  32\sigma_{3}(n/4) +  18\sigma_{3}(n/6) \\& \quad +  288\sigma_{3}(n/12)  +  6a_{4,12}(n), \\
\end{split}
\end{equation*}
\begin{equation*}
\begin{split}
{\mathcal N}_{l,s}( 1^{1} 2^{1}       ; 1^{4}                          ; n  )  &  =  \frac{52}{5}\sigma_{3}(n) -    \frac{78}{5}\sigma_{3}(n/2) +  \frac{108}{5}\sigma_{3}(n/3) +  \frac{416}{5}\sigma_{3}(n/4) -    \frac{162}{5}\sigma_{3}(n/6)  \\ 
& \quad + \frac{864}{5}\sigma_{3}(n/12) +   \frac{48}{5}a_{4,6}(n) +  \frac{96}{5}a_{4,6}(n/2) -  6a_{4,12}(n), \\
{\mathcal N}_{l,t}( 1^{1} 2^{1}                     ; 1^{2} 3^{2}      ; n-1 ) &  = \sigma_{3}(n) -  \sigma_{3}(n/2) -  9\sigma_{3}(n/3)  +  9\sigma_{3}(n/6),     \\
{\mathcal N}_{l,t}( 1^{1} 2^{1}                     ; 2^{2} 6^{2}      ; n-2 ) &  = \frac{3}{8}\sigma_{3}(n) -    \frac{19}{8}\sigma_{3}(n/2) -    \frac{27}{8}\sigma_{3}(n/3) +  2\sigma_{3}(n/4) +  \frac{171}{8}\sigma_{3}(n/6) \\ 
& \quad -  18\sigma_{3}(n/12)  -    \frac{3}{8}a_{4,12}(n), \\
{\mathcal N}_{l,s}( 2^{1} 4^{1}       ; 1^{2} 3^{2}                    ; n  )  &  =  \sigma_{3}(n) +  \sigma_{3}(n/2) -  9\sigma_{3}(n/3) -  32\sigma_{3}(n/4) -  9\sigma_{3}(n/6) \\& \quad +  288\sigma_{3}(n/12)   +  3a_{4,12}(n), \\
{\mathcal N}_{l,t}( 2^{1} 4^{1}                     ; 1^{2}.3^{2}      ; n-1)  &  =  \frac{1}{4}\sigma_{3}(n) -    \frac{1}{4}\sigma_{3}(n/2) -    \frac{9}{4}\sigma_{3}(n/3)  +  \frac{9}{4}\sigma_{3}(n/6)    +  \frac{3}{4}a_{4,12}(n), \\
{\mathcal N}_{l,s}( 2^{3}             ; 1^{1} 3^{1}                    ; n  )  &  =  2\sigma_{3}(n)  -  18\sigma_{3}(n/3) -  32\sigma_{3}(n/4)  +  288\sigma_{3}(n/12),    \\
\end{split}
\end{equation*}
\begin{equation*}
\begin{split}
{\mathcal N}_{l,t}( 2^{3}                           ; 2^{1} 6^{1}      ; n-1)  &  =  \sigma_{3}(n) -  9\sigma_{3}(n/2) -  9\sigma_{3}(n/3) +  8\sigma_{3}(n/4) +  81\sigma_{3}(n/6) -  72\sigma_{3}(n/12),    \\
{\mathcal N}_{l,s}( 4^{3}             ; 1^{1} 3^{1}                    ; n  )  &  =  \frac{1}{5}\sigma_{3}(n) -    \frac{9}{5}\sigma_{3}(n/2) +  \frac{9}{5}\sigma_{3}(n/3) +  \frac{128}{5}\sigma_{3}(n/4) -    \frac{81}{5}\sigma_{3}(n/6) \\ 
& \quad +  \frac{1152}{5}\sigma_{3}(n/12)  +  \frac{9}{5}a_{4,6}(n) +  \frac{18}{5}a_{4,6}(n/2),  \\
{\mathcal N}_{l,s}( 1^{2} 2^{1}       ; 1^{1} 3^{1}                    ; n  )  &  =  8\sigma_{3}(n) -  6\sigma_{3}(n/2) -  72\sigma_{3}(n/3) -  32\sigma_{3}(n/4) +  54\sigma_{3}(n/6) \\& \quad +  288\sigma_{3}(n/12)    +  6a_{4,12}(n), \\
{\mathcal N}_{l,t}( 1^{2} 2^{1}                     ; 2^{1} 6^{1}      ; n-1)  &  =  \frac{5}{2}\sigma_{3}(n) -    \frac{21}{2}\sigma_{3}(n/2) -    \frac{45}{2}\sigma_{3}(n/3) +  8\sigma_{3}(n/4) +  \frac{189}{2}\sigma_{3}(n/6)\\ 
& \quad  -  72\sigma_{3}(n/12)   -    \frac{3}{2}a_{4,12}(n), \\
{\mathcal N}_{l,s}( 1^{1} 2^{2}       ; 1^{1} 3^{1}                    ; n  )  &  =  \frac{16}{5}\sigma_{3}(n) -    \frac{24}{5}\sigma_{3}(n/2) +  \frac{144}{5}\sigma_{3}(n/3) +  \frac{128}{5}\sigma_{3}(n/4) -    \frac{216}{5}\sigma_{3}(n/6)\\ 
& \quad +  \frac{1152}{5}\sigma_{3}(n/12) +  \frac{24}{5}a_{4,6}(n) +  \frac{48}{5}a_{4,6}(n/2),  \\
\end{split}
\end{equation*}
\begin{equation*}
\begin{split}
{\mathcal N}_{l,s}( 2^{1} 4^{2}       ; 1^{1} 3^{1}                    ; n  )  &  =  \frac{1}{2}\sigma_{3}(n) +  \frac{3}{2}\sigma_{3}(n/2) -    \frac{9}{2}\sigma_{3}(n/3) -  32\sigma_{3}(n/4) -    \frac{27}{2}\sigma_{3}(n/6)\\
& \quad +  288\sigma_{3}(n/12)   +  \frac{3}{2}a_{4,12}(n), \\
{\mathcal N}_{l,s}(1^{1} 2^{1} 4^{1} ; 1^{1} 3^{1}                     ; n   ) &  = 2\sigma_{3}(n)  -  18\sigma_{3}(n/3) -  32\sigma_{3}(n/4)  +  288\sigma_{3}(n/12)   +  6a_{4,12}(n), \\
\end{split}
\end{equation*}
\begin{equation*}
\begin{split}
{\mathcal N}_{l,t}( 1^{1} 2^{1}.4^{1}                ; 2^{1} 6^{1}     ; n-1)  &  =  \frac{1}{4}\sigma_{3}(n) +  \frac{15}{4}\sigma_{3}(n/2) -    \frac{9}{4}\sigma_{3}(n/3) -  4\sigma_{3}(n/4) \\
& \quad -    \frac{135}{4}\sigma_{3}(n/6) +  36\sigma_{3}(n/12)   +  \frac{3}{4}a_{4,12}(n). \\
\end{split}
\end{equation*}
%}

\smallskip

\subsubsection{Sample formulas for higher figurate number for $a=3$:}
We use a basis of $M_{4}(\Gamma_0(24), \chi_{0})$ and $M_{4}(\Gamma_0(24), \chi_{8})$, respectively as given in \cite[\S 4]{rss-ijnt}, to obtain the following two formulas.
% are obtained using the basis of $M_{4}(24, \chi_{0})$ and $M_{4}(24, \chi_{8})$, respectively given in \cite[Sec 4]{rss-ijnt}. 
\begin{equation*}
	\begin{split}
		{\mathcal N}_{3}(2^4 4^4; n-1) & =  \frac{1}{10} \sigma_3(n)  -\frac{9}{10} \sigma_3(n/2)  -\frac{1}{10} \sigma_3(n/3) + \frac{8}{10} \sigma_3(n/4) + \frac{9}{10} \sigma_3(n/6) \\ \quad &   -\frac{8}{10} \sigma_3(n/12)  + 
		\frac{7}{30} a_{4,6}(n) + \frac{7}{15} a_{4,6}(n/2) + 0 a_{4,6}(n/4) \\ \quad & + \frac{1}{4} a_{4,8}(n) + \frac{7}{4} a_{4,8}(n/3) + \frac{1}{6} a_{4,12}(n) -  a_{4,12}(n/2) + \frac{1}{4} a_{4,24}(n),
	\end{split}
\end{equation*}
\begin{equation*}
	\begin{split}
		{\mathcal N}_{3}(1^{2} 2^{1} 4^{5}; n-1) & = \frac{32}{451} \sigma_{3, \chi_{8}, {\bf 1}}(n) +  \frac{32}{451} \sigma_{3, \chi_{8}, {\bf 1}}(n/2) + \frac{3401}{16236} a_{4,8, \chi_{8};1}(n) + \frac{910}{451} a_{4,8, \chi_{8};1}(n/3) \\  \quad & + \frac{4}{451}a_{4,8, \chi_{8};2}(n) + \frac{906}{451} a_{4,8, \chi_{8};2}(n/3) - \frac{283}{1476}a_{4,24, \chi_{8};1}(n)  - \frac{173}{492}a_{4,24, \chi_{8};2}(n) \\  \quad & + \frac{427}{492}a_{4,24, \chi_{8};3}(n) -\frac{28}{369}a_{4,24, \chi_{8};4}(n) + \frac{142}{41} a_{4,24, \chi_{8};5}(n) + \frac{17}{492} a_{4,24, \chi_{8};6}(n). \\
	\end{split}
\end{equation*}
The following formulas are obtained using the basis of modular forms of weight $5$ and level $24$ (as given in Table C in the appendix). 
\begin{equation*}
	\begin{split}
		{\mathcal N}_{3}(1^{2} 2^{5} 4^{3} ; n-1) & =  \frac{256}{6897}  \sigma_{4, \chi_{-8}, {\bf 1}}(n)  -\frac{256}{6897}    \sigma_{4, \chi_{-8}, {\bf 1}}(n/3)  -\frac{1}{110352}  \tau_{5,8, \chi_{-8};1'}(n) -\frac{1}{1452}    \tau_{5,8, \chi_{-8};1'}(n/3) \\ 
		& \quad  + \frac{7}{27588}   \tau_{5,8, \chi_{-8};2}(n) -\frac{32941}{27588}  \tau_{5,8, \chi_{-8};2}(n/3) +  \frac{173}{44}    \tau_{5,8, \chi_{-8};3}(n) -\frac{14647}{4598}   \tau_{5,8, \chi_{-8};3}(n/3)\\ 
		& \quad +\frac{116}{121}  \tau_{5,24, \chi_{-8};1}(n) -\frac{251}{484}   \tau_{5,24, \chi_{-8};2}(n) + \frac{3821}{1452}   \tau_{5,24, \chi_{-8};3}(n) -\frac{6517}{1452}  \tau_{5,24, \chi_{-8};4}(n) \\ 
		& \quad  +  \frac{1753}{484}   \tau_{5,24, \chi_{-8};5}(n) +  -\frac{1151}{132}   \tau_{5,24, \chi_{-8};6}(n) + \frac{9}{1936}   \tau_{5,24, \chi_{-8};7}(n) +\frac{16}{11}  \tau_{5,24, \chi_{-8};8}(n), \\ 
	\end{split}
\end{equation*}
\begin{equation*}
	\begin{split}
		{\mathcal N}_{3}(1^{4} 2^{2} 4^{4} ; n-1) & = \frac{16}{305}   \sigma_{4, \chi_{-4}, {\bf 1}}(n) +   \frac{16}{305}  \sigma_{4, \chi_{-4}, {\bf 1}}(n/3)  + \frac{4}{305}    \tau_{5,4,\chi_{-4}}(n)  +   4    \tau_{5,4,\chi_{-4}}(n/2) \\ 
		& \quad  +   \frac{5189}{305}    \tau_{5,4,\chi_{-4}}(n/3) -4    \tau_{5,4,\chi_{-4}}(n/6) -\frac{326}{183}   \tau_{5,12,\chi_{-4};1}(n)  +  \frac{57}{61}    \tau_{5,12,\chi_{-4};1}(n/2) \\ 
		& \quad  +   \frac{2}{61}    \tau_{5,12,\chi_{-4};2}(n)  +   \frac{347}{183}    \tau_{5,12,\chi_{-4};2}(n/2) +  \frac{88}{3}    \tau_{5,12,\chi_{-4};3}(n) +   8    \tau_{5,12,\chi_{-4};4}(n) \\ 
		& \quad -\frac{4}{3} \tau_{5,12,\chi_{-4};4}(n/2), \\
	\end{split}
\end{equation*}
\begin{equation*}
	\begin{split}
		{\mathcal N}_{3}(2^{8} 4^{2} ; n-1) & =   \frac{16}{305}   \sigma_{4, \chi_{-4}, {\bf 1}}(n)  -\frac{256}{305}   \sigma_{4, \chi_{-4}, {\bf 1}}(n/2)  +  \frac{16}{305}   \sigma_{4, \chi_{-4}, {\bf 1}}(n/3) -\frac{256}{305}   \sigma_{4, \chi_{-4}, {\bf 1}}(n/6) \\
		\quad &   +   \frac{4}{305}    \tau_{5,4,\chi_{-4}}(n)  +   \frac{1156}{305}    \tau_{5,4,\chi_{-4}}(n/2)  +   \frac{2444}{305}   \tau_{5,4,\chi_{-4}}(n/3)  -\frac{64}{305}    \tau_{5,4,\chi_{-4}}(n/6) \\
		\quad &  +   \frac{40}{183}    \tau_{5,12,\chi_{-4};1}(n)  +   \frac{57}{61}    \tau_{5,12,\chi_{-4};1}(n/2)  +   \frac{2}{61}    \tau_{5,12,\chi_{-4};2}(n)  +   \frac{164}{183}    \tau_{5,12,\chi_{-4};2}(n/2) \\
		\quad &  +   \frac{1312}{183}    \tau_{5,12,\chi_{-4};3}(n)  -\frac{241}{61}    \tau_{5,12,\chi_{-4};3}(n/2) +   \frac{456}{61}   \tau_{5,12,\chi_{-4};4}(n) -\frac{1}{183} \tau_{5,12,\chi_{-4};4}(n/2). \\  
	\end{split}
\end{equation*}

\bigskip

\smallskip

{\small 
	\noindent {\bf Acknowledgements:} 
	This work started when both the authors were at the Harish-Chandra Research Institute (HRI), Prayagraj. Some part of the 
	work was done when both the authors were visiting NISER, Bhubaneswar. We thank both the institutes for their support. 
%\newpage

\section*{Appendix-I}
%\subsection{Explicit examples}

% \newpage
\begin{center}
 \textbf{Explicit bases for the space $M_{k}(\Gamma_0(N),\chi)$, $k = 2,3, 5$; $N=6,8,12, 24$}
 \end{center}

\bigskip

Before we give the tables, we list below the notations used. The vector space $M_k(\Gamma_0(N),\chi)$ is decomposed as 
${\mathcal E}_(\Gamma_0(N),\chi)$, the space generated by the Eisenstein series and $S_k(\Gamma_0(N),\chi)$, the subspace of 
cusp forms. The dimensions of these subspaces are denoted by $e_{k,N,\chi}$ and $s_{k,N,\chi}$ respectively such that the dimension $\nu_{k,N,\chi}$ of the space $M_k(\Gamma_0(N),\chi)$ is the sum $e_{k,N,\chi} + s_{k,N,\chi}$. 

The basis elements of the space of cusp forms are given in terms of newforms and their duplications. We denote a newform of 
weight $k$, level $N$ and character $\chi$ by $\Delta_{k,N,\chi;j}(\tau)$, if there is more than one newform of a given weight, level and character. If it is unique, we drop the subscript $j$ and write simply as $\Delta_{k,N,\chi}(\tau)$. Further, if the character is the 
principal character, then we also drop the subscript $\chi$ from the notation. The $n$-th Fourier coefficient of the newform 
$\Delta_{k,N,\chi;j}(n)$ is denoted as $\tau_{k,N,\chi;j}(n)$. If the space of newforms is one dimensional, we drop the subscript $j$ and also drop the subscript $\chi$, if it is principal. 

For an integer $k \ge 2,$ let $E_k$ denote the normalized Eisenstein series of weight $k$ for the full modular group $SL_2({\mathbb Z})$, given by $\displaystyle{E_k(\tau) = 1 - \frac{2k}{B_k}\sum_{n\ge 1} \sigma_{k-1}(n) q^n}$, where $q=e^{2 i\pi \tau}$,$\tau \in {\mathbb H}$, $\sigma_r(n) = \sum_{d\vert n} d^{r}$ and $B_k$ is the $k^{\rm th}$ Bernoulli number defined by following identitty $\displaystyle{\frac{x}{e^x-1} = \sum_{m=0}^\infty \frac{B_m}{m!} x^m}$. 
 Apart from the Eisenstein series $E_k(\tau)$, we also use the following generalized Eisenstein series. Suppose that $\chi$ and $\psi$ are primitive Dirichlet characters with conductors $M$ and $N$, respectively. For a positive integer $k$,  let 
%\begin{equation}\label{eisenstein}
$\displaystyle{E_{k,\chi,\psi}(\tau) :=  - \frac{B_{k,\psi}}{2k} ~ \delta_{M,1} + \sum_{n\ge 1}\left(\sum_{d\vert n} \psi(d) \cdot \chi(n/d) d^{k-1}\right) q^n}$,
%\end{equation}
where $\delta_{m,n}$ is the Kronecker delta function and 
%$$
%c_0 = \begin{cases}
%0 &{\rm ~if~} M>1,\\
%- \frac{B_{k,\psi}}{2k} & {\rm ~if~} M=1,
%\end{cases}
%$$
$B_{k,\psi}$ is the  generalized Bernoulli number with respect to the character $\psi$. 
Then, the Eisenstein series $E_{k,\chi,\psi}(\tau)$ belongs to the space $M_k(\Gamma_0(MN), \chi \psi)$, provided $\chi(-1)\psi(-1) = (-1)^k$ and $MN\not=1$.  The $n$-th Fourier coefficient ($n\ge 1$) of the Eisenstein series $E_{k,\chi,\psi}(\tau)$ 
is denoted by $\sigma_{k-1;\chi,\psi}(n)$. 

An eta-quotient is defined as a finite product of integer powers of $\eta(\tau)$ where, \linebreak  $\eta(\tau)=q^{1/24} \prod_{n\ge1}(1-q^n)$ is the Dedekind eta function  and we denote it as follows. 
\linebreak
%\begin{equation}\label{eta-q} 
$\displaystyle{\delta_{1}^{r_{\delta_{1}}} \delta_{2}^{r_{\delta_{2}}} \cdots \delta_{s}^{r_{\delta_{s}}} := \prod_{i=1}^s \eta^{r_{\delta_{i}}}(\delta_{i} \tau )}$,
%\end{equation}
where $\delta_{i}$'s are positive integers and $r_{\delta_{i}}$'s are non-zero integers.
%%  For more details on the construction of these Eisenstein series, we refer to \cite{{miyake}, {stein}}. 
 
To get explicit formulas in the case of $4$ and $6$ variables, we use the following basis for the vector spaces $M_2(\Gamma_0(N),\chi)$ and 
$M_3(\Gamma_0(N),\chi)$. These are presented in two separate tables (Table $A$ and Table $B$) below. We get the explicit formulas in the case of higher figurate number (for $a =3,4$) by considering the $8$ and  $10$ variables. We give explicit basis for the vector space $M_5(\Gamma_0(N),\chi)$ in Table $C.$ The character $\chi_0$ denotes the principal character modulo the respective level $N$ and $\chi_m$ denotes the Kronecker symbol $\left(\frac{m}{\cdot}\right)$.

%\newpage

%\newpage
{%\small

\newpage
%\bigskip
The following eta-quotients are used to give explicit basis of modular spaces $M_{2}(\Gamma_0(N),\chi)$.
\smallskip
{\Large
\begin{center}
\begin{tabular}{ll}
$\Delta_{2,24,\chi_{0}}(\tau) = 2^{1}4^{1}6^{1}12^{1}$ ,  &\\
$\Delta_{2, 24, \chi_{8};1}(\tau) = 1^{1}2^{-1}3^{-1}6^{4}8^{2}12^{-1}$ , & $\Delta_{2, 24, \chi_{8};2}(\tau)   =  1^{2}4^{-1}6^{-1}8^{1}12^{4}24^{-1}$ ,  \\
$\Delta_{2, 24, \chi_{24};1}(\tau) =  1^{1}2^{-1}3^{-1}4^{1}6^{4}12^{-2}24^{2}$ , & $ \Delta_{2, 24, \chi_{24};2}(\tau)   =  1^{2}2^{-2}4^{4}6^{1}8^{-1}12^{-1}24^{1} .$    \\ 
 \end{tabular}
\end{center}
}

\bigskip
\begin{center}
\textbf{Table A. Basis for $M_{2}(\Gamma_0(N),\chi)$}\\
%, $d \vert 24$ and $\chi$ be dirichlet character modulo $d \vert 24$} \\

\bigskip
\begin{tabular}{|p{1.75 cm }|p{0.8 cm}|p{0.7 cm }|p{6cm}|p{4.75 cm }|}

\hline

{\textbf{ Space}} & \multicolumn{2}{|c|}{\textbf{Dimension}} & 
 \multicolumn{2}{|c|}{\textbf{Basis for  $M_2(\Gamma_0(N),\chi)$}}\\ \cline{2-5}
&&&&\\
 {$(N,\chi)$}     & $e_{2,N,\chi}$ & $s_{2,N,\chi}$  & {Basis for $ {\mathcal E}_2(\Gamma_0(N),\chi)$ }& {Basis for $S_2(\Gamma_0(N),\chi)$} \\
 &&&&\\
\hline 
&&&&\\
 $(6,\chi_{0})$   & 3   &  0   &  $\{ \phi_{1,b}, b\vert 6, b \neq 1\}$               &  \{0\}  \\
                           
\hline                     
&&&&\\                     
 $(8,\chi_{0})$   &  3  &  0   &   $\{ \phi_{1,b}, b|8, b \neq 1 \}$               & \{0\} \\
                           
\hline                     
&&&&\\                     
 $(8,\chi_{8})$   &  2  &  0   & $\{E_{2,{\bf 1},\chi_{8}}(\tau)$, $ E_{2,\chi_{8},{\bf 1}}(\tau)\}$ & \{0\} \\
                           
\hline                     
&&&&\\                     
 $(12,\chi_{0})$  &  5  &  0   & $\{ \phi_{1,b}, b|12, b \neq 1\}$   & \{0\}  \\
                                    
\hline                     
&&&&\\                     
 $(12,\chi_{12})$ &  4  &  0  & $\{E_{2, {\bf 1}, \chi_{12}}(\tau),  E_{2, \chi_{12}, {\bf 1}}(\tau), $ $ E_{2, \chi_{4}, \chi_{3}}(\tau),  E_{2, \chi_{3}, \chi_{4}}(\tau)\}$ & \{0\} \\

\hline                     
&&&&\\                     
 $(24,\chi_{0})$  &  7  &  1   & $\{ \phi_{1,b}, b|24, b \neq 1\}$ & $\{\Delta_{2, 24, \chi_{0}}(\tau) \}$ \\
                           
\hline                     
&&&&\\                     
 $(24,\chi_{8})$  &  4  &  2   & $\{E_{2, {\bf 1}, \chi_{8}}(az), a|3;$ $E_{2, \chi_{8}, {\bf 1}}(bz), b|3\}$ & $\{\Delta_{2, 24, \chi_{8};1}(\tau),  
 \Delta_{2, 24, \chi_{8};2}(\tau)\}$ \\
                           
 \hline                     
&&&&\\                     
 $(24,\chi_{12})$ &  8  &  0  & $\{E_{2, {\bf 1}, \chi_{12}}(az), a|2;  E_{2, \chi_{12}, {\bf 1}}(bz), b|2;$ $E_{2, \chi_{4}, \chi_{3}}(cz), c|2;  
 E_{2, \chi_{3}, \chi_{4}}(dz), d|2\}$ & \{0\} \\
                           
\hline                     
&&&&\\                     
 $(24,\chi_{24})$ &  4  &  2   & $\{E_{2, {\bf 1}, \chi_{24}}(\tau),  E_{2, \chi_{24}, {\bf 1}}(\tau), $ $ E_{2, \chi_{8}, \chi_{3}}(\tau),  E_{2, \chi_{3}, \chi_{8}}(\tau)\}$ & $\{\Delta_{2, 24, \chi_{24};1}(\tau), \Delta_{2, 24, \chi_{24};2}(\tau)\}$ \\
% &&&&\\
 
\hline 
\end{tabular}
\end{center}
}

\bigskip
The following eta-quotients are used to give explicit basis of respective modular spaces $M_{3}(\Gamma_0(N),\chi)$.
\smallskip
\textbf{
\begin{center}
\begin{tabular}{ll}
 $ \Delta_{3,8,\chi_{-8}}(\tau)        =  1^2 2^1 4^{1} 8^2 ,                               $ & $ \Delta_{3,12,\chi_{-3}}(\tau)       =  2^{3}6^{3} ,                                     $  \\
 $ \Delta_{3,12,\chi_{-4};1}(\tau)     =  1^{4}2^{-1}4^{1}6^{1}12^{1} ,                     $ & $ \Delta_{3,12,\chi_{-4};2}(\tau)     =  1^{1}2^{1}3^{1}6^{-1}12^{4},                     $  \\                                                                  
 $ \Delta_{3,24,\chi_{-3};1}(\tau)     =  1^{3}2^{1}3^{-1}4^{4}6^{1}8^{-3}24^{1} ,          $ & $ \Delta_{3,24,\chi_{-3};2}(\tau)     =  1^{-3}2^{4}3^{1}4^{1}8^{3}12^{1}24^{-1} ,        $  \\                      
 $ \Delta_{3, 24, \chi_{-8};1}(\tau)   =  1^{-2}2^{4}4^{4}6^{1}8^{-2}12^{1} ,               $ & $ \Delta_{3, 24, \chi_{-8};2}(\tau)   =  1^{2}4^{3}6^{3}8^{-1}12^{-2}24^{1},              $  \\
 $ \Delta_{3, 24, \chi_{-8};3}(\tau)   = 2^{3}3^{2}4^{-2}8^{1}12^{3}24^{-1},                $ & $ \Delta_{3, 24, \chi_{-8};4}(\tau)   =  1^{1}2^{1}3^{-1}4^{1}6^{2}8^{1}12^{2}24^{-1},    $  \\         
 $ \Delta_{3, 24, \chi_{-24};1}(\tau)  =  1^{-3}2^{9}3^{-1}4^{-3}6^{4}12^{-2}24^{2} ,       $ & $ \Delta_{3, 24, \chi_{-24};2}(\tau)   =  1^{-2}2^{8}6^{1}8^{-1}12^{-1}24^{1},            $  \\
 $ \Delta_{3, 24, \chi_{-24};3}(\tau)  =  1^{1}2^{-5}3^{-1}4^{11}6^{4}8^{-4}12^{-2}24^{2} , $ & $ \Delta_{3, 24, \chi_{-24};4}(\tau)   =  1^{2}2^{-6}4^{14}6^{1}8^{-5}12^{-1}24^{1} ,     $  \\
 $ \Delta_{3, 24, \chi_{-24};5}(\tau)  =  1^{1}2^{-1}3^{-5}4^{1}6^{14}12^{-6}24^{2},        $ & $ \Delta_{3, 24, \chi_{-24};6}(\tau)   =  1^{2}2^{-2}3^{-4}4^{4}6^{11}8^{-1}12^{-5}24^{1}.$  \\
 \end{tabular}
\end{center}
}
\newpage
%{\tiny
\begin{center}
\textbf{Table B. Basis for $M_{3}(\Gamma_0(N),\chi)$}
%, $d \vert 24$ and $\chi$ be dirichlet character modulo $d \vert 24$} \\

\bigskip

\begin{tabular}{|p{1.5 cm }|p{0.7 cm}|p{0.7 cm }|p{5.9cm}|p{5.9 cm }|}

\hline

{\textbf{Space}} & \multicolumn{2}{|c|}{\textbf{Dimension}} & 
 \multicolumn{2}{|c|}{\textbf{Basis for  $M_3(\Gamma_0(N),\chi)$}}\\ \cline{2-5}
 &&&&\\
 {$(N,\chi)$} & $e_{3,N,\chi}$ & $s_{3,N,\chi}$  & {Basis for $ {\mathcal E}_3(\Gamma_0(N),\chi)$ }& {Basis for $S_3(\Gamma_0(N),\chi)$}
  \\ \hline 
&&&&\\
 $(4,\chi_{-4})$    &   2 & 0       & $\{E_{3, {\bf 1}, \chi_{-4}}(\tau),  E_{3, \chi_{-4}, {\bf 1}}(\tau)\}$                   &  $\{0\}$  \\
\hline                                                                                                                   
&&&&\\                                                                                                                   
 $(3,\chi_{-3})$    &   2 & 0       & $\{E_{3, {\bf 1}, \chi_{-3}}(\tau),  E_{3, \chi_{-3}, {\bf 1}}(\tau)\}$                   & $\{0\}$ \\
                                                                                                                         
 \hline                                                                                                                  
 &&&&\\                                                                                                                  
 $(6,\chi_{-3})$    &   4 & 0       & $\{E_{3, {\bf 1}, \chi_{-3}}(az), a|2;   E_{3, \chi_{-3}, {\bf 1}}(bz), b|2\}$      & $\{0\}$ \\
\hline                                                                                                                   
&&&&\\                                                                                                                   
$(8,\chi_{-4})$     &   4 & 0       & $\{E_{3, {\bf 1}, \chi_{-4}}(az), a|2; $ $ E_{3, \chi_{-4}, {\bf 1}}(bz), b|2\}$    & $\{0\}$  \\
                                                                                                                         
\hline                              
&&&&\\                              
$(8,\chi_{-8})$     &   2 & 1       & $\{E_{3, {\bf 1}, \chi_{-8}}(\tau),$ $ E_{3, \chi_{-8}, {\bf 1}}(\tau)\}$                 & $\{\Delta_{3, 8, \chi_{-8}}(\tau)\}$ \\
                                    
\hline                              
&&&&\\                              
$(12,\chi_{-3})$    &   6 & 1       & $\{E_{3, {\bf 1}, \chi_{-3}}(az), a|4; $ $ E_{3, \chi_{-3}, {\bf 1}}(bz), b|4\}$    & $\{\Delta_{3, 12, \chi_{-3}}(\tau) \}$ \\
                                                                                                                          
\hline                                                                                                                    
&&&&\\                                                                                                                    
$(12,\chi_{-4})$    &   4 & 2       & $\{E_{3, {\bf 1}, \chi_{-4}}(az), a|3; $ $ E_{3, \chi_{-4}, {\bf 1}}(bz), b|3\}$    & $\{ \Delta_{3, 12, \chi_{-4};1}(\tau),\Delta_{3, 12, \chi_{-4};2}(\tau) \}$ \\
                                                                                                                          
\hline                                                                                                                    
&&&&\\                                                                                                                    
$(24,\chi_{-3})$    &   8 & 4       & $\{E_{3, {\bf 1}, \chi_{-3}}(az), a|8;$ $ E_{3, \chi_{-3}, {\bf 1}}(bz), b|8\}$    & $\{\Delta_{3, 12, \chi_{-3}}(az), a|2; \Delta_{3, 24, \chi_{-3};s}(\tau)$, \\
&&&& \qquad \quad $s = 1,2 \}$ \\
\hline                                                                                                                    
&&&&\\                                                                                                                    
$(24,\chi_{-4})$    &   8 & 4       & $\{E_{3, {\bf 1}, \chi_{-4}}(az), a|6;$ $  E_{3, \chi_{-4}, {\bf 1}}(bz), b|6\}$    & $\{ \Delta_{3, 12, \chi_{-4};1}(az), a|2$; 
$\Delta_{3, 12, \chi_{-4};2}(bz), b|2 \}$ \\
\hline                                                                                                                    
&&&&\\                                                                                                                    
$(24,\chi_{-8})$    &   4 & 6       & $\{E_{3, {\bf 1}, \chi_{-8}}(az); a|3,$ $  E_{3, \chi_{-8}, {\bf 1}}(bz); b|3\}$    & $\{\Delta_{3, 8, \chi_{-8}}(az), a|3; $ $   \Delta_{3, 24, \chi_{-8};s}(\tau); 1 \leq s \leq 4 \}$ \\
\hline                                                                                                                    
&&&&\\                              
$(24,\chi_{-24})$   &   4 & 6       & $\{E_{3, {\bf 1}, \chi_{-24}}(\tau),  E_{3, \chi_{-24}, {\bf 1}}(\tau), $ $ E_{3, \chi_{-3}, \chi_{8}}(\tau),  E_{3, \chi_{8}, \chi_{-3}}(\tau)\}$ & $\{\Delta_{3, 24, \chi_{-24};r}(\tau); 1 \leq r \leq 6 \}$ \\
\hline 
\end{tabular}
\end{center}

\bigskip

In the following table, we list bases for the spaces of modular forms $M_5(\Gamma_0(24), \chi)$ for the various levels and odd characters $\chi$. 

%\bigskip

%\newpage
% \begin{table}

%\caption{ Modular From Space : Dimension and Basis.}
 
\begin{center}

\textbf{Table C. Basis for $M_{5}(\Gamma_0(N),\chi)$} 

\bigskip

\begin{tabular}{|p{1.8 cm }|p{1 cm}|p{1 cm }|p{4.5cm}|p{6.5 cm }|}

\hline

{\textbf{Modular  Space}} & \multicolumn{2}{|c|}{\textbf{Dimension}} & 
 \multicolumn{2}{|c|}{\textbf{Basis for  $M_5(\Gamma_0(N),\chi)$}}\\ \cline{2-5}
 &&&&\\
 {$(N,\chi)$} & $e_{5,N,\chi}$ & $s_{5,N,\chi}$  & {Basis for $ {\mathcal E}_5(\Gamma_0(N),\chi)$ }& {Basis for $S_5(\Gamma_0(N),\chi)$}
  \\ \hline 
&&&&\\
 $(4,\chi_{-4})$  &2&1 & $
\{E_{5, {\bf 1}, \chi_{-4}}(z),  E_{5, \chi_{-4}, {\bf 1}}(z)\}
$ & $
\Delta_{5, 4, \chi_{-4}}(z)$ \\
\hline 
&&&&\\
 $(3,\chi_{-3})$  &  2 & 0   & $
\{E_{5, {\bf 1}, \chi_{-3}}(z),  E_{5, \chi_{-3}, {\bf 1}}(z)\}
$  & {0} \\
 
 \hline
 &&&&\\
 $(6,\chi_{-3})$  &  4 & 2   & $
\{E_{5, {\bf 1}, \chi_{-3}}(az); a|2,$ $  E_{5, \chi_{-3}, {\bf 1}}(bz); b|2\}
$ & $
\{\Delta_{5, 6, \chi_{-3};1}(z),$ $ \Delta_{5, 6, \chi_{-3};2}(z)\}
$ \\
\hline 
 $(8,\chi_{-4})$  &  4 & 2   & $
\{E_{5, {\bf 1}, \chi_{-4}}(az); a|2, $ $ E_{5, \chi_{-4}, {\bf 1}}(bz); b|2\}
$ & $
\{\Delta_{5, 4, \chi_{-4}}(az);a|2\}
$ \\

\hline
\end{tabular}
\end{center}

\begin{center}

\begin{tabular}{|p{1.8 cm }|p{1 cm}|p{1 cm }|p{4.5cm}|p{6.5 cm }|}

\hline

{\textbf{Modular  Space}} & \multicolumn{2}{|c|}{\textbf{Dimension}} & 
 \multicolumn{2}{|c|}{\textbf{Basis for  $M_5(\Gamma_0(N),\chi)$}}\\ \cline{2-5}
 &&&&\\
 {$(N,\chi)$} & $e_{5,N,\chi}$ & $s_{5,N,\chi}$  & {Basis for $ {\mathcal E}_5(\Gamma_0(N),\chi)$ }& {Basis for $S_5(\Gamma_0(N),\chi)$}
  \\ \hline 
  
 &&&&\\
 $(8,\chi_{-8})$  &  2 & 3    & $
\{E_{5, {\bf 1}, \chi_{-8}}(z),$ $ E_{5, \chi_{-8}, {\bf 1}}(z)\}
$ & $
\{\Delta_{5, 8, \chi_{-8};1}(z),$ $ \Delta_{5, 8, \chi_{-8};2}(z),\Delta_{5, 8, \chi_{-8};3}(z)\}
$ \\

 \hline
 &&&&\\
 $(12,\chi_{-3})$  &  6 &5    & $
\{E_{5, {\bf 1}, \chi_{-3}}(az); a|4, $ $ E_{5, \chi_{-3}, {\bf 1}}(bz); b|4\}
$ & $
\{\Delta_{5, 12, \chi_{-3};1}(z),\Delta_{5, 12, \chi_{-3};2}(z),$ $ \Delta_{5, 12, \chi_{-3};3}(z),\Delta_{5, 12, \chi_{-3};4}(z),F(z) \}
$ \\
 
 \hline 
 &&&&\\
 $(12,\chi_{-4})$  &  4 & 6   & $
\{E_{5, {\bf 1}, \chi_{-4}}(az); a|3, $ $ E_{5, \chi_{-4}, {\bf 1}}(bz); b|3\}
$ & $
\{\Delta_{5, 4, \chi_{-4}}(az);a|3 , $ $ \Delta_{5, 12, \chi_{-4};1}(z),\Delta_{5, 12, \chi_{-4};2}(z),$ $ \Delta_{5, 12, \chi_{-4};3}(z),\Delta_{5, 12, \chi_{-4};4}(z) \}
$ \\
 
 \hline 
 &&&&\\
 $(24,\chi_{-3})$  &  8 & 12   & $
\{E_{5, {\bf 1}, \chi_{-3}}(az); a|8, $ $ E_{5, \chi_{-3}, {\bf 1}}(bz); b|8\}
$ & $
\{\Delta_{5, 12, \chi_{-3};r}(z); 1 \leq r \leq 4 , $ $ \Delta_{5, 24, \chi_{-3};s}(z); 1 \leq s \leq 8 \}
$ \\
 \hline 
 &&&&\\
 $(24,\chi_{-4})$  &  8 & 12   & $
\{E_{5, {\bf 1}, \chi_{-4}}(az); a|6,$ $  E_{5, \chi_{-4}, {\bf 1}}(bz); b|6\}
$ & $
\{\Delta_{5, 4, \chi_{-4}}(tz);t|6 , \Delta_{5, 12, \chi_{-4};1}(az);a|2, $ $ \Delta_{5, 12, \chi_{-4};2}(bz);b|2, \Delta_{5, 12, \chi_{-4};3}(cz);c|2,
$ $ \Delta_{5, 12, \chi_{-4};4}(dz);d|2 \}
$ \\
 \hline 
 &&&&\\
 $(24,\chi_{-8})$  &  4 & 14   & $
\{E_{5, {\bf 1}, \chi_{-8}}(az); a|3,$ $  E_{5, \chi_{-8}, {\bf 1}}(bz); b|3\}
$ & $
\{\Delta_{5, 8, \chi_{-8};1'}(az);a|3 , \Delta_{5, 8, \chi_{-8};2}(bz);b|3, $ $ \Delta_{5, 8, \chi_{-8};3}(cz); c|3,  \Delta_{5, 24, \chi_{-8};s}(z); 1 \leq s \leq 8 \}
$ \\
 \hline 
 &&&&\\
 $(24,\chi_{-24})$  &  4 & 14   & $
\{E_{5, {\bf 1}, \chi_{-24}}(z),  E_{5, \chi_{-24}, {\bf 1}}(z), $ $ E_{5, \chi_{-3}, \chi_{8}}(z),  E_{5, \chi_{8}, \chi_{-3}}(z)\}
$ & $
\{\Delta_{5, 24, \chi_{-24};r}(z); 1 \leq r \leq 14 \}
$ \\
\hline 
\end{tabular}
\end{center}

\bigskip

where the following eta-quotients are used to give explicit basis of respective modular spaces $M_{5}(\Gamma_0(N),\chi)$.

\begin{equation*}
\begin{split}
 \Delta_{5,4,\chi_{-4};1}(z)   & =  1^4 2^2 4^4             , \quad           \Delta_{5,12,\chi_{-3};1}(z)   =  1^{-4}2^{9}3^{4}4^{-2}6^{5}12^{-2}                   \\       
 \Delta_{5,6,\chi_{-3};1}(z)   & =  1^{-1} 2^{8} 3^{3}      , \quad           \Delta_{5,12,\chi_{-3};2}(z)   =  1^{-4}2^{10}3^{4}4^{-1}6^{-2}12^{3}                  \\       
 \Delta_{5,6,\chi_{-3};2}(z)   & =  1^{3} 3^{-1} 6^{8}      , \quad           \Delta_{5,12,\chi_{-3};3}(z)   =  1^{-4}2^{12}3^{4}4^{-5}6^{-4}12^{7}                  \\       
 \Delta_{5,8,\chi_{-8};1}(z)   & =  1^2 2^9 4^{-3} 8^2      , \quad           \Delta_{5,12,\chi_{-3};4}(z)   =  2^{10}4^{-5}6^{-2}12^{7}                             \\       
 \Delta_{5,8,\chi_{-8};2}(z)   & =  1^{-2} 2^7 4^{7} 8^{-2} , \quad           \Delta_{5,12,\chi_{-3};5}(z)   = 2^{3}6^{3}* \phi_{1,2}                                \\
 \Delta_{5,8,\chi_{-8};3}(z)   & =  1^{2} 2^{-3} 4^{9} 8^{2}, \quad                                 \\
\end{split}
\end{equation*}

\begin{equation*}
\begin{split}
 \Delta_{5,12,\chi_{-4};1}(z)    &  =  1^{-2}2^{5}3^{2}4^{1}6^{3}12^{1}              , \quad  \Delta_{5,12,\chi_{-4};3}(z)    =  1^{2}2^{3}3^{-2}4^{1}6^{5}12^{1}                   \\  
 \Delta_{5,12,\chi_{-4};2}(z)    &  =  1^{1}2^{3}3^{1}4^{2}6^{5}12^{-2}              , \quad  \Delta_{5,12,\chi_{-4};4}(z)    =  1^{1}2^{5}3^{1}4^{-2}6^{3}12^{2}                   \\
 \Delta_{5, 24, \chi_{-3};1}(z)  &  =  1^{-2}2^{3}3^{2}4^{2}6^{1}24^{1}              , \quad  \Delta_{5, 24, \chi_{-3};6}(z)   =  1^{-2}2^{3}3^{2}4^{6}6^{5}8^{-2}12^{-4}24^{2}     \\
 \Delta_{5, 24, \chi_{-3};2}(z)  &  =  1^{-2}2^{3}3^{2}4^{5}6^{-3}8^{-3}12^{3}24^{5} , \quad  \Delta_{5, 24, \chi_{-3};7}(z)   =  1^{-2}2^{3}3^{2}6^{1}8^{4}12^{2}                  \\
 \Delta_{5, 24, \chi_{-3};3}(z)  &  =  1^{-2}2^{3}3^{2}4^{4}6^{1}8^{-4}12^{-2}24^{8} , \quad  \Delta_{5, 24, \chi_{-3};8}(z)   =  1^{-3}2^{1}3^{9}4^{1}6^{1}12^{1}                  \\
 \Delta_{5, 24, \chi_{-3};4}(z)  &  =  1^{-1}2^{3}3^{-1}4^{3}6^{1}8^{-4}12^{1}24^{8} , \quad  \Delta_{5,8,\chi_{-8};1'}(z)    =  1^{-6}2^{21}4^{-7}8^{2}                            \\ 
 \Delta_{5, 24, \chi_{-3};5}(z)  &  =  1^{-2}2^{5}3^{6}4^{4}6^{-3}                      \\
 \end{split}
 \end{equation*}

 \begin{equation*}
 \begin{split}
 \Delta_{5, 24, \chi_{-8};1}(z)   & =  1^{-2}2^{1}3^{4}6^{4}8^{3}12^{3}24^{-3}          , \quad   \Delta_{5, 24, \chi_{-8};5}(z)   =  1^{-2}2^{1}3^{8}4^{-1}6^{-2}8^{2}24^{4}            \\
 \Delta_{5, 24, \chi_{-8};2}(z)   & =  1^{-2}2^{1}3^{8}4^{7}6^{-6}8^{-2}12^{4}          , \quad   \Delta_{5, 24, \chi_{-8};6}(z)   =  1^{-2}2^{1}3^{8}4^{2}6^{-6}8^{-1}12^{3}24^{5}      \\
 \Delta_{5, 24, \chi_{-8};3}(z)   & =  1^{-2}2^{1}3^{8}4^{-3}6^{-2}8^{6}12^{2}          , \quad   \Delta_{5, 24, \chi_{-8};7}(z)   =  1^{-2}2^{3}3^{12}4^{1}6^{-6}8^{2}                  \\
 \Delta_{5, 24, \chi_{-8};4}(z)   & =  1^{-2}2^{1}3^{8}6^{-6}8^{3}12^{5}24^{1}          , \quad   \Delta_{5, 24, \chi_{-8};8}(z)   =  1^{-1}2^{1}3^{5}6^{-2}8^{-2}12^{1}24^{8}           \\
 \Delta_{5, 24, \chi_{-24};1}(z)  & =  1^{-4}2^{10}3^{-2}4^{-4}6^{7}8^{3}12^{3}24^{-3}  , \quad   \Delta_{5, 24, \chi_{-24};8}(z)   =  1^{-3}2^{10}3^{-1}4^{-4}6^{1}8^{-2}12^{1}24^{8}    \\
 \Delta_{5, 24, \chi_{-24};2}(z)  & =  1^{-4}2^{10}3^{2}4^{3}6^{-3}8^{-2}12^{4}         , \quad   \Delta_{5, 24, \chi_{-24};9}(z)   =  1^{-4}2^{12}3^{6}4^{-3}6^{-3}8^{2}                 \\
 \Delta_{5, 24, \chi_{-24};3}(z)  & =  1^{-4}2^{10}3^{2}4^{-7}6^{1}8^{6}12^{2}          , \quad   \Delta_{5, 24, \chi_{-24};10}(z)  =  1^{-4}2^{10}3^{2}4^{-1}6^{5}12^{-4}24^{2}          \\
 \Delta_{5, 24, \chi_{-24};4}(z)  & =  1^{-4}2^{10}3^{2}4^{-4}6^{-3}8^{3}12^{5}24^{1}   , \quad   \Delta_{5, 24, \chi_{-24};11}(z)  =  1^{1}2^{2}3^{-1}4^{3}6^{3}8^{-1}12^{2}24^{1}       \\
 \Delta_{5, 24, \chi_{-24};5}(z)  & =  1^{-4}2^{10}3^{2}4^{-5}6^{1}8^{2}24^{4}          , \quad   \Delta_{5, 24, \chi_{-24};12}(z)  =  1^{-2}2^{2}3^{4}4^{1}6^{1}12^{2}24^{2}             \\
 \Delta_{5, 24, \chi_{-24};6}(z)  & =  1^{-4}2^{10}3^{2}4^{-2}6^{-3}8^{-1}12^{3}24^{5}  , \quad   \Delta_{5, 24, \chi_{-24};13}(z)  =  2^{2}3^{6}4^{1}6^{-7}8^{2}12^{10}24^{-4}           \\
 \Delta_{5, 24, \chi_{-24};7}(z)  & =  1^{-4}2^{10}3^{2}4^{-3}6^{1}8^{-2}12^{-2}24^{8}  , \quad   \Delta_{5, 24, \chi_{-24};14}(z)  =   3^{2}4^{-3}6^{-3}8^{6}12^{12}24^{-4}              \\
\end{split}
\end{equation*}  

%\newpage
\section{Appendix-II}
\begin{center}
\textbf{List of forms considered for explicit formulas}
\end{center}
%\section{List of forms considered for explicit formulas}

In this section, we give a list of forms in $4$ and $6$ variables for the purpose of providing explicit examples for our main results. 
Since the number of examples is too big, we present a few sample formulas in each case and the other formulas are expressed in terms 
of linear combination of basis elements (of the respective vector space of modular forms) and the linear combination coefficients are listed in  
tabular format at the end of this paper. \\

\smallskip
%\newpage
\centerline{\bf {Table 1. Forms with 4 variables}}

\bigskip
Recall that $T_{\mathcal C}$ denotes the sum of triangular numbers with coefficients in ${\mathcal C}$, given by \eqref{triangular} and 
${\mathcal M}_{s,t}$, ${\mathcal M}_{l,t}$ are the mixed forms given by the equation \eqref{st-lt}. \\
%\smallskip
{\tiny
\begin{center}

%\textbf{Table: 4 variable forms} \\

\smallskip

\begin{tabular}{|p{1.5 cm }|p{2.7cm}|p{6.7 cm }|p{4.75cm}|}

\hline

{\textbf{Space}} & \multicolumn{3}{|c|}{\textbf{Form}} \\ \cline{2-4} 
&&&\\
& \textbf{$T_{\mathcal C}$:$(1^{e_1}2^{e_2}3^{e_3}4^{e_4}6^{e_5})$} & {\bf${\mathcal M}_{s,t}$:$(1^{j_1}2^{j_2}3^{j_1}6^{j_4};1^{e_1}2^{e_2}3^{e_1}4^{e_4}6^{e_5})$} & \textbf{${\mathcal M}_{l,t}$:$(1^{i_1}2^{i_2}4^{i_1}8^{i_4};1^{e_1}2^{e_2}3^{e_1}4^{e_4}6^{e_5})$}  \\
\hline
&&&\\
$M_{2}(4,\chi_{0})$& $(2^{4})$ & &  \\

\hline
&&&\\
 $M_{2}(6,\chi_{0})$& $(1^{2}3^{2})$ & &  \\

\hline
&&&\\
\hline

\end{tabular}

\end{center}
}
%\vfill
%\newpage

\bigskip

\centerline{\bf {Table 2(A). Forms with 6 variables}}

\smallskip

Apart from the three types as in Table I, here we also consider the mixed form ${\mathcal M}$ given by \eqref{mixed}, having all three types 
of forms. 

%\newpage
{\tiny
\begin{center}
%\textbf{Table : 6 variable forms \\}
%%\bigskip
\begin{tabular}{|c|c|c|}

\hline
{\textbf{Space}} & \multicolumn{2}{|c|}{\textbf{Form}} \\ \cline{2-3} 
&&\\
& {\bf $T_{\mathcal C}$:$(1^{e_1}2^{e_2}3^{e_3}4^{e_4}6^{e_5})$} & {\bf ${\mathcal M}_{l,t}$ :($1^{i_1}2^{i_2}4^{i_3}8^{i_4};1^{e_1}2^{e_2}3^{e_3}4^{e_4}6^{e_5}$)} \\
\hline 
&&\\                                                                                                                
$M_{3}(4,\chi_{-4})$ & $(1^{4} 2^{2})$ &   \\                                                                                                                
\hline                                                                                                                

&&\\                                                                                                                
$M_{3}(8,\chi_{-4})$ & $(2^{4} 4^{2}), (4^{6}) $ &   \\                                                                                                                
\hline                                                                                                                
&&\\
\hline

\end{tabular}
\end{center}

\bigskip

\centerline{\bf {Table 2(B). Forms with 6 variables}}

\smallskip

\begin{center}

\begin{tabularx}{\textwidth}{|p{1.5 cm}|p{8.1 cm}|p{7.2 cm}|}
%\caption{Table : 6 variable foms}\\
\hline
     
                            \multicolumn{1}{|c|}{\textbf{Space}}  & \multicolumn{2}{|c|}{\textbf{Form}} \\ \cline{2-3} 
& & \\                            
& \textbf{${\mathcal M}_{s,t}$ :($1^{i_1}2^{i_2}3^{i_3}6^{i_4};1^{e_1}2^{e_2}3^{e_3}4^{e_4}6^{e_5}$)} & \textbf {${\mathcal M}$ : ($1^{i_1}2^{i_2}4^{i_3}8^{i_4};1^{j_1}2^{j_2}3^{j_3}6^{j_4};1^{e_1}2^{e_2}3^{e_3}4^{e_4}6^{e_5}$)} \\
\hline 
     
\endfirsthead
%     \caption{Table 2  (continued\ldots)}\\
      \hline
                              \multicolumn{1}{|c|}{\textbf{Space}}  & \multicolumn{2}{|c|}{\textbf{Form}} \\ \cline{2-3}
& & \\                              
   & \textbf{${\mathcal M}_{s,t}$ :($1^{i_1}2^{i_2}3^{i_3}6^{i_4};1^{e_1}2^{e_2}3^{e_3}4^{e_4}6^{e_5}$)} & \textbf {${\mathcal M}$ : ($1^{i_1}2^{i_2}4^{i_3}8^{i_4};1^{j_1}2^{j_2}3^{j_3}6^{j_4};1^{e_1}2^{e_2}3^{e_3}4^{e_4}6^{e_5}$)} \\
\hline 
     
\endhead
      \hline
      \multicolumn{3}{|r|}{{Continued on Next Page\ldots}}     \\ \hline
      
\endfoot
    
\endlastfoot
\hline
&&\\
 $M_{3}(4,\chi_{-4}) $ &      $(1^{2}; 2^{4})$                                                                                                                                    &                                                                                                                              \\                                                                                                                                                                                                                                               
\hline                                                                                                                                                                                                                                                                                                     
&&\\                                                                                                                                                                                                                                                                                                      
 $M_{3}(8,\chi_{-4}) $ &      $(1^{4}; 4^{2}),(2^{4}; 4^{2}),(1^{2} 2^{2}; 4^{2}),(1^{2}; 4^{4}),(2^{2}; 4^{4}),(2^{2}; 2^{4}),$                                                  &                                                                                                                              \\
                       &      $(1^{1} 2^{1}  ; 1^{2} 2^{1} 4^{1}),(1^{2} 2^{1}; 2^{2} 4^{1}),(2^{3}; 2^{2} 4^{1})$                                                                &                                                                                                                              \\
\hline 
\end{tabularx}

\end{center}
}

{\small
\begin{center}
\centerline{\bf {Table 2(C). Higher figurate forms (when $a=4$)}}
%\textbf{Highe figurate froms when $a=4$ \\}
%%\bigskip
\begin{tabular}{|c|c|}

\hline
{\textbf{Space}} &\textbf{Form $\phi_{4}^{(1^{a}, 2^{b}, 3^{c}, 6^{d})}$} \\ 
\hline 
& \\
$M_{3}(24, \chi_{-3})$ &
$(1^{5}, 2^{0}, 3^{1}, 6^{0}),$  
$(1^{2}, 2^{1}, 3^{2}, 6^{1}),$  
$(1^{1}, 2^{0}, 3^{5}, 6^{0}),$ \\ 
& $(1^{0}, 2^{5}, 3^{0}, 6^{1}),$  
$(1^{0}, 2^{3}, 3^{0}, 6^{3}),$  
$(1^{0}, 2^{1}, 3^{0}, 6^{5})$ \\
\hline 
& \\
$M_{3}(24, \chi_{-4})$ & $(1^{4}, 2^{2}, 3^{0}, 6^{0}),$  
$(1^{4}, 2^{0}, 3^{0}, 6^{2}),$  
$(1^{1}, 2^{3}, 3^{1}, 6^{1}),$ \\ 
& $(1^{1}, 2^{1}, 3^{1}, 6^{3}),$  
$(1^{0}, 2^{2}, 3^{4}, 6^{0}),$  
$(1^{0}, 2^{0}, 3^{4}, 6^{2})$ \\
\hline
%\end{tabular}
%\end{center}
%\begin{center}
%%\textbf{Highe figurate froms when $a=4$ \\}
%%\bigskip
%\begin{tabular}{|c|c|}
%\hline
%{\textbf{Space}} &\textbf{Form} \\ 
\hline 
& \\
$M_{4}(24, \chi_{0})$ & $(1^{8}, 2^{0}, 3^{0}, 6^{0}),$  
$(1^{5}, 2^{1}, 3^{1}, 6^{1}),$  
$(1^{4}, 2^{0}, 3^{4}, 6^{0}),$  
$(1^{2}, 2^{4}, 3^{2}, 6^{0}),$ \\
& 
$(1^{2}, 2^{2}, 3^{2}, 6^{2}),$  
$(1^{2}, 2^{0}, 3^{2}, 6^{4}),$  
$(1^{1}, 2^{1}, 3^{5}, 6^{1}),$  
$(1^{0}, 2^{8}, 3^{0}, 6^{0}),$ \\
 & 
$(1^{0}, 2^{6}, 3^{0}, 6^{2}),$  
$(1^{0}, 2^{4}, 3^{0}, 6^{4}),$  
$(1^{0}, 2^{2}, 3^{0}, 6^{6}),$  
$(1^{0}, 2^{0}, 3^{8}, 6^{0})$ \\
\hline 
& \\
$M_{4}(24, \chi_{12})$ & $(1^{4}, 2^{3}, 3^{0}, 6^{1}),$  
$(1^{4}, 2^{1}, 3^{0}, 6^{3}),$  
$(1^{3}, 2^{2}, 3^{3}, 6^{0}),$  
$(1^{3}, 2^{0}, 3^{3}, 6^{2}),$
$(1^{1}, 2^{6}, 3^{1}, 6^{0}),$ \\
&   
$(1^{1}, 2^{4}, 3^{1}, 6^{2}),$  
$(1^{1}, 2^{2}, 3^{1}, 6^{4}),$ 
$(1^{1}, 2^{0}, 3^{1}, 6^{6}),$  
$(1^{0}, 2^{3}, 3^{4}, 6^{1}),$  
$(1^{0}, 2^{1}, 3^{4}, 6^{3})$  \\
\hline
%\end{tabular}
%\end{center}
%\begin{center}
%%\textbf{Highe figurate froms when $a=4$ \\}
%%\bigskip
%\begin{tabular}{|c|c|}
%\hline
%{\textbf{Space}} &\textbf{Form} \\ 
\hline 
& \\
$M_{5}(24, \chi_{-3})$ & $(1^{8},2^{1},3^{0},6^{1}),$ 
$(1^{7},2^{0},3^{3},6^{0}),$ 
$(1^{5},2^{4},3^{1},6^{0}),$ 
$(1^{5},2^{2},3^{1},6^{2}),$ 
$(1^{5},2^{0},3^{1},6^{4}),$ \\
&
$(1^{4},2^{1},3^{4},6^{1}),$ 
$(1^{3},2^{0},3^{7},6^{0}),$ 
$(1^{2},2^{5},3^{2},6^{1}),$ 
$(1^{2},2^{3},3^{2},6^{3}),$ 
$(1^{2},2^{1},3^{2},6^{5}),$ \\
 &
$(1^{1},2^{4},3^{5},6^{0}),$ 
$(1^{1},2^{2},3^{5},6^{2}),$ 
$(1^{1},2^{0},3^{5},6^{4}),$ 
$(1^{0},2^{9},3^{0},6^{1}),$ 
$(1^{0},2^{7},3^{0},6^{3}),$ \\
&
$(1^{0},2^{5},3^{0},6^{5}),$ 
$(1^{0},2^{3},3^{0},6^{7}),$ 
$(1^{0},2^{1},3^{8},6^{1}),$ 
$(1^{0},2^{1},3^{0},6^{9})$  \\
\hline 
& \\
$M_{5}(24, \chi_{-4})$ &  $(1^{6},2^{2},3^{2},6^{0}),$  
$(1^{6},2^{0},3^{2},6^{2}),$  
$(1^{4},2^{6},3^{0},6^{0}),$  
$(1^{4},2^{4},3^{0},6^{2}),$  
$(1^{4},2^{2},3^{0},6^{4}),$ \\
&
$(1^{4},2^{0},3^{0},6^{6}),$  
$(1^{3},2^{3},3^{3},6^{1}),$  
$(1^{3},2^{1},3^{3},6^{3}),$  
$(1^{2},2^{2},3^{6},6^{0}),$  
$(1^{2},2^{0},3^{6},6^{2}),$ \\
& 
$(1^{1},2^{7},3^{1},6^{1}),$  
$(1^{1},2^{5},3^{1},6^{3}),$  
$(1^{1},2^{3},3^{1},6^{5}),$  
$(1^{1},2^{1},3^{1},6^{7}),$  
$(1^{0},2^{6},3^{4},6^{0}),$ \\
&  
$(1^{0},2^{4},3^{4},6^{2}),$  
$(1^{0},2^{2},3^{4},6^{4}),$  
$(1^{0},2^{0},3^{4},6^{6})$ \\
\hline
\end{tabular}
\end{center}
}

\newpage
\section*{Appendix-III}

As mentioned in \S4.1, to get explicit formulas given by \eqref{formula1} for the list of forms in Tables 1 and 2, we use 
the following tables. These tables (Tables 3 to 14) give the required coefficients `$t_{i}$' that appear in the formulas \eqref{formula1}.

\smallskip

\subsection{Tabular format (4 variables):}

We provide the sample formulas for representation numbers in tables as follows.

\begin{center}

\renewcommand{\arraystretch}{1.2}
 %\tiny
%\begin{sidewaystable}[htb]
\begin{tabularx}{\textwidth}{|p{3.5cm}| p{0.8cm}|p{0.8cm}|p{0.8cm}|p{0.8cm}|p{0.8cm}|p{0.8cm}|p{0.8cm}|p{0.8cm}|p{0.8cm}|p{0.8cm}|p{0.8cm}|p{0.8cm}|}
\caption{Sample formulas for number of representations of a positive integer by the forms in $M_{2}(12,\chi_{0})$}\\ %for number of representations of Integer by Forms. $M_{2}(12,\chi_0)$}\\
\hline
     
                            \multicolumn{1}{|c|}{\textbf{Representation number}}  & \multicolumn{6}{|c|}{\textbf{Coefficient of linear span }}\\ \cline{2-7} 
    {} & {$t_1$} & {$t_2$} & {$t_3$} & {$t_4$} & {$t_5$} & {$t_6$} \\ \hline 
     
\endfirsthead
     \caption{Coefficient of linear span. (continued\ldots)}\\
      \hline
                              \multicolumn{1}{|c|}{\textbf{Representation number}}  & \multicolumn{6}{|c|}{\textbf{Coefficient of linear span }}\\ \cline{2-7} 
    {} & {$t_1$} & {$t_2$} & {$t_3$} & {$t_4$} & {$t_5$} & {$t_6$}  \\ \hline 
     
\endhead
      \hline
      \multicolumn{6}{|r|}{{Continued on Next Page\ldots}}     \\ \hline
      
\endfoot
    
\endlastfoot
\hline
$ {\mathcal N}_{s,t}(1^{1}3^{1}; 2^{1}6^{1}; n-1) $ & $  -\frac{1}{24} $  & $  \frac{1}{24} $  & $  \frac{1}{8} $  & $  0 $  & $  -\frac{1}{8} $  & $  0 $ \\ \hline
$ {\mathcal N}_{s,t}(1^{1}     ; 2^{1}3^{2}; n-1) $ & $  -\frac{1}{24} $  & $  \frac{1}{24} $  & $  \frac{1}{8} $  & $  0 $  & $  -\frac{1}{8} $  & $  0 $ \\ \hline
$ {\mathcal N}_{s,t}(3^{1}     ; 1^{2}6^{1}; n-1) $ & $  -\frac{1}{24} $  & $  \frac{1}{24} $  & $  \frac{1}{8} $  & $  0 $  & $  -\frac{1}{8} $  & $  0 $ \\ \hline
$ {\mathcal N}_{l,t}(1^{1}     ; 2^{1}6^{1}; n-1) $ & $  -\frac{1}{24} $  & $  -\frac{1}{8} $  & $  \frac{1}{8} $  & $  \frac{1}{6} $  & $  \frac{3}{8} $  & $  -\frac{1}{2} $ \\ \hline
$ {\mathcal N}_{l,t}(2^{1}     ; 2^{1}6^{1}; n-1) $ & $  -\frac{1}{24} $  & $  \frac{1}{8} $  & $  -\frac{1}{8} $  & $  -\frac{1}{12} $  & $  \frac{3}{8} $  & $  -\frac{1}{4} $ \\ \hline
$ {\mathcal N}_{l,t}(4^{1}     ; 2^{1}6^{1}; n-1) $ & $  -\frac{1}{24} $  & $  \frac{1}{8} $  & $  \frac{1}{8} $  & $  -\frac{1}{12} $  & $  -\frac{3}{8} $  & $  \frac{1}{4} $ \\
\hline

\end{tabularx}

\end{center}
\begin{center}
\renewcommand{\arraystretch}{1.2}
 %\tiny
%\begin{sidewaystable}[htb]
\begin{tabularx}{\textwidth}{|p{3.6cm}| p{0.8cm}|p{0.8cm}|p{0.8cm}|p{0.8cm}|}
\caption{Sample formulas for number of representations of a positive integer by the forms in $M_{2}(12,\chi_{12})$}\\ %for number of representations of Integer by Forms.. $M_{2}(12,x_{12})$}\\
\hline
     
                            \multicolumn{1}{|c|}{\textbf{Representation number}}  & \multicolumn{4}{|c|}{\textbf{Coefficient of linear span }}\\ \cline{2-5} 
    {} & {$t_1$} & {$t_2$} & {$t_3$} & {$t_4$} \\ \hline 
     
\endfirsthead
     \caption{Coefficient of linear span. (continued\ldots)}\\
      \hline
                              \multicolumn{1}{|c|}{\textbf{Representation number}}  & \multicolumn{4}{|c|}{\textbf{Coefficient of linear span}}\\ \cline{2-5} 
    {} & {$t_1$} & {$t_2$} & {$t_3$} & {$t_4$}  \\ \hline 
     
\endhead
      \hline
      \multicolumn{5}{|r|}{{Continued on Next Page\ldots}}     \\ \hline
      
\endfoot
    
\endlastfoot
\hline
$ \delta_{4}        (1^{1}2^{2}3^{1}   ;n-1)   $  &  $  0 $  &  $  \frac{3}{4} $  &  $  \frac{1}{4} $  &  $  0  $ \\ \hline
$ \delta_{4}        (1^{1}3^{1}6^{2}   ;n-2)   $  &  $  0 $  &  $  \frac{1}{4} $  &  $  -\frac{1}{4} $  &  $  0  $ \\ \hline
$ {\mathcal N}_{s,t}(1^{2}; 2^{1}6^{1} ;n-1)   $  &  $  0 $  &  $  \frac{3}{2} $  &  $  -\frac{1}{2} $  &  $  0  $ \\ \hline
$ {\mathcal N}_{s,t}(3^{2}; 2^{1}6^{1} ;n-1)   $  &  $  0 $  &  $  \frac{1}{2} $  &  $  \frac{1}{2} $  &  $  0  $ \\ \hline
$ {\mathcal N}_{s,t}(3^{1} ; 2^{1}3^{2};n-1)   $  &  $  0 $  &  $  \frac{1}{2} $  &  $  \frac{1}{2} $  &  $  0  $ \\ \hline
$ {\mathcal N}_{s,t}(1^{1} ; 1^{2}6^{1};n-1)   $  &  $  0 $  &  $  \frac{3}{2} $  &  $  -\frac{1}{2} $  &  $  0  $ \\ \hline

\end{tabularx}

\end{center}

\begin{center}
\renewcommand{\arraystretch}{1.2}
 %\tiny
%\begin{sidewaystable}[htb]
\begin{tabularx}{\textwidth}{|p{3.6cm}| p{0.8cm}|p{0.8cm}|p{0.8cm}|p{0.8cm}|p{0.8cm}|p{0.8cm}|p{0.8cm}|p{0.8cm}|p{0.8cm}|}
\caption{Sample formulas for number of representations of a positive integer by the forms in $M_{2}(24,\chi_{0})$}\\ %for number of representations of Integer by Forms.. $M_{2}(24,x_0)$}\\
\hline
     
                            \multicolumn{1}{|c|}{\textbf{Representation number}}  & \multicolumn{9}{|c|}{\textbf{Coefficient of linear span }}\\ \cline{2-10} 
    {} & {$t_1$} & {$t_2$} & {$t_3$} & {$t_4$} & {$t_5$} & {$t_6$} & {$t_7$} & {$t_8$} & {$t_9$} \\ \hline 
     
\endfirsthead
     \caption{Coefficient of linear span. (continued\ldots)}\\
      \hline
                              \multicolumn{1}{|c|}{\textbf{Representation number}}  & \multicolumn{9}{|c|}{\textbf{Coefficient of linear span }}\\ \cline{2-10} 
    {} & {$t_1$} & {$t_2$} & {$t_3$} & {$t_4$} & {$t_5$} & {$t_6$} & {$t_7$} & {$t_8$} & {$t_9$}  \\ \hline 
     
\endhead
      \hline
      \multicolumn{9}{|r|}{{Continued on Next Page\ldots}}     \\ \hline
      
\endfoot
    
\endlastfoot
\hline

$ {\mathcal N}_{s,t}(3^{2}      ; 4^{2}         ;n-1)      $ &  $  -\frac{1}{48}$  &  $  \frac{1}{16}$  &  $  \frac{1}{16}$  &  $  -\frac{5}{24}$  &  $  -\frac{3}{16}$  &  $  \frac{1}{6}$  &  $  \frac{5}{8}$  &  $  -\frac{1}{2}$  &  $  \frac{1}{2}$  \\ \hline 
$ {\mathcal N}_{s,t}(6^{2}      ; 4^{2}         ;n-1)      $ &  $  -\frac{1}{48}$  &  $  \frac{1}{16}$  &  $  \frac{1}{16}$  &  $  -\frac{1}{24}$  &  $  -\frac{3}{16}$  &  $  0$  &  $  \frac{1}{8}$  &  $  0$  &  $  \frac{1}{2}$  \\ \hline 
$ {\mathcal N}_{s,t}(2^{1}6^{1} ; 2^{1}6^{1}    ;n-1)      $ &  $  -\frac{1}{48}$  &  $  \frac{1}{16}$  &  $  -\frac{1}{16}$  &  $  -\frac{1}{24}$  &  $  \frac{3}{16}$  &  $  0$  &  $  -\frac{1}{8}$  &  $  0$  &  $  \frac{1}{2}$  \\ \hline 
$ {\mathcal N}_{s,t}(2^{1}      ;4^{1}6^{2}     ;n-2)      $ &  $  0$  &  $  -\frac{1}{24}$  &  $  0$  &  $  \frac{1}{24}$  &  $  \frac{1}{8}$  &  $  0$  &  $  -\frac{1}{8}$  &  $  0$  &  $  0$  \\ \hline 
$ {\mathcal N}_{s,t}(6^{1}      ;1^{1}3^{1}4^{1};n-1)      $ &  $  -\frac{1}{48}$  &  $  \frac{1}{48}$  &  $  \frac{1}{16}$  &  $  0$  &  $  -\frac{1}{16}$  &  $  0$  &  $  0$  &  $  0$  &  $  \frac{1}{2}$  \\ \hline 
$ {\mathcal N}_{l,t}(8^{1}      ; 2^{1}6^{1}    ;n-1)      $ &  $  -\frac{1}{96}$  &  $  \frac{1}{32}$  &  $  -\frac{1}{32}$  &  $  -\frac{1}{48}$  &  $  \frac{3}{32}$  &  $  0$  &  $  -\frac{1}{16}$  &  $  0$  &  $  \frac{3}{4}$  \\
\hline 

\end{tabularx}

\end{center}
\begin{center}

\renewcommand{\arraystretch}{1.2}
 %\tiny
%\begin{sidewaystable}[htb]
\begin{tabularx}{\textwidth}{|p{3.6cm}| p{0.8cm}|p{0.8cm}|p{0.8cm}|p{0.8cm}|p{0.8cm}|p{0.8cm}|}
\caption{Sample formulas for number of representations of a positive integer by the forms in $M_{2}(24,\chi_{8})$}\\ %for number of representations of Integer by Forms.. $M_{2}(24,x_8)$}\\
\hline
     
                            \multicolumn{1}{|c|}{\textbf{Representation number}}  & \multicolumn{6}{|c|}{\textbf{Coefficient of linear span }}\\ \cline{2-7} 
    {} & {$t_1$} & {$t_2$} & {$t_3$} & {$t_4$} & {$t_5$} & {$t_6$} \\ \hline 
     
\endfirsthead
     \caption{Coefficient of linear span. (continued\ldots)}\\
      \hline
                              \multicolumn{1}{|c|}{\textbf{Representation number}}  & \multicolumn{6}{|c|}{\textbf{Coefficient of linear span }}\\ \cline{2-7} 
    {} & {$t_1$} & {$t_2$} & {$t_3$} & {$t_4$} & {$t_5$} & {$t_6$}  \\ \hline 
     
\endhead
      \hline
      \multicolumn{6}{|r|}{{Continued on Next Page\ldots}}     \\ \hline
      
\endfoot
    
\endlastfoot
\hline
$ {\mathcal N}_{s,t}(3^{1}6^{1} ; 4^{2}     ;n-1)  $ & $  0 $ & $  0 $ & $  \frac{2}{5} $ & $  -\frac{3}{5} $ & $  \frac{2}{5} $ & $  \frac{1}{5} $ \\ \hline
$ {\mathcal N}_{s,t}(1^{1}6^{1} ; 2^{1}6^{1};n-1)  $ & $  0 $ & $  0 $ & $  \frac{4}{5} $ & $  -\frac{6}{5} $ & $  \frac{4}{5} $ & $  -\frac{3}{5} $ \\ \hline
$ {\mathcal N}_{s,t}(2^{1}3^{1} ; 2^{1}6^{1};n-1)  $ & $  0 $ & $  0 $ & $  \frac{2}{5} $ & $  \frac{12}{5} $ & $  \frac{2}{5} $ & $  \frac{1}{5} $ \\ \hline
$ {\mathcal N}_{s,t}(1^{1} ; 4^{1}6^{2}     ;n-2)  $ & $  0 $ & $  0 $ & $  \frac{1}{5} $ & $  \frac{6}{5} $ & $  \frac{1}{5} $ & $  -\frac{2}{5} $ \\ \hline
$ {\mathcal N}_{s,t}(3^{1} ; 1^{1}3^{1}4^{1};n-1)  $ & $  0 $ & $  0 $ & $  \frac{4}{5} $ & $  -\frac{6}{5} $ & $  -\frac{1}{5} $ & $  \frac{2}{5} $ \\ \hline
$ {\mathcal N}_{s,t}(2^{1} ; 2^{1}3^{2}     ;n-1)  $ & $  0 $ & $  0 $ & $  \frac{2}{5} $ & $  \frac{12}{5} $ & $  \frac{2}{5} $ & $  \frac{1}{5} $ \\ \hline
$ {\mathcal N}_{s,t}(6^{1} ; 1^{2}6^{1}     ;n-1)  $ & $  0 $ & $  0 $ & $  \frac{4}{5} $ & $  -\frac{6}{5} $ & $  \frac{4}{5} $ & $  -\frac{3}{5} $ \\
\hline

\end{tabularx}

\end{center}
\begin{center}

\renewcommand{\arraystretch}{1.2}
 %\tiny
%\begin{sidewaystable}[htb]
\begin{tabularx}{\textwidth}{|p{3.6cm}| p{0.8cm}|p{0.8cm}|p{0.8cm}|p{0.8cm}|p{0.8cm}|p{0.8cm}|p{0.8cm}|p{0.8cm}|}
\caption{Sample formulas for number of representations of a positive integer by the forms in $M_{2}(24,\chi_{12})$}\\ %for number of representations of Integer by Forms.. $M_{2}(24,x_{12})$}\\
\hline
     
                            \multicolumn{1}{|c|}{\textbf{Representation number}}  & \multicolumn{8}{|c|}{\textbf{Coefficient of linear span }}\\ \cline{2-9} 
    {} & {$t_1$} & {$t_2$} & {$t_3$} & {$t_4$} & {$t_5$} & {$t_6$} & {$t_7$} & {$t_8$} \\ \hline 
     
\endfirsthead
     \caption{Coefficient of linear span. (continued\ldots)}\\
      \hline
                              \multicolumn{1}{|c|}{\textbf{Representation number}}  & \multicolumn{8}{|c|}{\textbf{Coefficient of linear span }}\\ \cline{2-9} 
    {} & {$t_1$} & {$t_2$} & {$t_3$} & {$t_4$} & {$t_5$} & {$t_6$} & {$t_7$} & {$t_8$}  \\ \hline 
     
\endhead
      \hline
      \multicolumn{9}{|r|}{{Continued on Next Page\ldots}}     \\ \hline
      
\endfoot
    
\endlastfoot
\hline

$ {\mathcal N}_{s,t}(1^{1}3^{1} ; 4^{2}     ; n-1)        $ & $  -\frac{1}{4} $ & $   \frac{1}{4} $ & $   \frac{3}{4} $ & $   0 $ & $   -\frac{1}{4} $ & $   0 $ & $   \frac{3}{4} $ & $   \frac{3}{4} $ \\  \hline 
$ {\mathcal N}_{s,t}(2^{1}6^{1} ; 4^{2}     ; n-1)        $ & $  0 $ & $   0 $ & $   \frac{3}{4} $ & $   -\frac{3}{2} $ & $   \frac{1}{4} $ & $   \frac{1}{2} $ & $   0 $ & $   0 $ \\  \hline 
$ {\mathcal N}_{s,t}(2^{2} ; 2^{1}6^{1}     ; n-1)        $ & $  0 $ & $   0 $ & $   \frac{3}{2} $ & $   -3 $ & $   -\frac{1}{2} $ & $   -1 $ & $   0 $ & $   0 $ \\  \hline 
$ {\mathcal N}_{s,t}(6^{2} ; 2^{1}6^{1}     ; n-1)        $ & $  0 $ & $   0 $ & $   \frac{1}{2} $ & $   -1 $ & $   \frac{1}{2} $ & $   1 $ & $   0 $ & $   0 $ \\  \hline 
$ {\mathcal N}_{s,t}(2^{1} ; 1^{1}3^{1}4^{1}; n-1)        $ & $  0 $ & $   0 $ & $   \frac{3}{4} $ & $   0 $ & $   \frac{1}{4} $ & $   0 $ & $   0 $ & $   0 $ \\  \hline 
$ {\mathcal N}_{s,t}(6^{1} ; 2^{2}4^{1}     ; n-1)        $ & $  0 $ & $   0 $ & $   \frac{3}{4} $ & $   -\frac{3}{2} $ & $   \frac{1}{4} $ & $   \frac{1}{2} $ & $   0 $ & $   0 $ \\  \hline 
$ {\mathcal N}_{s,t}(6^{1} ; 4^{1}6^{2}     ; n-2)        $ & $  0 $ & $   0 $ & $   0 $ & $   \frac{1}{2} $ & $   0 $ & $   \frac{1}{2} $ & $   0 $ & $   0 $ \\  \hline 
$ {\mathcal N}_{l,t}(1^{1} ; 4^{2}          ; n-1)        $ & $  -\frac{1}{4} $ & $   \frac{1}{4} $ & $   \frac{3}{4} $ & $   3 $ & $   -\frac{1}{4} $ & $   1 $ & $   \frac{3}{4} $ & $   \frac{3}{4} $ \\
\hline 

\end{tabularx}

\end{center}
\begin{center}

\renewcommand{\arraystretch}{1.2}
 %\tiny
%\begin{sidewaystable}[htb]
\begin{tabularx}{\textwidth}{|p{3.6cm}| p{0.8cm}|p{0.8cm}|p{0.8cm}|p{0.8cm}|p{0.8cm}|p{0.8cm}|}
\caption{Sample formulas for number of representations of a positive integer by the forms in $M_{2}(24,\chi_{24})$}\\ %for number of representations of Integer by Forms.. $M_{2}(24,x_{24})$}\\
\hline
     
                            \multicolumn{1}{|c|}{\textbf{Representation number}}  & \multicolumn{6}{|c|}{\textbf{Coefficient of linear span }}\\ \cline{2-7} 
    {} & {$t_1$} & {$t_2$} & {$t_3$} & {$t_4$} & {$t_5$} & {$t_6$} \\ \hline 
     
\endfirsthead
     \caption{Coefficient of linear span. (continued\ldots)}\\
      \hline
                              \multicolumn{1}{|c|}{\textbf{Representation number}}  & \multicolumn{6}{|c|}{\textbf{Coefficient of linear span }}\\ \cline{2-7} 
    {} & {$t_1$} & {$t_2$} & {$t_3$} & {$t_4$} & {$t_5$} & {$t_6$}  \\ \hline 
     
\endhead
      \hline
      \multicolumn{7}{|r|}{{Continued on Next Page\ldots}}     \\ \hline
      
\endfoot
    
\endlastfoot
\hline

$ \delta_{4}         (3^{2}4^{1}6^{1}        ; n-2)          $     &   $ 0 $ & $  \frac{1}{6} $ & $  -\frac{1}{6} $ & $  0 $ & $  \frac{1}{3} $ & $  0 $  \\  \hline
$ {\mathcal N}_{s,t} (1^{1}6^{1} ; 4^{2}     ; n-1)       $     &   $ 0 $ & $  \frac{1}{2} $ & $  -\frac{1}{6} $ & $  0 $ & $  2 $ & $  \frac{2}{3} $  \\  \hline
$ {\mathcal N}_{s,t} (2^{1}3^{1} ; 4^{2}     ; n-1)       $     &   $ 0 $ & $  \frac{1}{2} $ & $  \frac{1}{6} $ & $  0 $ & $  0 $ & $  \frac{1}{3} $  \\  \hline
$ {\mathcal N}_{s,t} (1^{1}2^{1} ; 2^{1}6^{1}; n-1)  $     &   $ 0 $ & $  1 $ & $  \frac{1}{3} $ & $  0 $ & $  0 $ & $  -\frac{1}{3} $  \\  \hline
$ {\mathcal N}_{s,t} (3^{1}6^{1} ; 2^{1}6^{1}; n-1)  $     &   $ 0 $ & $  \frac{1}{3} $ & $  \frac{1}{3} $ & $  0 $ & $  \frac{2}{3} $ & $  \frac{1}{3} $  \\  \hline
$ {\mathcal N}_{s,t} (1^{1} ; 1^{1}3^{1}4^{1}; n-1)  $     &   $ 0 $ & $  1 $ & $  -\frac{1}{3} $ & $  0 $ & $  1 $ & $  \frac{1}{3} $  \\  \hline
$ {\mathcal N}_{s,t} (2^{1} ; 1^{2}6^{1}     ; n-1)       $     &   $ 0 $ & $  1 $ & $  \frac{1}{3} $ & $  0 $ & $  0 $ & $  -\frac{1}{3} $  \\  \hline
$ {\mathcal N}_{s,t} (3^{1} ; 4^{1}6^{2}     ; n-2)       $     &   $ 0 $ & $  \frac{1}{6} $ & $  -\frac{1}{6} $ & $  0 $ & $  \frac{1}{3} $ & $  0 $  \\
\hline

\end{tabularx}
\end{center}

\subsection{Tabular format (6 variables)}

Here, we present  the formulas for the number of  representations  of a positive integer by the forms of weight $3$ and level $12$ with some character. 

%\tiny

\begin{center}

 \renewcommand{\arraystretch}{1.2}
 %\tiny
%\begin{sidewaystable}[htb]
\begin{tabularx}{\textwidth}{|p{4cm}| p{0.8cm}|p{0.8cm}|p{0.8cm}|p{0.8cm}|p{0.8cm}|p{0.8cm}|p{0.8cm}|}
\caption{Sample formulas for number of representations of a positive integer by the forms in $M_{3}(12,\chi_{-3})$}\\ %for number of representations of Integer by Forms.. $M_{3}(12,x_{-3})$.}\\
\hline
     
                            \multicolumn{1}{|c|}{\textbf{Representation number}}  & \multicolumn{7}{|c|}{\textbf{Coefficient of linear span }}\\ \cline{2-8} 
    {} & {$t_1$} & {$t_2$} & {$t_3$} & {$t_4$} & {$t_5$} & {$t_6$} & {$t_7$}  \\ \hline 
     
\endfirsthead
     \caption{coefficient in linear span. (continued\ldots)}\\
      \hline
                              \multicolumn{1}{|c|}{\textbf{Representation number}}  & \multicolumn{7}{|c|}{\textbf{Coefficient of linear span }}\\ \cline{2-8} 
    {} & {$t_1$} & {$t_2$} & {$t_3$} & {$t_4$} & {$t_5$} & {$t_6$} & {$t_7$}   \\ \hline 
     
\endhead
      \hline
      \multicolumn{8}{|r|}{{Continued on Next Page\ldots}}     \\ \hline
      
\endfoot
    
\endlastfoot
\hline

$\delta_6(1^{2}2^{1}3^{2}6^{1};n-2)                     $  & $\frac{-1}{16}$  &   $\frac{1}{16}$  &   $  0 $  &   $\frac{3}{16}$  &   $\frac{3}{16}$  &   $  0 $  &   $\frac{-1}{8}$\\ \hline
$\delta_6(2^{5}6^{1}          ;n-2)                     $  & $  0 $  &   $\frac{-1}{8}$  &   $\frac{1}{8}$  &   $  0 $  &   $\frac{9}{8}$  &   $\frac{9}{8}$  &   $  0 $\\ \hline
$\delta_6(2^{1}6^{5}          ;n-4)                     $  & $  0 $  &   $\frac{-1}{8}$  &   $\frac{1}{8}$  &   $  0 $  &   $\frac{1}{8}$  &   $\frac{1}{8}$  &   $  0 $\\ \hline
$\delta_6(2^{3}6^{3}          ;n-3)                     $  & $\frac{-1}{32}$  &   $\frac{-3}{32}$  &   $\frac{1}{8}$  &   $\frac{3}{32}$  &   $\frac{-9}{32}$  &   $\frac{-3}{8}$  &   $\frac{-1}{16}$\\ \hline

${\mathcal N}_{l,t} (1^{2}       ; 2^{1}6^{1}; n-1)     $  & $\frac{-5}{4}$  &   $\frac{-3}{4}$  &   2  &   $\frac{15}{4}$  &   $\frac{-9}{4}$  &   -6  &   $\frac{-3}{2}$\\ \hline
${\mathcal N}_{l,t} (1^{1}2^{1}  ; 2^{1}6^{1}; n-1)     $  & $\frac{1}{4}$  &   $\frac{-9}{4}$  &   2  &   $\frac{3}{4}$  &   $\frac{27}{4}$  &   6  &   $  0 $\\ \hline
${\mathcal N}_{l,t} (1^{1}4^{1}  ; 2^{1}6^{1}; n-1)     $  & $\frac{-1}{8}$  &   $\frac{9}{8}$  &   -1  &   $\frac{3}{8}$  &   $\frac{27}{8}$  &   3  &   $\frac{3}{4}$\\ \hline
${\mathcal N}_{l,t} (2^{2}       ; 2^{1}6^{1}; n-1)     $  & $\frac{-1}{2}$  &   $\frac{-3}{2}$  &   2  &   $\frac{3}{2}$  &   $\frac{-9}{2}$  &   -6  &   $  0 $\\
\hline

\hline

\end{tabularx}

\end{center}
\begin{center}

\renewcommand{\arraystretch}{1.2}
 %\tiny
%\begin{sidewaystable}[htb]
\begin{tabularx}{\textwidth}{|p{4cm}| p{0.8cm}|p{0.8cm}|p{0.8cm}|p{0.8cm}|p{0.8cm}|p{0.8cm}|}
\caption{Sample formulas for number of representations of a positive integer by the forms in $M_{3}(12,\chi_{-4})$}\\ %for number of representations of Integer by Forms.. $M_{3}(12,x_{-4})$.}\\
\hline
     
                            \multicolumn{1}{|c|}{\textbf{Representation number}}  & \multicolumn{6}{|c|}{\textbf{Coefficient of linear span }}\\ \cline{2-7} 
    {} & {$t_1$} & {$t_2$} & {$t_3$} & {$t_4$} & {$t_5$} & {$t_6$}   \\ \hline 
     
\endfirsthead
     \caption{Coefficient of linear span (continued\ldots)}\\
      \hline
                              \multicolumn{1}{|c|}{\textbf{Representation number}}  & \multicolumn{6}{|c|}{\textbf{Coefficient of linear span }}\\ \cline{2-7} 
    {} & {$t_1$} & {$t_2$} & {$t_3$} & {$t_4$} & {$t_5$} & {$t_6$}    \\ \hline 
     
\endhead
      \hline
      \multicolumn{7}{|r|}{{Continued on Next Page\ldots}}     \\ \hline
      
\endfoot
    
\endlastfoot
\hline

$\delta_{6} (1^{4}6^{2}          ; n-2)                  $ & $ 0 $  & $ 0 $  & $ \frac{2}{7} $  & $ \frac{9}{7} $  & $ \frac{-2}{7} $  & $ \frac{-9}{7} $  \\ \hline
$\delta_{6} (1^{1}2^{3}3^{1}6^{1}; n-2)                  $ & $ 0 $  & $ 0 $  & $ \frac{5}{28} $  & $ \frac{-9}{28} $  & $ \frac{-5}{28} $  & $ \frac{-3}{7} $  \\ \hline
$\delta_{6} (1^{1}2^{1}3^{1}6^{3}; n-3)                  $ & $ 0 $  & $ 0 $  & $ \frac{1}{28} $  & $ \frac{15}{28} $  & $ \frac{-1}{28} $  & $ \frac{-2}{7} $  \\ \hline
$\delta_{6} (2^{2}3^{4}          ; n-2)                  $ & $ 0 $  & $ 0 $  & $ \frac{1}{7} $  & $ \frac{-6}{7} $  & $ \frac{-1}{7} $  & $ \frac{-1}{7} $  \\ \hline
$\delta_{6} (3^{4}6^{2}          ; n-3)                  $ & $ 0 $  & $ 0 $  & $ 0 $  & 1  & $ 0 $  & $ 0 $  \\ \hline

${\mathcal N}_{l,t} (1^{1} ; 1^{1}2^{2}3^{1}; n-1)       $  & $ 0 $  & $ 0 $  & $ \frac{10}{7} $  & $ \frac{-18}{7} $  & $ \frac{-3}{7} $  & $ \frac{-3}{7} $  \\ \hline
${\mathcal N}_{l,t} (1^{1} ; 1^{1}3^{1}6^{2}; n-2)       $  & $ 0 $  & $ 0 $  & $ \frac{2}{7} $  & $ \frac{30}{7} $  & $ \frac{-2}{7} $  & $ \frac{-9}{7} $  \\ \hline
${\mathcal N}_{l,t} (2^{1} ; 1^{1}2^{2}3^{1}; n-1)       $  & $ 0 $  & $ 0 $  & $ \frac{4}{7} $  & $ \frac{18}{7} $  & $ \frac{3}{7} $  & $ \frac{3}{7} $  \\
\hline

\hline

\end{tabularx}

\end{center}

{\scriptsize
\begin{center}

\renewcommand{\arraystretch}{1.2}
 %\tiny
%\begin{sidewaystable}[htb]
\begin{tabularx}{\textwidth}{|p{4.2cm}| p{0.5cm}|p{0.5cm}|p{0.5cm}|p{0.5cm}|p{0.5cm}|p{0.5cm}|p{0.5cm}|p{0.5cm}|p{0.5cm}|p{0.5cm}|p{0.5cm}|p{0.5cm}|}
\caption{Sample formulas for number of representations of a positive integer by the forms in $M_{3}(24,\chi_{-3})$}\\ %for number of representations of Integer by Forms.. $M_{3}(24,x_{-3})$}\\
\hline
     
                            \multicolumn{1}{|c|}{\textbf{Representation number}}  & \multicolumn{12}{|c|}{\textbf{Coefficient of linear span }}\\ \cline{2-13} 
    {} & {$t_1$} & {$t_2$} & {$t_3$} & {$t_4$} & {$t_5$} & {$t_6$} & {$t_7$} & {$t_8$} & {$t_9$} & {$t_{10}$} & {$t_{11}$} & {$t_{12}$}\\ \hline 
     
\endfirsthead
     \caption{Coefficient of linear span. (continued\ldots)}\\
      \hline
                              \multicolumn{1}{|c|}{\textbf{Representation number}}  & \multicolumn{12}{|c|}{\textbf{Coefficient of linear span }}\\ \cline{2-13} 
    {} & {$t_1$} & {$t_2$} & {$t_3$} & {$t_4$} & {$t_5$} & {$t_6$} & {$t_7$} & {$t_8$} & {$t_9$} & {$t_{10}$} & {$t_{11}$} & {$t_{12}$} \\ \hline 
     
\endhead
      \hline
      \multicolumn{13}{|r|}{{Continued on Next Page\ldots}}     \\ \hline
      
\endfoot
    
\endlastfoot
\hline

$\delta_6 (1^{1}3^{1}4^{2}6^{2}; n-3) $ & $ \frac{1}{64} $  & $ \frac{-1}{64} $  & $ 0 $  & $ 0 $  & $ \frac{-3}{64} $  & $ \frac{-3}{64} $  & $ \frac{-3}{4} $  & 6  & $ \frac{-7}{32} $  & $ 0 $  & $ \frac{1}{12} $  & $ \frac{1}{6} $\\ \hline
$\delta_6 (2^{1}4^{4}6^{1}      ; n-3) $ & $ \frac{1}{128} $  & $ \frac{3}{128} $  & $ \frac{-1}{32} $  & $ 0 $  & $ \frac{9}{128} $  & $ \frac{-27}{128} $  & $ \frac{-27}{32} $  & $ \frac{9}{2} $  & $ \frac{-9}{64} $  & $ \frac{3}{16} $  & $ \frac{1}{16} $  & $ 0 $\\ \hline
$\delta_6 (1^{1}2^{2}3^{1}4^{2}; n-2) $ & $ \frac{-1}{64} $  & $ \frac{1}{64} $  & $ 0 $  & $ 0 $  & $ \frac{27}{64} $  & $ \frac{27}{64} $  & $ 0 $  & $ 0 $  & $ \frac{3}{32} $  & $ \frac{3}{4} $  & $ 0 $  & $ \frac{-1}{2} $\\ \hline

${\mathcal N}_{l,t}(1^{1}8^{1}             ; 2^{1}6^{1}; n-1) $ & $ \frac{1}{16} $  & $ \frac{-9}{16} $  & $ \frac{1}{2} $  & $ 0 $  & $ \frac{-15}{16} $  & $ \frac{9}{16} $  & 6  & -36  & $ \frac{3}{8} $  & $ \frac{3}{2} $  & $ \frac{-1}{2} $  & 2\\ \hline
${\mathcal N}_{l,t}(2^{1}8^{1}             ; 2^{1}6^{1}; n-1) $ & $ \frac{-1}{8} $  & $ \frac{-3}{8} $  & $ \frac{1}{2} $  & $ 0 $  & $ \frac{-15}{8} $  & $ \frac{-27}{8} $  & -6  & 36  & $ \frac{-3}{2} $  & $ \frac{-3}{2} $  & $ \frac{1}{2} $  & 4\\ \hline
${\mathcal N}_{l,t}(4^{1}8^{1}             ; 2^{1}6^{1}; n-1) $ & $ \frac{1}{16} $  & $ \frac{3}{16} $  & $ \frac{-1}{4} $  & $ 0 $  & $ \frac{-15}{16} $  & $ \frac{-27}{16} $  & $ \frac{15}{4} $  & -36  & $ \frac{3}{8} $  & -3  & $ \frac{-1}{2} $  & 2\\ \hline
${\mathcal N}_{l,t}(8^{1}8^{1}             ; 2^{1}6^{1}; n-1) $ & $ \frac{-1}{32} $  & $ \frac{-3}{32} $  & $ \frac{1}{8} $  & $ 0 $  & $ \frac{-33}{32} $  & $ \frac{-45}{32} $  & $ \frac{-21}{8} $  & 18  & $ \frac{-3}{16} $  & $ \frac{-3}{4} $  & $ \frac{1}{4} $  & 2\\
\hline

\hline

\end{tabularx}

\end{center}
}

\begin{center}

\renewcommand{\arraystretch}{1.2}
 %\tiny
%\begin{sidewaystable}[htb]
\begin{tabularx}{\textwidth}{|p{4.2cm}| p{0.5cm}|p{0.5cm}|p{0.5cm}|p{0.5cm}|p{0.5cm}|p{0.5cm}|p{0.5cm}|p{0.5cm}|p{0.5cm}|p{0.5cm}|p{0.5cm}|p{0.5cm}|}
\caption{Sample formulas for number of representations of a positive integer by the forms in $M_{3}(24,\chi_{-4})$}\\ %for number of representations of Integer by Traingular Forms.. $M_{3}(24,x_{-4})$}\\
\hline
     
                            \multicolumn{1}{|c|}{\textbf{Representation number}}  & \multicolumn{12}{|c|}{\textbf{Coefficient of linear span }}\\ \cline{2-13} 
    {} & {$t_1$} & {$t_2$} & {$t_3$} & {$t_4$} & {$t_5$} & {$t_6$} & {$t_7$} & {$t_8$} & {$t_9$} & {$t_{10}$} & {$t_{11}$} & {$t_{12}$}\\ \hline 
     
\endfirsthead
     \caption{Coefficient of linear span. (continued\ldots)}\\
      \hline
                              \multicolumn{1}{|c|}{\textbf{Representation number}}  & \multicolumn{12}{|c|}{\textbf{Coefficient of linear span }}\\ \cline{2-13} 
    {} & {$t_1$} & {$t_2$} & {$t_3$} & {$t_4$} & {$t_5$} & {$t_6$} & {$t_7$} & {$t_8$} & {$t_9$} & {$t_{10}$} & {$t_{11}$} & {$t_{12}$} \\ \hline 
     
\endhead
      \hline
      \multicolumn{13}{|r|}{{Continued on Next Page\ldots}}     \\ \hline
      
\endfoot
    
\endlastfoot
\hline

$ \delta_6(1^{2}3^{2}4^{2};n-2) $ & $  0 $ & $  0 $ & $  0 $ & $  0 $ &  $ \frac{1}{7} $ & $  0 $ &  $ \frac{9}{14} $ & $  0 $ &  $ \frac{-1}{7} $ &  $ \frac{-9}{14} $ &  $ \frac{1}{2} $ & $  0 $ \\  \hline
$ \delta_6(2^{2}4^{2}6^{2};n-3)  $ & $  0 $ & $  0 $ & $  0 $ & $  0 $ &  $ \frac{1}{14} $ &  $ \frac{-2}{7} $ &  $ \frac{9}{28} $ &  $ \frac{-9}{7} $ &  $ \frac{-1}{14} $ &  $ \frac{-9}{28} $ &  $ \frac{1}{28} $ &  $ \frac{-3}{14} $ \\  \hline
$ \delta_6(4^{2}6^{4}     ;n-4)        $ & $  0 $ & $  0 $ & $  0 $ & $  0 $ & $  0 $ &  $ \frac{1}{7} $ & $  0 $ &  $ \frac{-6}{7} $ & $  0 $ & $  0 $ &  $ \frac{-1}{7} $ &  $ \frac{-1}{7} $ \\  \hline

${\mathcal N}_{l,t} (1^{2}       ; 4^{2}          ;n-1 ) $ &  $ \frac{2}{7} $ &  $ \frac{-2}{7} $ &  $ \frac{-9}{7} $ &  $ \frac{9}{7} $ &  $ \frac{20}{7} $ &  $ \frac{-32}{7} $ &  $ \frac{90}{7} $ &  $ \frac{-144}{7} $ &  $ \frac{-15}{7} $ &  $ \frac{-72}{7} $ &  $ \frac{48}{7} $ & $  0 $ \\  \hline
${\mathcal N}_{l,t} (1^{1}2^{1}  ; 4^{2}          ;n-1 ) $ &  $ \frac{-5}{14} $ &  $ \frac{5}{14} $ &  $ \frac{-9}{14} $ &  $ \frac{9}{14} $ &  $ \frac{5}{7} $ &  $ \frac{40}{7} $ &  $ \frac{-9}{7} $ &  $ \frac{-72}{7} $ &  $ \frac{9}{14} $ &  $ \frac{18}{7} $ &  $ \frac{-18}{7} $ & $  0 $ \\  \hline
${\mathcal N}_{l,t} (1^{1}4^{1}  ; 4^{2}          ;n-1 ) $ &  $ \frac{2}{7} $ &  $ \frac{-2}{7} $ &  $ \frac{-9}{7} $ &  $ \frac{9}{7} $ &  $ \frac{2}{7} $ &  $ \frac{16}{7} $ &  $ \frac{9}{7} $ &  $ \frac{72}{7} $ &  $ \frac{3}{7} $ &  $ \frac{9}{7} $ &  3 &  $ \frac{36}{7} $ \\  \hline
${\mathcal N}_{l,t} (2^{2}       ; 4^{2}          ;n-1 ) $ &  $ \frac{2}{7} $ &  $ \frac{-2}{7} $ &  $ \frac{-9}{7} $ &  $ \frac{9}{7} $ &  $ \frac{8}{7} $ &  $ \frac{-32}{7} $ &  $ \frac{36}{7} $ &  $ \frac{-144}{7} $ &  $ \frac{-3}{7} $ &  $ \frac{-18}{7} $ &  $ \frac{6}{7} $ & $  0 $ \\
\hline

\hline

\end{tabularx}

\end{center}
\begin{center}

\renewcommand{\arraystretch}{1.2}
 %\tiny
%\begin{sidewaystable}[htb]
\begin{tabularx}{\textwidth}{|p{4.2cm}| p{0.5cm}|p{0.5cm}|p{0.8cm}|p{0.8cm}|p{0.8cm}|p{0.8cm}|p{0.8cm}|p{0.8cm}|p{0.8cm}|p{0.8cm}|}
\caption{Sample formulas for number of representations of a positive integer by the forms in $M_{3}(24,\chi_{-8})$}\\ %for number of representations of Integer by Traingular Forms.. $M_{3}(24,x_{-8})$}\\
\hline
     
                            \multicolumn{1}{|c|}{\textbf{Representation number}}  & \multicolumn{10}{|c|}{\textbf{Coefficient of linear span }}\\ \cline{2-11} 
    {} & {$t_1$} & {$t_2$} & {$t_3$} & {$t_4$} & {$t_5$} & {$t_6$} & {$t_7$} & {$t_8$} & {$t_9$} & {$t_{10}$} \\ \hline 
     
\endfirsthead
     \caption{Coefficient of linear span (continued\ldots)}\\
      \hline
                              \multicolumn{1}{|c|}{\textbf{Representation number}}  & \multicolumn{10}{|c|}{\textbf{Coefficient of linear span }}\\ \cline{2-11} 
    {} & {$t_1$} & {$t_2$} & {$t_3$} & {$t_4$} & {$t_5$} & {$t_6$} & {$t_7$} & {$t_8$} & {$t_9$} & {$t_{10}$}  \\ \hline 
     
\endhead
      \hline
      \multicolumn{11}{|r|}{{Continued on Next Page\ldots}}     \\ \hline
      
\endfoot
    
\endlastfoot
\hline

$ \delta_6 (1^{3} 3^{1}4^{1}6^{1} ;n-2) $  &  $  0 $    &  $  0 $    &   $ \frac{8}{39} $   &   $ \frac{6}{13} $   &   $ \frac{-2}{39} $   &   $ \frac{-3}{26} $   &   $ \frac{1}{52} $   &   $ \frac{1}{13} $   &   $ \frac{-3}{52} $   &   $ \frac{-5}{26} $  \\  \hline
$ \delta_6 (2^{3}  3^{2}4^{1}     ;n-2) $  &  $  0 $    &  $  0 $    &   $ \frac{5}{39} $   &   $ \frac{-6}{13} $   &   $ \frac{-5}{156} $   &   $ \frac{-111}{52} $   &   $ \frac{3}{13} $   &   $ \frac{5}{104} $   &   $ \frac{-33}{104} $   &   $ \frac{-3}{52} $  \\  \hline
$ \delta_6 (2^{1}  3^{2}4^{1}6^{2};n-3) $  &  $  0 $    &  $  0 $    &   $ \frac{1}{39} $   &   $ \frac{4}{13} $   &   $ \frac{-1}{156} $   &   $ \frac{35}{52} $   &   $ \frac{-11}{156} $   &   $ \frac{1}{104} $   &   $ \frac{9}{104} $   &   $ \frac{-7}{156} $  \\  \hline

$ {\mathcal N}_{l,t} (1^{1} ; 3^{2}4^{1}6^{1};n-2) $   &  $  0 $    &  $  0 $    &   $ \frac{8}{39} $   &   $ \frac{32}{13} $   &   $ \frac{-2}{39} $   &   $ \frac{31}{13} $   &   $ \frac{-3}{13} $   &   $ \frac{1}{13} $   &   $ \frac{9}{13} $   &   $ \frac{-9}{13} $  \\  \hline
$ {\mathcal N}_{l,t} (2^{1} ; 3^{2}4^{1}6^{1};n-2) $   &  $  0 $    &  $  0 $    &   $ \frac{8}{39} $   &   $ \frac{-20}{13} $   &   $ \frac{-2}{39} $   &   $ \frac{-34}{13} $   &   $ \frac{7}{26} $   &   $ \frac{1}{13} $   &   $ \frac{-21}{26} $   &   $ \frac{4}{13} $  \\  \hline
$ {\mathcal N}_{l,t} (4^{1} ; 3^{2}4^{1}6^{1};n-2) $   &  $  0 $    &  $  0 $    &   $ \frac{2}{39} $   &   $ \frac{8}{13} $   &   $ \frac{-1}{78} $   &   $ \frac{-43}{26} $   &   $ \frac{5}{26} $   &   $ \frac{1}{52} $   &   $ \frac{9}{52} $   &   $ \frac{-11}{26} $  \\  \hline
$ {\mathcal N}_{l,t} (8^{1} ; 3^{2}4^{1}6^{1};n-2) $   &  $  0 $    &  $  0 $    &   $ \frac{2}{39} $   &   $ \frac{-5}{13} $   &   $ \frac{-1}{78} $   &   $ \frac{-17}{26} $   &   $ \frac{5}{26} $   &   $ \frac{1}{52} $   &   $ \frac{9}{52} $   &   $ \frac{-11}{26} $  \\
\hline

\hline

\end{tabularx}

\end{center}
\begin{center}

\renewcommand{\arraystretch}{1.2}
 %\tiny
%\begin{sidewaystable}[htb]
\begin{tabularx}{\textwidth}{|p{4.2cm}| p{0.5cm}|p{0.6cm}|p{0.6cm}|p{0.5cm}|p{0.8cm}|p{0.8cm}|p{0.8cm}|p{0.8cm}|p{0.8cm}|p{0.8cm}|}
\caption{Sample formulas for number of representations of a positive integer by the forms in $M_{3}(24,\chi_{-24})$}\\ %for number of representations of Integer by Traingular Forms..$M_{3}(24,x_{-24})$}\\
\hline
     
                            \multicolumn{1}{|c|}{\textbf{Representation number}}  & \multicolumn{10}{|c|}{\textbf{Coefficient of linear span }}\\ \cline{2-11} 
    {} & {$t_1$} & {$t_2$} & {$t_3$} & {$t_4$} & {$t_5$} & {$t_6$} & {$t_7$} & {$t_8$} & {$t_9$} & {$t_{10}$} \\ \hline 
     
\endfirsthead
     \caption{Coefficient of linear span (continued\ldots)}\\
      \hline
                              \multicolumn{1}{|c|}{\textbf{Representation number}}  & \multicolumn{10}{|c|}{\textbf{Coefficient of linear span }}\\ \cline{2-11} 
    {} & {$t_1$} & {$t_2$} & {$t_3$} & {$t_4$} & {$t_5$} & {$t_6$} & {$t_7$} & {$t_8$} & {$t_9$} & {$t_{10}$}  \\ \hline 
     
\endhead
      \hline
      \multicolumn{11}{|r|}{{Continued on Next Page\ldots}}     \\ \hline
      
\endfoot
    
\endlastfoot
\hline

$ \delta_6 (1^{2}4^{1}6^{3}          ;n-3) $  &  $  0 $    &   $ \frac{3}{46} $   &   $ \frac{1}{46} $   &  $  0 $    &   $ \frac{-17}{92} $   &   $ \frac{-19}{184} $   &   $ \frac{31}{46} $   &   $ \frac{21}{92} $   &   $ \frac{-39}{92} $   &   $ \frac{-39}{184} $  \\  \hline
$ \delta_6 (1^{1}2^{1}3^{3}4^{1}     ;n-2) $  &  $  0 $    &   $ \frac{3}{23} $   &   $ \frac{-1}{23} $   &  $  0 $    &   $ \frac{1}{46} $   &   $ \frac{9}{184} $   &   $ \frac{-1}{46} $   &   $ \frac{-1}{46} $   &   $ \frac{-3}{46} $   &   $ \frac{-21}{184} $  \\  \hline
$ \delta_6 (1^{2}2^{2}4^{1}6^{1}     ;n-2) $  &  $  0 $    &   $ \frac{9}{46} $   &   $ \frac{-1}{46} $   &  $  0 $    &   $ \frac{-15}{92} $   &   $ \frac{-5}{92} $   &   $ \frac{15}{23} $   &   $ \frac{19}{92} $   &   $ \frac{-45}{92} $   &   $ \frac{-15}{46} $  \\  \hline
$ \delta_6 (3^{2}4^{3}6^{1}          ;n-3) $  &  $  0 $    &   $ \frac{3}{92} $   &   $ \frac{1}{92} $   &  $  0 $    &   $ \frac{29}{184} $   &   $ \frac{1}{92} $   &   $ \frac{-15}{92} $   &   $ \frac{-1}{92} $   &   $ \frac{-39}{184} $   &   $ \frac{-1}{23} $  \\  \hline

$ {\mathcal N}_{l,t} (1^{1} ; 1^{2}2^{1}4^{1};n-1) $    &  $  0 $    &   $ \frac{36}{23} $   &   $ \frac{-4}{23} $   &  $  0 $    &   $ \frac{-51}{92} $   &   $ \frac{3}{46} $   &   $ \frac{51}{23} $   &   $ \frac{15}{23} $   &   $ \frac{-153}{92} $   &   $ \frac{-51}{46} $  \\  \hline
$ {\mathcal N}_{l,t} (2^{1} ; 1^{2}2^{1}4^{1};n-1) $    &  $  0 $    &   $ \frac{18}{23} $   &   $ \frac{2}{23} $   &  $  0 $    &   $ \frac{-63}{92} $   &   $ \frac{-21}{92} $   &   $ \frac{54}{23} $   &   $ \frac{18}{23} $   &   $ \frac{-117}{92} $   &   $ \frac{-39}{92} $  \\  \hline
$ {\mathcal N}_{l,t} (4^{1} ; 1^{2}2^{1}4^{1};n-1) $    &  $  0 $    &   $ \frac{9}{23} $   &   $ \frac{-1}{23} $   &  $  0 $    &   $ \frac{39}{92} $   &   $ \frac{9}{23} $   &   $ \frac{-39}{23} $   &   $ \frac{-27}{46} $   &   $ \frac{117}{92} $   &   $ \frac{39}{46} $  \\
\hline

\end{tabularx}

\end{center}

\section{appendix-II}

We give the representation formula of a positive integer by the sum of higher figurate number (when $a=4$) with the coefficients $1,2,3$ and $6$ with $6$ variables (in Table $15$-$16$), $8$ variables (in Table $17$) and $10$ variables (in Table $18$-$19$). We use the $n^{\rm th}$ Fourier coefficients of basis elements of weight $3$ level $24$ with some Dirichlet Character  constructed in table B  to give these representation formulas.
{\scriptsize
\renewcommand{\arraystretch}{1.2}
 %\tiny
%\begin{sidewaystable}[htb]
\begin{tabularx}{\textwidth}{|p{3.1cm}| p{0.5cm}|p{0.5cm}|p{0.5cm}|p{0.5cm}|p{0.5cm}|p{0.5cm}|p{0.5cm}|p{0.5cm}|p{0.5cm}|p{0.5cm}|p{0.5cm}|p{0.5cm}|}
\caption{Sample formulas for number of representations of a positive integer by the sum of  figurate number with coeffients (when $a=3$) in $M_{3}(24,\chi_{-3})$ }\\ %for number of representations of Integer by Forms.. $M_{2}(24,x_{12})$}\\
\hline
     
                            \multicolumn{1}{|c|}{\textbf{Figurate Form}}  & \multicolumn{12 }{|c|}{\textbf{Coefficient of linear span }}\\ \cline{2-13} 
    {} & {$c_1$} & {$c_2$} & {$c_3$} & {$c_4$} & {$c_5$} & {$c_6$} & {$c_7$} & {$c_8$}& {$c_9$} & {$c_{10}$} & {$c_{11}$} & {$c_{12}$} \\ \hline 
     
\endfirsthead
     \caption{Coefficient of linear span. (continued\ldots)}\\
      \hline
                              \multicolumn{1}{|c|}{\textbf{Figurate Form}}  & \multicolumn{12}{|c|}{\textbf{Coefficient of linear span }}\\ \cline{2-13} 
    {} & {$c_1$} & {$c_2$} & {$c_3$} & {$c_4$} & {$c_5$} & {$c_6$} & {$c_7$} & {$c_8$}& {$c_9$} & {$c_{10}$} & {$c_{11}$} & {$c_{12}$}  \\ \hline 
     
\endhead
      \hline
      \multicolumn{13}{|r|}{{Continued on Next Page\ldots}}     \\ \hline
      
\endfoot
    
\endlastfoot
\hline

%higher fig number (a=4) wt 3 level 24 char X-3

$N_{4}(1^{5}, 2^{0}, 3^{1}, 6^{0}; n-1)$  & $ \frac{-1}{8} $ & $ \frac{1}{8} $ & $ 0 $ & $ 0 $ & $ \frac{9}{8} $ & $ \frac{9}{8} $ & $ 0 $ & $ 0 $ & $ 0 $ & $ 0 $ & $ 0 $ & $ 0$ \\ \hline
$N_{4}(1^{2}, 2^{1}, 3^{2}, 6^{1}; n-2)$  & $ \frac{-1}{16} $ & $ \frac{1}{16} $ & $ 0 $ & $ 0 $ & $ \frac{3}{16} $ & $ \frac{3}{16} $ & $ 0 $ & $ 0 $ & $ \frac{-1}{8} $ & $ 0 $ & $ 0 $ & $ 0$ \\ \hline
$N_{4}(1^{1}, 2^{0}, 3^{5}, 6^{0}; n-2)$  & $ \frac{-1}{8} $ & $ \frac{1}{8} $ & $ 0 $ & $ 0 $ & $ \frac{1}{8} $ & $ \frac{1}{8} $ & $ 0 $ & $ 0 $ & $ 0 $ & $ 0 $ & $ 0 $ & $ 0$ \\ \hline
$N_{4}(1^{0}, 2^{5}, 3^{0}, 6^{1}; n-2)$  & $ 0 $ & $ \frac{-1}{8} $ & $ \frac{1}{8} $ & $ 0 $ & $ 0 $ & $ \frac{9}{8} $ & $ \frac{9}{8} $ & $ 0 $ & $ 0 $ & $ 0 $ & $ 0 $ & $ 0$ \\ \hline
$N_{4}(1^{0}, 2^{3}, 3^{0}, 6^{3}; n-3)$  & $ \frac{-1}{32} $ & $ \frac{-3}{32} $ & $ \frac{1}{8} $ & $ 0 $ & $ \frac{3}{32} $ & $ \frac{-9}{32} $ & $ \frac{-3}{8} $ & $ 0 $ & $ \frac{-1}{16} $ & $ 0 $ & $ 0 $ & $ 0$ \\ \hline
$N_{4}(1^{0}, 2^{1}, 3^{0}, 6^{5}; n-4)$  & $ 0 $ & $ \frac{-1}{8} $ & $ \frac{1}{8} $ & $ 0 $ & $ 0 $ & $ \frac{1}{8} $ & $ \frac{1}{8} $ & $ 0 $ & $ 0 $ & $ 0 $ & $ 0 $ & $ 0$ \\
\hline
\end{tabularx}

\renewcommand{\arraystretch}{1.2}
%\tiny
%\begin{sidewaystable}[htb]
\begin{tabularx}{\textwidth}{|p{3.1cm}| p{0.5cm}|p{0.5cm}|p{0.5cm}|p{0.5cm}|p{0.5cm}|p{0.5cm}|p{0.5cm}|p{0.5cm}|p{0.5cm}|p{0.5cm}|p{0.5cm}|p{0.5cm}|}
	\caption{Sample formulas for number of representations of a positive integer by the sum of  figurate number with coeffients (when $a=3$) in $M_{3}(24,\chi_{-4})$}\\ %for number of representations of Integer by Forms.. $M_{2}(24,x_{12})$}\\
\hline

\multicolumn{1}{|c|}{\textbf{Figurate Form}}  & \multicolumn{12 }{|c|}{\textbf{Coefficient of linear span }}\\ \cline{2-13} 
{} & {$c_1$} & {$c_2$} & {$c_3$} & {$c_4$} & {$c_5$} & {$c_6$} & {$c_7$} & {$c_8$}& {$c_9$} & {$c_{10}$} & {$c_{11}$} & {$c_{12}$} \\ \hline 

\endfirsthead
\caption{Coefficient of linear span. (continued\ldots)}\\
\hline
\multicolumn{1}{|c|}{\textbf{Figurate Form}}  & \multicolumn{12}{|c|}{\textbf{Coefficient of linear span }}\\ \cline{2-13} 
{} & {$c_1$} & {$c_2$} & {$c_3$} & {$c_4$} & {$c_5$} & {$c_6$} & {$c_7$} & {$c_8$}& {$c_9$} & {$c_{10}$} & {$c_{11}$} & {$c_{12}$}  \\ \hline 

\endhead
\hline
\multicolumn{13}{|r|}{{Continued on Next Page\ldots}}     \\ \hline

\endfoot

\endlastfoot
\hline

%higher fig number (a=4) wt 3 level 24 char X-4
% we use the basis constructed in \cite{ramki-lalit} to give these representaation formula 

$N_{4}(1^{4}, 2^{2}, 3^{0}, 6^{0}; n-1)$  & $ 0 $ & $ 0 $ & $ 0 $ & $ 0 $ & $ 1 $ & $ 0 $ & $ 0 $ & $ 0 $ & $ 0 $ & $ 0 $ & $ 0 $ & $ 0$ \\ \hline
$N_{4}(1^{4}, 2^{0}, 3^{0}, 6^{2}; n-2)$  & $ 0 $ & $ 0 $ & $ 0 $ & $ 0 $ & $ \frac{2}{7} $ & $ 0 $ & $ \frac{9}{7} $ & $ 0 $ & $ \frac{-2}{7} $ & $ \frac{-9}{7} $ & $ 0 $ & $ 0$ \\ \hline
$N_{4}(1^{1}, 2^{3}, 3^{1}, 6^{1}; n-2)$  & $ 0 $ & $ 0 $ & $ 0 $ & $ 0 $ & $ \frac{5}{28} $ & $ 0 $ & $ \frac{-9}{28} $ & $ 0 $ & $ \frac{-5}{28} $ & $ \frac{-3}{7} $ & $ 0 $ & $ 0$ \\ \hline
$N_{4}(1^{1}, 2^{1}, 3^{1}, 6^{3}; n-3)$  & $ 0 $ & $ 0 $ & $ 0 $ & $ 0 $ & $ \frac{1}{28} $ & $ 0 $ & $ \frac{15}{28} $ & $ 0 $ & $ \frac{-1}{28} $ & $ \frac{-2}{7} $ & $ 0 $ & $ 0$ \\ \hline
$N_{4}(1^{0}, 2^{2}, 3^{4}, 6^{0}; n-2)$  & $ 0 $ & $ 0 $ & $ 0 $ & $ 0 $ & $ \frac{1}{7} $ & $ 0 $ & $ \frac{-6}{7} $ & $ 0 $ & $ \frac{-1}{7} $ & $ \frac{-1}{7} $ & $ 0 $ & $ 0$ \\ \hline
$N_{4}(1^{0}, 2^{0}, 3^{4}, 6^{2}; n-3)$  & $ 0 $ & $ 0 $ & $ 0 $ & $ 0 $ & $ 0 $ & $ 0 $ & $ 1 $ & $ 0 $ & $ 0 $ & $ 0 $ & $ 0 $ & $ 0$ \\
\hline

\hline
\end{tabularx}
}

%We use the $n^{\rm th}$ Fourier coefficients of basis elements of weight $4$ level $24$ with some Dirichlet Character  constructed in \cite{rss-ijnt} to give these representations formula.

%ModularSpace%(wt4N=24X1)% 
{\scriptsize
\renewcommand{\arraystretch}{1.2}
\begin{longtable}{|p{2.9 cm}|p{1 cm}|p{1cm}|p{1cm}|p{1cm}|p{1cm}|p{1cm}|p{1cm}|p{1cm}|}
\caption{Sample formulas for number of representations of a positive integer by the sum of  figurate number with coeffients (when a =4) in $M_{4}(24,\chi_{0})$}\\
\hline

\multirow{2}{*}{\textbf{Figurate Form}}  & \multicolumn{8}{|c|}{\textbf{Coefficient of linear span $M_k$}}\\ \cline{2-9} %
 {} & {$c_1$} & {$c_2$} & {$c_3$} & {$c_4$} & {$c_5$} & {$c_6$} & {$c_7$} & {$c_8$} \\ \cline{2-9} & {$c_9$} & {$c_{10}$} & {$c_{11}$}  & {$c_{12}$} & {$c_{13}$} & {$c_{14}$} & {$c_{15}$} & {$c_{16}$} \\ \hline

  \endfirsthead

\multicolumn{9}{c}%
{{\bfseries  \thetable{} -- continued from previous page}} \\
\hline
\multirow{2}{*}{\textbf{Figurate Form}}  & \multicolumn{8}{|c|}{\textbf{Coefficient of linear span $M_k$}}\\ \cline{2-9} %
   {} & {$c_1$} & {$c_2$} & {$c_3$} & {$c_4$} & {$c_5$} & {$c_6$} & {$c_7$} & {$c_8$} \\ \cline{2-9} & {$c_9$} & {$c_{10}$} & {$c_{11}$}  & {$c_{12}$} & {$c_{13}$} & {$c_{14}$} & {$c_{15}$} & {$c_{16}$} \\ \hline 
 
\endhead

\hline \multicolumn{9}{|r|}{{Continued on next page}} \\ \hline
\endfoot

\hline \hline
\endlastfoot
\hline

%higher fig number *a=4* wt 4 level 24 char X1

$N_{4}(1^{8}, 2^{0}, 3^{0}, 6^{0}; n-1)$  &  $ \frac{1}{240} $ & $ \frac{-1}{240} $ & $ 0 $ & $ 0 $ & $ 0 $ & $ 0 $ & $ 0 $ & $ 0 $ \\  \cline{2-9} & $ 0 $ & $ 0 $ & $ 0 $ & $ 0 $ & $ 0 $ & $ 0 $ & $ 0 $ & $ 0$ \\ \hline
$N_{4}(1^{5}, 2^{1}, 3^{1}, 6^{1}; n-2)$  &  $ \frac{13}{19200} $ & $ \frac{-13}{19200} $ & $ \frac{9}{6400} $ & $ 0 $ & $ \frac{-9}{6400} $ & $ 0 $ & $ 0 $ & $ 0 $ \\  \cline{2-9} & $ \frac{-1}{10} $ & $ \frac{-1}{2} $ & $ 0 $ & $ 0 $ & $ 0 $ & $ \frac{-1}{16} $ & $ 0 $ & $ 0$ \\ \hline
$N_{4}(1^{4}, 2^{0}, 3^{4}, 6^{0}; n-2)$  &  $ \frac{1}{2400} $ & $ \frac{-1}{2400} $ & $ \frac{3}{800} $ & $ 0 $ & $ \frac{-3}{800} $ & $ 0 $ & $ 0 $ & $ 0 $ \\  \cline{2-9} & $ \frac{-1}{10} $ & $ 0 $ & $ 0 $ & $ 0 $ & $ 0 $ & $ 0 $ & $ 0 $ & $ 0$ \\ \hline
$N_{4}(1^{2}, 2^{4}, 3^{2}, 6^{0}; n-2)$  &  $ \frac{7}{19200} $ & $ \frac{-7}{19200} $ & $ \frac{-9}{6400} $ & $ 0 $ & $ \frac{9}{6400} $ & $ 0 $ & $ 0 $ & $ 0 $ \\  \cline{2-9} & $ \frac{-1}{40} $ & $ \frac{1}{4} $ & $ 0 $ & $ 0 $ & $ 0 $ & $ \frac{-1}{16} $ & $ 0 $ & $ 0$ \\
\hline

\hline
\end{longtable}
}

%ModularSpace%(wt=5N=24X-3)% 
{\scriptsize
\renewcommand{\arraystretch}{1.5}
\begin{longtable}{| p{2.9 cm} | p{0.8 cm}| p{0.8cm}| p{0.8cm}|p{0.8cm}| p{0.8cm}| p{0.8cm}|p{0.8 cm}| p{0.8cm}| p{0.8cm}|p{0.8cm}|}
\caption{Sample formulas for number of representations of a positive integer by the sum of  figurate number with coefficients  (when a =4)in $M_{5}(24,X_{-3})$  .}\\
\hline

\multirow{2}{*}{\textbf{Figurate Form}}  & \multicolumn{10}{|c|}{\textbf{Coefficient of linear span $M_k$}}\\ \cline{2-11} %
  {} & {$c_1$} & {$c_2$} & {$c_3$} & {$c_4$} & {$c_5$} & {$c_6$} & {$c_7$} & {$c_8$} & {$c_9$} & {$c_{10}$} \\ \cline{2-11} & {$c_{11}$}  & {$c_{12}$} & {$c_{13}$} & {$c_{14}$} & {$c_{15}$} & {$c_{16}$} & {$c_{17}$} & {$c_{18}$} & {$c_{19}$} & {$c_{20}$} \\ \hline 
   \endfirsthead

\multicolumn{11}{c}%
{{\bfseries  \thetable{} -- continued from previous page}} \\
\hline
\multirow{2}{*}{\textbf{Figurate Form}}  & \multicolumn{10}{|c|}{\textbf{Coefficient of linear span $M_k$}}\\ \cline{2-11} %
    {} & {$c_1$} & {$c_2$} & {$c_3$} & {$c_4$} & {$c_5$} & {$c_6$} & {$c_7$} & {$c_8$} & {$c_9$} & {$c_{10}$} \\ \cline{2-11} & {$c_{11}$}  & {$c_{12}$} & {$c_{13}$} & {$c_{14}$} & {$c_{15}$} & {$c_{16}$} & {$c_{17}$} & {$c_{18}$} & {$c_{19}$} & {$c_{20}$} \\ \hline 
  
\endhead

\hline \multicolumn{11}{|r|}{{Continued on next page}} \\ \hline
\endfoot

\hline \hline
\endlastfoot
\hline

$N_{4}(1^{8},2^{1},3^{0},6^{1}; n-2)$ & $ \frac{1}{704} $ & $ \frac{-1}{704} $ & $ 0 $ & $ 0 $ & $ \frac{81}{704} $ & $ \frac{81}{704} $ & $ 0 $ & $ 0 $ & $ \frac{-95}{352} $ & $ \frac{-31}{88} $ \\  \cline{2-11} & $\frac{-137}{352} $ & $ \frac{-5}{352} $ & $ \frac{-27}{11} $ & $ \frac{-27}{11} $ & $ 0 $ & $ 0 $ & $ \frac{27}{176} $ & $ \frac{27}{88} $ & $ \frac{27}{22} $ & $ 0$  \\  \hline
$N_{4}(1^{7},2^{0},3^{3},6^{0}; n-2)$ & $ \frac{1}{352} $ & $ \frac{-1}{352} $ & $ 0 $ & $ 0 $ & $ \frac{27}{352} $ & $ \frac{27}{352} $ & $ 0 $ & $ 0 $ & $ \frac{-25}{88} $ & $ \frac{3}{22} $ \\  \cline{2-11} & $\frac{-61}{88} $ & $ \frac{-9}{88} $ & $ \frac{-36}{11} $ & $ \frac{-36}{11} $ & $ 0 $ & $ 0 $ & $ \frac{9}{44} $ & $ \frac{9}{22} $ & $ \frac{18}{11} $ & $ 0$  \\  \hline
$N_{4}(1^{5},2^{4},3^{1},6^{0}; n-2)$ & $ \frac{-1}{1408} $ & $ \frac{1}{1408} $ & $ 0 $ & $ 0 $ & $ \frac{81}{1408} $ & $ \frac{81}{1408} $ & $ 0 $ & $ 0 $ & $ \frac{-5}{88} $ & $ \frac{13}{44} $ \\  \cline{2-11} & $\frac{-19}{44} $ & $ \frac{-1}{8} $ & $ 0 $ & $ 0 $ & $ 0 $ & $ 0 $ & $ 0 $ & $ 0 $ & $ 0 $ & $ 0$  \\  \hline
$N_{4}(1^{5},2^{2},3^{1},6^{2}; n-3)$ & $ \frac{1}{1408} $ & $ \frac{-1}{1408} $ & $ 0 $ & $ 0 $ & $ \frac{27}{1408} $ & $ \frac{27}{1408} $ & $ 0 $ & $ 0 $ & $ \frac{-25}{352} $ & $ \frac{-19}{88} $ \\  \cline{2-11} & $ \frac{5}{352} $ & $ \frac{13}{352} $ & $ \frac{-9}{11} $ & $ \frac{-9}{11} $ & $ 0 $ & $ 0 $ & $ \frac{9}{176} $ & $ \frac{9}{88} $ & $ \frac{9}{22} $ & $ 0$  \\
\hline

\hline
\end{longtable}%
}

\end{document}